\documentclass[reqno]{amsart}

\usepackage[all]{xy}
\usepackage{graphicx}

\usepackage{amssymb}
\usepackage{amsmath}
\usepackage{mathrsfs}
\usepackage{epsfig}
\usepackage{amscd}
\usepackage{graphicx, color}

\def\E{\ifmmode{\mathbb E}\else{$\mathbb E$}\fi} 
\def\N{\ifmmode{\mathbb N}\else{$\mathbb N$}\fi} 
\def\R{\ifmmode{\mathbb R}\else{$\mathbb R$}\fi} 
\def\Q{\ifmmode{\mathbb Q}\else{$\mathbb Q$}\fi} 
\def\C{\ifmmode{\mathbb C}\else{$\mathbb C$}\fi} 
\def\H{\ifmmode{\mathbb H}\else{$\mathbb H$}\fi} 
\def\Z{\ifmmode{\mathbb Z}\else{$\mathbb Z$}\fi} 
\def\P{\ifmmode{\mathbb P}\else{$\mathbb P$}\fi} 
\def\T{\ifmmode{\mathbb T}\else{$\mathbb T$}\fi} 
\def\SS{\ifmmode{\mathbb S}\else{$\mathbb S$}\fi} 
\def\DD{\ifmmode{\mathbb D}\else{$\mathbb D$}\fi} 

\newcommand{\del}{\partial}
\newcommand{\Cont}{{\operatorname{Cont}}}

\newcommand{\ben}{\begin{enumerate}}
\newcommand{\een}{\end{enumerate}}
\newcommand{\be}{\begin{equation}}
\newcommand{\ee}{\end{equation}}
\newcommand{\bea}{\begin{eqnarray}}
\newcommand{\eea}{\end{eqnarray}}
\newcommand{\beastar}{\begin{eqnarray*}}
\newcommand{\eeastar}{\end{eqnarray*}}
\newcommand{\bc}{\begin{center}}
\newcommand{\ec}{\end{center}}

\theoremstyle{theorem}
\newtheorem{thm}{Theorem}[section]
\newtheorem{cor}[thm]{Corollary}
\newtheorem{lem}[thm]{Lemma}
\newtheorem{prop}[thm]{Proposition}

\theoremstyle{definition}
\newtheorem{defn}[thm]{Definition}
\newtheorem{rem}[thm]{Remark}
\newtheorem{ques}[thm]{Question}

\newtheorem{situ}[thm]{Situation}

\newtheorem*{thm*}{thm}

\numberwithin{equation}{section}

\def\R{{\mathbb R}}
\def\osc{{\hbox{\rm osc }}}
\def\Crit{{\hbox{Crit}}}

\def\E{{\mathbb E}}
\def\Z{{\mathbb Z}}
\def\C{{\mathbb C}}
\def\R{{\mathbb R}}
\def\P{{\mathbb P}}

\def\N{{\mathbb N}}

\def\11{{\mathbb I}}

\def\delbar{{\overline \partial}}

\def\C{\mathbb{C}}
\def\Z{\mathbb{Z}}

\def\T{\mathbb{T}}

\def\Q{\mathbb{Q}}

\def\E{\ifmmode{\mathbb E}\else{$\mathbb E$}\fi} 
\def\N{\ifmmode{\mathbb N}\else{$\mathbb N$}\fi} 
\def\R{\ifmmode{\mathbb R}\else{$\mathbb R$}\fi} 
\def\Q{\ifmmode{\mathbb Q}\else{$\mathbb Q$}\fi} 
\def\C{\ifmmode{\mathbb C}\else{$\mathbb C$}\fi} 
\def\H{\ifmmode{\mathbb H}\else{$\mathbb H$}\fi} 
\def\Z{\ifmmode{\mathbb Z}\else{$\mathbb Z$}\fi} 
\def\P{\ifmmode{\mathbb P}\else{$\mathbb P$}\fi} 
\def\SS{\ifmmode{\mathbb S}\else{$\mathbb S$}\fi} 
\def\DD{\ifmmode{\mathbb D}\else{$\mathbb D$}\fi} 

\def\R{{\mathbb R}}
\def\osc{{\hbox{\rm osc}}}
\def\Crit{{\hbox{Crit}}}
\def\E{{\mathbb E}}
\def\Z{{\mathbb Z}}
\def\C{{\mathbb C}}
\def\R{{\mathbb R}}

\def\N{{\mathbb N}}

\def\JJ{{\mathcal J}}

\def\delbar{{\overline \partial}}

\def\CA{{\mathcal A}}

\def\CC{{\mathcal C}}
\def\CD{{\mathcal D}}

\def\CF{{\mathcal F}}

\def\CH{{\mathcal H}}

\def\CJ{{\mathcal J}}

\def\CL{{\mathcal L}}
\def\CM{{\mathcal M}}

\def\CP{{\mathcal P}}

\def\CP{{\mathcal P}}

\def\CT{{\mathcal T}}

\def\CW{{\mathcal W}}

\def\darr#1{\raise1.5ex\hbox{$\longleftrightarrow$}
\mkern-16.5mu #1}

\def\roughly#1{\raise.3ex\hbox{$#1$\kern-.75em
\lower1ex\hbox{$\sim$}}}

\def\opname#1{\mathop{\kern0pt{\rm #1}}\nolimits}

\def\Im{\opname{Im}}

\def\dim{\opname{dim}}

\def\rank{\opname{rank}}

\def\supp{\operatorname{supp}}

\def\Dev{\operatorname{Dev}}

\def\span{\operatorname{span}}

\def\Cont{\operatorname{Cont}}
\def\Crit{\operatorname{Crit}}
\def\Spec{\operatorname{Spec}}
\def\Sing{\operatorname{Sing}}
\def\GFQI{\frak{G}}
\def\Index{\operatorname{Index}}

\def\Image{\operatorname{Image}}

\begin{document}

\quad \vskip1.375truein

\def\mq{\mathfrak{q}}
\def\mp{\mathfrak{p}}
\def\mH{\mathfrak{H}}
\def\mh{\mathfrak{h}}
\def\ma{\mathfrak{a}}
\def\ms{\mathfrak{s}}
\def\mm{\mathfrak{m}}
\def\mn{\mathfrak{n}}
\def\mz{\mathfrak{z}}
\def\mw{\mathfrak{w}}
\def\Hoch{{\tt Hoch}}
\def\mt{\mathfrak{t}}
\def\ml{\mathfrak{l}}
\def\mT{\mathfrak{T}}
\def\mL{\mathfrak{L}}
\def\mg{\mathfrak{g}}
\def\md{\mathfrak{d}}
\def\mr{\mathfrak{r}}
\def\Cont{\operatorname{Cont}}
\def\Crit{\operatorname{Crit}}
\def\Spec{\operatorname{Spec}}
\def\Sing{\operatorname{Sing}}
\def\GFQI{\text{\rm GFQI}}
\def\Index{\operatorname{Index}}
\def\Cross{\operatorname{Cross}}
\def\Ham{\operatorname{Ham}}
\def\Fix{\operatorname{Fix}}
\def\Graph{\operatorname{Graph}}
\def\id{\text\rm{id}}
\def\HI{\operatorname{HI}}

\title[Contact instanton Floer cohomology]
{Legendrian contact instanton cohomology and its
spectral invariants on the one-jet bundle}
\author{Yong-Geun Oh, Seungook Yu}
\address{Center for Geometry and Physics, Institute for Basic Science (IBS),
77 Cheongam-ro, Nam-gu, Pohang-si, Gyeongsangbuk-do, Korea 790-784
\& POSTECH, Gyeongsangbuk-do, Korea}
\email{yongoh1@postech.ac.kr}
\address{POSTECH, \&
Center for Geometry and Physics, Institute for Basic Science (IBS),
77 Cheongam-ro, Nam-gu, Pohang-si, Gyeongsangbuk-do, Korea 790-784}
\email{yso1460@postech.ac.kr}
\thanks{This work is supported by the IBS project \# IBS-R003-D1}

\date{December 12, 2023}

\begin{abstract} In the present paper, we develop the Floer-style elliptic Morse theory
for the Hamiltonian-perturbed contact action functional attached to the Legendrian links.
Motivated by the present authors' construction \cite{oh-yso:weinstein} of the perturbed action functional defined on the Carnot path space introduced in \cite{oh-yso:weinstein} as the canonical generating function,
we apply a Floer-type theory to the aforementioned functional and associate the \emph{Legendrian contact instanton cohomology},
denote by $HI^*(J^1B,H;R)$, to each Legendrian submanifold
contact isotopic to the zero section of one-jet bundle. 
Then we give a Floer theoretic
construction of Legendrian spectral
invariants and establish their basic properties. This theory subsumes
the Lagrangian intersection theory and spectral invariants on the cotangent
bundle previously developed by the first-named author in \cite{oh:jdg,oh:cag},
and its extension to exact immersed Lagrangian submanifolds.
The main ingredient for the study is the interplay between
the geometric analysis of the Hamiltonian-perturbed contact instantons
and the calculus of contact Hamiltonian geometry.
\end{abstract}


\maketitle

\tableofcontents

\section{Introduction}

This is a sequel to the series \cite{oh:contacton-Legendrian-bdy}-\cite{oh:shelukhin-conjecture},
 which is also the first paper in another series of papers
in preparation in which we make a Floer theoretic construction of a system of Legendrian spectral invariants
for Legendrian links on tame contact manifolds and investigate their entanglement structure.
However,  after the necessary analytic foundation of
the equation are all established in \cite{oh:contacton-Legendrian-bdy}-\cite{oh:shelukhin-conjecture},
the present paper is self-contained,  the nature of which is largely dynamical and topological.
The heart of the matter of the present paper, besides the construction of Legendrian
contact instanton cohomology and its spectral invariants, lies in the interplay between the geometric analysis
 of perturbed contact instantons and
the calculus of contact Hamiltonian geometry.

This being said, the main purpose of the present paper is to carry out the contact counterpart  of the first named author's Floer theoretic construction of the spectral invariants on the one-jet bundle which we believe coincides with the Viterbo-type $\GFQI$ spectral invariants
used by Th\'eret in \cite{theret}, Bhupal \cite{bhupal} and Sandon \cite{sandon}.
(See Appendix \ref{sec:gfqi-spectral-invariants} for some quick review on the latter construction.)

\subsection{Contact Hamilton's equation and perturbed action functional}

The starting point of the work \cite{oh:jdg} was
Weinstein's ingenious observation that the classical action functional is a generating function of
the time-one image $\phi_H^1(o_{T^*B})$ of the zero section $T^*B$ under the Hamiltonian flow of 
$H = H(t,x)$.  Weinstein is motivated by Laudenbach-Sikorav's broken-trajectory approximation
of the action functional. We will call this observation \emph{Weinstein's de-approximation} of 
Laudenbach-Sikorav's construction of $\GFQI$. 

Similarly as in the symplectic case treated in \cite{oh:jdg},
the first step for our purpose in the present paper is to formulate the Legendrian counterpart of
Weinstein's de-approximation. More precisely, let
$$
\lambda = dz - pdq
$$
be the standard contact one-form on $J^1B$ and $\xi =: \ker \lambda$ the associated
contact structure. We then consider the
Legendrian submanifold $R = \psi_H^1(o_{J^1B})$ which is the time-one image of the contact flow $\psi_H^t$
associated to the time-dependent function $H = H(t,y)$ with $y = (x,z) \in J^1B$.
We denote by
$$
\pi_B: J^1B \to B,\quad \pi_{\text{\rm cot}}
: J^1B \to T^*B, \quad z: J^1B \to \R
$$
the obvious projections. Then we write
$$
\lambda = dz - \pi^*\theta
$$
where $\theta = pdq$ is the Liouville one-form on the cotangent bundle $T^*B$.

The contact Hamilton's equation $\dot y = X_H(t,y)$ can be split into
\be\label{eq:contact-Hamilton's-eq}
\begin{cases}
(\dot \gamma - X_H(t,\gamma(t))^\pi = 0 \\
\gamma^*\lambda + H(t,\gamma(t))\, dt = 0
\end{cases}
\ee
as done in \cite{oh:contacton-Legendrian-bdy}.
The following perturbed action functional is introduced in \cite{oh:contacton-Legendrian-bdy}, \cite{oh:entanglement1}
for the general Legendrian submanifold in any contact manifold which we now specialize to the case of $J^1B$
equipped with the canonical contact form $\lambda = dz - \pi^*\theta = dz - pdq$.

Let $H = H(t,y)$ be a contact Hamiltonian and a pair $(R_0,R_1)$ of Legendrian submanifolds of $J^1B$.
We denote by $\psi_H^t$ its flow and introduce the notation
\be\label{eq:phiHt}
\phi_H^t:= \psi_H^t (\psi_H^1)^{-1}
\ee
following the notation from \cite{oh:entanglement1}. While $\psi_H^t(y)$ is the $H$
contact Hamiltonian trajectory with \emph{initial condition $y$}, i.e., at $t = 0$,
$\psi_H^t (\psi_H^1)^{-1}(x)$ is the $H$ contact Hamiltonian trajectory with \emph{final condition $x$}, i.e., at $t = 1$.
(See the introduction of \cite{oh:cag} for a similar usage of systematic notations in the symplectic
context. However we warn readers not to get confused with
the standard notation $\phi_H^t$ for the symplectic Hamiltonian flow used above $\phi_H^1$.)

Motivated by the results from \cite{oh:entanglement1,oh:perturbed-contacton-bdy},
we consider the functional $\CA_H: \CL(R_0,R_1) \rightarrow \R$ defined by
\be\label{eq:perturbed-action}
\CA_H(\gamma) : =
\int_0^1 e^{g_{(\phi^t_H)^{-1}} (\gamma(t))} (\lambda(\dot \gamma (t)) + H_t (\gamma (t))) dt
\ee
where the function $g_\psi$ is the \emph{conformal exponent} of $\psi$ defined by
$\psi^*\lambda = e^{g_\psi}\lambda$.
However we need to make a couple of adjustments. First of all note that
$\CA_H(\gamma) = 0$ for any contact Hamiltonian trajectories. Secondly to be consistent with
the classical action functional on the cotangent bundle which is given by
$$
\CA_H^{\text{\rm cl}}(\gamma) = \int_\gamma p\, dq - H(t,\gamma(t))\, dt
$$
the following variation of $\CA_H$ turns out to be the right choice for our purpose,
after consideration of contact analog to Weinstein's observation in \cite{oh-yso:weinstein}.

\begin{defn}[Effective action functional] We define $\widetilde \CA_H: \CL(J^1B)\to \R$ to be
\bea\label{eq:tildeCAH-intro}
\widetilde{\mathcal{A}}_H(\gamma) & = & - \CA_H (\gamma) + z(\gamma(1)) \nonumber \\
& = & - \int_0^1 e^{g_{(\phi_H^t)^{-1}} (\gamma (t))} (\lambda_{\gamma(t)} (\dot \gamma (t)) + H(\gamma(t))) dt + z(\gamma(1)).
\eea
\end{defn}
Note that $\widetilde \CA_H(\gamma) = z(\gamma(1))$ \emph{on shell}, i.e.,
for any contact Hamiltonian trajectory.
(See Lemma \ref{lem:CAHu=CAw} for the relationship between
the perturbed action functional \eqref{eq:perturbed-action} and
the standard Reeb action functional in contact geometry.)

\subsection{Legendrian contact instanton cohomology}

From now on, we will  assume $B$ is a compact smooth manifold without boundary, 
unless mentioned otherwise.

We recall  the following standard definitions in contact geometry.
\begin{defn} Let $\lambda$ be a contact form of a contact manifold $(M, \xi)$ and $R \subset M$ a connected Legendrian submanifold. Denote by $\mathfrak{Reeb} (M,\lambda)$ (resp. $\mathfrak{Reeb} (M,R;\lambda)$) the set of closed Reeb orbits (resp. the set of self Reeb chords of $R$).
\begin{enumerate}
\item We define ${\rm Spec} (M,\lambda)$ to be the set
$$
{\rm Spec} (M,\lambda) = \bigg\{ \int_\gamma \lambda \mid \lambda \in \mathfrak{Reeb} (M,\lambda) \bigg\}
$$
and call the \emph{action spectrum} of $(M,\lambda)$.
\item We define the \emph{period gap} to be the constant given by
$$
T (M,\lambda) := \inf \bigg\{ \int_\gamma \lambda \mid \lambda \in \mathfrak{Reeb} (M,\lambda) \bigg\} > 0
$$
\end{enumerate}
We define ${\rm Spec} (M,R;\lambda)$ and the associated $T (M,\lambda;R)$ similarly using the set $\mathfrak{Reeb} (M,R;\lambda)$ of Reeb chords of $R$.
\end{defn}
We set $T (M,\lambda) = \infty$ (resp. $ T (M,\lambda;R) = \infty $) if there is no closed Reeb orbit (resp. no $R$-Reeb chord). Then we define
$$
T_\lambda (M;R) := \min \{ T (M,\lambda), T (M,\lambda;R) \}
$$
and call it the \emph{(chord) period gap} of $R$ in $M$.

Next we consider a two-component link of the type $(\psi_H^1(R_0),R))$.
The following form of the curves  
\be\label{eq:translated-Ham-chord-intro}
\gamma^\pm(t) = \phi^t_H\left(\overline{\gamma}^\pm_{T_\pm} (t)\right)
\ee
appear as the asymptotic limits of finite energy solutions
\eqref{eq:perturbed-contacton-bdy-intro} \cite{oh:entanglement1,oh:perturbed-contacton-bdy}
where
$$
(\overline{\gamma}^\pm, T_\pm) \in \mathfrak{Reeb}\left(\psi^1_H(R_0), R_1\right).
$$
The following definition is introduced by the first named author in \cite{oh:shelukhin-conjecture}.

\begin{defn}[Translated Hamiltonian chords] \label{defn:trans-Ham-chords-intro}
Let $(R_0,R_1)$ be a 2-component Legendrian link of $(M,\lambda)$.
We call a curve $\gamma$ of the form \eqref{eq:translated-Ham-chord-intro}
a \emph{translated Hamiltonian chord} from $R_0$ to $R_1$ with $\gamma(0)  \in R_0$. We denote by
$$
\mathfrak X^{\text{\rm trn}}((R_0,R_1);H)
$$
the set thereof.
\end{defn}
The set $\mathfrak X^{\text{\rm trn}}((R_0,R_1);H)$ plays the role of generators of the
Floer cohomology associated to \eqref{eq:perturbed-contacton-bdy-intro} below.

We form a $\frac{1}{2}\Z$-graded free $\Z_2$-module
$$
CI^*(H:B) = \Z_2 \{\mathfrak{X}(H;o_{J^1B},o_{J^1B})\}.
$$
Its boundary map is defined by counting the cardinality of the moduli space
$$
\CM (H,J; \gamma^-, \gamma^+) := \widetilde \CM (H,J; \gamma^-, \gamma^+)/\sim
$$
of finite energy perturbed contact instantons $u : \R \times [0,1] \rightarrow J^1B$ 
which are finite energy solutions of 
\be\label{eq:perturbed-contacton-bdy-intro}
\begin{cases}
(du - X_H \otimes dt)^{\pi,(0,1)} = 0, \quad d(e^{g_{H,u}}(u^*\lambda + H dt)\circ j) = 0 \\
u(\tau,0),\; u(\tau,1) \in o_{J^1B}\\
u(-\infty) = \gamma^-, \, u(+\infty) = \gamma^+
\end{cases}
\ee
for each $\gamma^-, \gamma^+ \in \mathfrak{X}(H;o_{J^1B},o_{J^1B})$
satisfying  $\mu(\gamma^+) - \mu(\gamma^-) = 1$.

For the study of the perturbed equation \eqref{eq:perturbed-contacton-bdy-intro},
we will always take the \emph{CI-bulk data} $(H,J)$ so that they are flat near $t = 0, \, 1$.  
Such a choice for $H = H(t,x)$ can be always assume WLOG by making the associated
contact isotopy constant near $t = 0, \, 1$ by reparameterizing the isotopy.  Such a 
reparameterization does not affect our main purpose of studying contact topology of 
Hamiltonian dynamics. Such a choice of $J$ can be always made again without affecting
the study of moduli spaces, especially the transversality study of the moduli spaces
of \eqref{eq:perturbed-contacton-bdy-intro}.
Since this choice will be much more important than in the symplectic case, we formalize 
the property by naming it.

\begin{defn}[Boundary flatness] We say a pair $(H,J)$ \emph{boundary flat} if 
we have
$$
(H_t,J_t) \equiv (H_0,J_0)
$$
in a neighborhood of ${0, \, 1} \subset [0,1]$.
\end{defn}

\begin{rem} \begin{enumerate}
\item As Floer did the elliptic Morse theory with the symplectic action functional in symplectic
geometry \cite{floer:Morse,floer:Ham}, we now do similar elliptic Morse theory with
the effective action functional $\widetilde \CA_H$.
\item The data such as $(H,J)$ in symplectic geometry is commonly called a \emph{Floer data}.
To avoid confusion of the current pair $(H,J)$ on contact manifolds, we call them a more
neutral name a \emph{CI-bulk data}, where `CI' stands for `contact instanton'.
\end{enumerate}
\end{rem}

The main interest of the present paper for the study of \eqref{eq:perturbed-contacton-bdy-intro}
is the case of  the pair
$$
(R_0,R_1) = (o_{J^1B}, o_{J^1B})
$$
where $o_{J^1B}$ is the zero section in the one-jet bundle.
\begin{lem}\label{lem:upshot}
The pair $(\lambda, o_{J^1B})$ is special in that it does not
carry any \emph{nonconstant}  Reeb chord $(\gamma,T)$, i.e., we have
$$
\mathfrak{Reeb}\left(\lambda; o_{J^1B}, o_{J^1B}\right) = \{\ell_q \mid q \in o_{J^1B} \} 
\cong B
$$
where $\ell_q:[0,1] \to J^1B$ is the constant path valued at $q \in B \cong o_{J^1B}$.
\end{lem}
For the pair, we define the matrix element
\be\label{eq:nHJsigma}
n_{(H,J)} (\gamma^- , \gamma^+ ) := \#_{\Z_2}\left(\CM (\gamma^-, \gamma^+)\right)
\ee
for such a pair $(\gamma^-, \gamma^+)$, and a homomorphism
$$
\delta_{(H,J)} : CI^*(H,J: B) \rightarrow CI^*(H,J: B)
$$
given by
$$
\delta_{(H,J)} ( \gamma^+ ) = \sum_\beta n_{(H,J)} (\gamma^- , \gamma^+) \, \gamma^-.
$$
\begin{rem}[Vanishing of curvature $\mathfrak m_0$]\label{rem:varying-triad}
 The property stated in Lemma \ref{lem:upshot} is \emph{not} preserved under the contact diffeomorphism
for a fixed contact form , i.e.,  the pair $(\lambda; \psi(o_{J^1B}))$ may admit a nonconstant
self Reeb chord. To utilize the property of the zero section, one needs to vary
both arguments of $(\lambda,R_0)$ to $(\psi_*\lambda, \psi(R_0))$ for $R_0 = o_{J^1B}$.
Furthermore to make the Floer-type theory work with the least obstruction, we actually vary the
whole triad $(M,\lambda, J)$ to $(M, \psi_*\lambda,\psi_*J)$. This practice was made 
in our previous article \cite{oh:entanglement1} and will be employed here too.
The upshot of this practice is that it will remove the curvature term 
$\mathfrak m_0$ which plagues the Floer-type theory for the (relative) contact homology
in the literature. This has been one of the major obstructions to defining Floer-type
homology via the machinery of pseudoholomorphic curves on symplectization.
\end{rem}

The following theorem from \cite{oh:entanglement1} is where and how the 
boundary-flatness requirement for the bulk-data $(H,J)$ is utilized in 
\cite{oh:entanglement1,oh:shelukhin-conjecture}.

\begin{thm}[Theorem 1.3 \cite{oh:entanglement1}] \label{thm:bdymap-intro}
Suppose $(M,\xi)$ is tame and $R \subset M$ is a compact Legendrian submanifold. Let $\lambda$ be a tame contact form such that
\begin{itemize}
\item $\psi = \psi^1_H$ and $ \| H \| < T_\lambda (M,R) $.
\item the pair $(\psi(R), R)$ is transversal in the sense that $\psi (R) \pitchfork Z_R$.
\end{itemize}
Let $J$ be a $\lambda$-adapted almost complex structure. Then
$$
\delta_{(H,J)} \circ \delta_{(H,J)} = 0.
$$
Furthermore for two different choices of such $J$ or of $H$, the complex are chain-homotopic to each other.
\end{thm}

By specializing to the triad $(J^1B,\lambda,o_{J^1B})$, this enables us to perform the Legendrian analogue 
of Floer's construction
of Lagrangian intersection cohomology \cite{floer:Morse} (for the exact case).

\begin{thm}[Legendrian contact instanton cohomology on $J^1B$]
 Let $H$ be nondegenerate. Suppose $(H,J)$ is boundary flat and Fredholm-regular. 
 Then following hold:
\begin{enumerate}
\item $\delta_{(H,J)} \circ \delta_{(H,J)} = 0$. We define the homology of $(CI^*(H,J;B), \delta_{(H,J)})$
$$
HI^* (H,J;B)=\text{\rm Ker }\,  \delta_{(H,J)} / \text{\rm Im }\, \delta_{(H,J)}
$$
and call it the \emph{(perturbed) contact instanton Floer cohomology} of $(H,J)$ on $B$.
\item The isomorphism type of the graded module $HI^* (H,J;B)$ does not depend on the choice of regular $(H,J)$'s.
\item For each given nondegenerate Hamiltonian $H$, there exists a canonical PSS-type isomorphism
$$
h_H^{\text{\rm PSS}}: H^*(B) \to HI^*(H,J;B).
$$
\end{enumerate}
\end{thm}
An immediate corollary is the following Arnold-type intersection result
for the zero section of $J^1B$, which is also a consequence of
the existence of Legendrian version of $\GFQI$
\cite{viterbo,theret,sandon}.

For any given subset $S \subset M$, we consider the following union
\be\label{eq:S-trace}
Z_S: = \bigcup_{t \in \R} \phi_{R_\lambda}^t(S)
\ee
which is called the \emph{Reeb  trace} of a subset $S\subset M$ in \cite{oh:entanglement1}.
\begin{cor}
For any ambient contact isotopy $\psi^t$ of the zero section $o_{J^1B}$,
we have
$$
\#\left(\psi^1(o_{J^1B}) \cap Z_{o_{J^1B}}\right) \geq \rank H^*(B,\Z_2).
$$
\end{cor}
We have do doubt that this result holds over the integer coefficients
after the full study of orientations of the moduli space of
perturbed contact instantons which is however postponed elsewhere.

We occasionally write the equation \eqref{eq:perturbed-contacton-bdy-intro} in the following simple suggestive form
as done in \cite{oh:perturbed-contacton-bdy}:
\be\label{eq:perturbed-contacton-intro}
\overline{\del}^\pi_H u = 0, \quad d(e^{g_{H,u}}(u^* \lambda_H \circ j)) = 0
\ee
where we will use the following notations
\beastar
\overline{\del}^\pi_H u & : = & (du - X_H(t,u) \otimes dt)^{\pi(0,1)}\\
u^*\lambda_H & : = & u^*\lambda + u^*H_t\, dt\\
g_{H,u} & :=  & g_{(\phi^t_H)^{-1}}(u).
\eeastar

\subsection{Legendrian spectral invariants via perturbed contact instantons}

Utilizing these background geometric preparation and analytic foundation,
we carry out the Floer theoretic construction of Legendrian spectral invariants,
which satisfy the following properties which are contact counterparts of those proved in
 \cite{oh:jdg,oh:cag}. For this purpose, we introduce the following definitions
\beastar
E^+(H) & := &  \int^{1}_{0} \max_{y} H_t(y)\, dt, \quad
E^-(H) := \int^1_0 - \min_y  H_t(y)\,  dt \nonumber\\
\|H\| &: = & E^+(H) + E^-(H) = \int_0^1 \left(\max_y H_t(y) - \min_y H_t(y)\right)\, dt
\eeastar
similarly as in the symplectic geometry.

\begin{thm}[Theorem \ref{thm:spectral-invariants}]
\label{thm:spectral-invariants-intro} Assume that $B$ is a closed manifold. Let
$H = H(t,y)$ be a contact Hamiltonian and denote $R = \psi_H^1(o_{J^1B})$.
The map $(H;a) \mapsto \rho(H;a)$ for $a \in H^*(B)$ satisfies the following:
\begin{enumerate}
\item (Spectrality)  $\rho(H;a) \in \Spec(o_{J^1B})$ for all $a \neq H^*(B)$.
\item (Hofer continuity)  Let $H, \, H': E \to \R$ be two Hamiltonians. Then
\be\label{eq:Hofer-continuity}
\int_0^1 \min(H_t - H'_t)\, dt \leq \rho(H;a) - \rho(H';a) \leq \int_0^1 \max(H_t - H'_t)\, dt
\ee
\item $\rho(0;a) = 0$ for all $a \in H^*(B)$.
\end{enumerate}
\end{thm}

\begin{cor} \begin{enumerate}
\item
If $H \geq H'$, then $\rho(H;a) \geq \rho(H';a)$. In particular it holds that
$\rho(H;a) \geq 0$ if $H \geq 0$.
\item We have $\int_0^1 \min H_t\, dt \leq \rho(H;a) \leq \int_0^1 \max H_t, \, dt$.
\end{enumerate}
\end{cor}

We compare these properties with those that the $\GFQI$ invariants
$c(a;S)$ satisfy which we summarize in Theorem \ref{thm:spec-S} in Appendix \ref{sec:gfqi-spectral-invariants}.
Note that there has been no counterpart of Hofer continuity for $c(a;S)$ proved in the literature
before, as far as we are aware.

\begin{rem}[Choice of homotopy $\{H^s\}$]\label{rem:choice-homotopy} 
We would like to emphasize that \emph{no statement in the above theorem
involves the conformal exponent of the relevant contact Hamiltonians or contactomorphisms}. The same form of inequality is proved
for the Lagrangian Floer cohomology by the first named author
in \cite{oh:jdg}. However the standard linear interpolation homotopy
$s\mapsto (1-s) H + s H'$ used in the proof of the relevant inequality in \cite{oh:jdg}
will produce some inequality involving the conformal factors. We need two different procedures from the symplectic case
e.g., for the proof of \eqref{eq:Hofer-continuity}:
\begin{enumerate}
\item  We have to employ a  homotopy $\{H^s\}$ between $H$ and $H'$ which is a kind different from $s\mapsto (1-s) H + s H'$
and also use a special kind of elongation function $\chi:\R \to [0,1]$ for our purpose.
(See the proof of Theorem \ref{thm:spectral-invariants}.)
\item We have to also use the so called curvature-free perturbed equation.
(See Equation \eqref{eq:perturbed-contacton-bdy-oJ1B} and Remark \ref{rem:curvature-free}.)
\end{enumerate}
\end{rem}
Let $H^\alpha$ and $H^\beta$ be given and let $\{H^s\}_{s \in [0,1]}$ be the homotopy
mentioned in Remark \ref{rem:choice-homotopy} with $H^0 = H^\alpha$, $H^1 = H^\beta$.
Let $\chi$ be the elongation function mentioned therein too.

Along the way, we also establish the following two fundamental a priori crucial estimates.

\begin{thm}[Uniform $\pi$-energy bound; Theorem \ref{thm:Epi-bound}]\label{thm:Epi-bound-intro}
For any given $\gamma^\alpha, \, \gamma^\beta \in \mathfrak{X}(J^1B,H;o_{J^1B}, o_{J^1B})$ and
$u \in \widetilde \CM(H^\chi,J;\gamma^\alpha,\gamma^\beta)$, we have
\be\label{eq:Epi-bound-intro}
E_H^\pi(u) \leq \widetilde \CA_{H^\alpha}(\gamma^\alpha)
 - \widetilde \CA_{H^\beta}(\gamma^\beta) + \int^{1}_{0} \max_{y}(H^{\beta}_{t} - H^{\alpha}_{t}) dt
\ee
\end{thm}

We also recall the definition of the vertical energy for contact instantons introduced
in \cite{oh:contacton} (for the closed string case) and \cite{oh:entanglement1} for the open string case
and establish its a priori bound applied to the continuity equation \eqref{eq:perturbed-contacton-intro}.

\begin{thm}[Uniform vertical energy bound; Theorem \ref{thm:Eperp-bound}]\label{thm:Eperp-bound-intro}
Let $u$ be any finite energy solution of \eqref{eq:perturbed-contacton-intro}. Then we have
$$
E_H ^\perp(u) \leq |\widetilde \CA_{H^\alpha}(\gamma^\alpha)| 
+ |\widetilde \CA_{H^\beta}(\gamma^\beta)|
+ E^+(H^\beta) + E^-(H^\alpha).
$$
\end{thm}

\subsection{Discussion and what remains to do}

\subsubsection{Relationship with (immersed) exact Lagrangian
Floer theory on the cotangent bundle}

When restricted to the cotangent bundle $T^*B$, our Legendrian contact instanton cohomology on $J^1B$
subsumes the Floer theory of \emph{immersed} exact Lagrangian Floer theory
with the invariance property wider than under the ambient Hamiltonian
isotopies. (See \cite{EHS-gafa}, \cite{ono:lag-leg} and \cite{akhao-joyce}
for related results studied by the traditional symplectic Floer theory.)
In this regard, the following facts are the ingredients which are
relevant:
\begin{enumerate}
\item There is a canonical lifting of Hamiltonian isotopy of
any \emph{immersed exact Lagrangian submanifold
equipped with its Liouville primitive} of $T^*B$ to a contact
isotopy of Legendrian submanifolds. We have only to lift the
Liouville primitives which is well-known to have explicit formula.
(See \cite[Proposition 3.4.8]{oh:book1}.)
\item The resulting contact isotopy is strict, i.e., we can realize
the isotopy by a strict contact isotopy and hence no conformal
exponent will appear for the lifted theory.
\item There is a canonical way of associating $(dz - \pi^*\theta)$-adapted 
CR almost complex structures on $J^1B$ to each
$\omega_0 = - d\theta$ tame almost complex structures on $T^*B$.
(See Appendix \ref{sec:JonT*B} for the explanation.)
\item Each perturbed Floer trajectory  on $T^*B$ can be naturally
lifted to a perturbed contact instantons.
(See Appendix \ref{sec:JonT*B} for the explanation.)
\item There is a canonical conversion rule from the effective action 
functional $\widetilde \CA_H$ and the classical action $\CA_H^{\text{\rm cl}}$
when the Hamiltonian $H$ on $T^*B$ is lifted to $J^1B$.
(See Proposition \ref{prop:lifted-action}.)
\end{enumerate}

This also implies that our Legendrian spectral invariants
naturally subsume the Lagrangian spectral invariants
on the cotangent bundle defined by
the first-named author in \cite{oh:cag} with wider invariance
properties than what is established therein. 

\subsubsection{Axioms of Legendrian spectral invariants}

In a sequel \cite{oh-yso:equivalence},
we will prove the following coincidence result
of our spectral invariants and the $\GFQI$ spectral invariants.

\begin{thm}[\cite{oh-yso:equivalence}] Let $S$ be a  $\GFQI$ of the Legendrian submanifold $R = \psi_H^1(o_{J^1B})$. Then
$$
\rho(H;a) = \pm c(a;S)
$$
(where the uniform choice of $\pm$ depends on the sign conventions used in the literature).
\end{thm}
Our equivalence proof will be the contact counterpart of Milinkovic's equivalence proof
\cite{milinkovic1} based on the result from \cite{oh-yso:weinstein}.
One corollary of this coincidence theorem will prove that $c(a;S)$ share the inequality
\eqref{eq:Hofer-continuity} provided the generating function $S$ generates the time-one image
$R = \psi_H^1(o_{J^1B})$ of the Hamiltonian $H$. (Compare this with the properties
of $c(a;S)$ laid out in Theorem \ref{thm:spec-S} in Appendix.)
In particular the Hofer-type continuity will also hold for $c(a;S)$.

It seems plausible that the following also hold similarly as
for the Lagrangian spectral invariants from \cite{oh:cag}, \cite{MVZ}.

\begin{ques} Do the following properties hold for the assignment $(H,a) \mapsto \rho(H;a)$?
\begin{enumerate}
\item Denote by $\mu = PD[pt] \in H^n(B)$ the orientation class, i.e., the Poincar\'e dual to the point class.
Then
$$
\rho(H;\mu) = - \rho(H;1).
$$
\item For any $a, \, b \in H^*(B)$,
$$
\rho(H*H';a\cup b) \leq \rho(H;a) + \rho(H';b)
$$
where $H*H'$ is the concatenation of $H$ and $H'$.
\item $\rho(H;1) \geq 0$ for all $H$.
\end{enumerate}
\end{ques}
These properties are known to hold for the case of \emph{mean-normalized Hamiltonians}
 for the symplectic spectral invariants constructed in \cite{oh:jdg}, \cite{MVZ}.
The proofs of these properties will involve the product structure on the contact instanton Floer cohomology $HI^*$
which will be studied elsewhere.

\subsubsection{$C^0$ contact dynamics}

M\"uller and Spaeth initiated contact dynamics in the $C^0$-level
\cite{mueller-spaeth1,mueller-spaeth2,mueller-spaeth3} motivated
by its symplectic version \cite{oh:hameo1}.
They have discovered many interesting new phenomena in the
contact case, and they especially discovered the importance of
the conformal exponent in their study of $C^0$ contact dynamics and of the
$C^0$-completion of the contact diffeomorphism group.
However their study on the effect of the conformal exponent is not
optimal in that the implicit relationship between the contact
Hamiltonian and the conformal exponent
$$
\frac{\del g_{\psi_H^t}}{\del t} = -R_\lambda[H](\psi_H^t)
$$
is not exploited. Our definitions of contact action functional \eqref{eq:tildeCAH-intro}
and of spectral invariants nicely combine the two which hints how one
should approach the $C^0$ aspect of contact dynamics.

In fact \eqref{eq:Hofer-continuity}
enables us to continuously extend the map $H \mapsto \rho(H;a)$ to any
$C^0$-limit of Hamiltonians which sets the grounds for the
$C^0$-study of Legendrian spectral invariants
as in the symplectic geometry. It would be an interesting
and important study to understand what are common points and
what are different points  between
Lagrangian spectral invariants and the current Legendrian spectral
invariants and their dynamical consequences on the group of contactomorphisms. We refer to \cite{seyfad}, \cite{oh:lag-capacity}
for some detailed study of spectral invariants in view of
$C^0$-Hamiltonian dynamics. In view of the recent great success of
the study of Floer theoretic spectral invariants in
$C^0$-Hamiltonian dynamics,  and 2 dimensional area-preserving
dynamics \cite{GHMSS},  \cite{polterov-shelukhin}, triggered by
\cite{GHS},  we anticipate that our Legendrian contact instanton spectral
invariants will play similar role in the study of $C^0$ contact
dynamics. This is a subject of future study.

\subsubsection{Further relevant open problems}

The following problems are those which are natural
continuation of the study made in the present paper. In general
one may search for the Legendrian analogues to the following
known results for the topology of Lagrangian submanifolds in the
cotangent bundle to name a few:
\begin{enumerate}
\item Existence of a graph-selector of the type
\cite{amorim-oh-santos} for compact Legendrian submanifold of $J^1B$
of degree 1 projection to the base $B$.
\item Study of Hofer-type geometry of the set of Legendrian submanifolds
studied in \cite{oh:jdg}, \cite{milinkovic2}, \cite{oh:lag-capacity}.
We suspect that this kind of study for Legendrian submanifolds
will have interesting applications to thermodynamics through
the contact geometric study of thermodynamic equilibria and their
interactions. (See \cite{mrugala,MNSS}, \cite{BCT} and
\cite{lim-oh} to name a few and many references therein
about contact geometric study of thermodynamics.)
\item Another obvious direction of researches is
to construct a Fukaya-type category and to ask various categorial
questions such as  the generation results of the types
given \cite{nadler-zaslow,nadler},
\cite{fukaya-seidel-smith1,fukaya-seidel-smith2} and \cite{abouzaid}.\end{enumerate}
\bigskip

Indeed, the first-named author is currently working on the construction of
the Fukaya-type category on general contact manifolds as a continuation
of the study of his quantitative study of contact topology of compact
Legendrian submanifolds of tame contact manifolds \cite{oh:entanglement1},
and investigate general entanglement structure of  the system of
spectral invariants arising from the moduli spaces of (perturbed)
contact instantons intertwining the long-range interaction of
components of general Legendrian links on general contact manifolds
in \cite{oh:entanglement2}, and their applications elsewhere.

\bigskip

\noindent{\bf Conventions and Notations:}

\medskip

\begin{enumerate}
\item {(Contact Hamiltonian)} The contact Hamiltonian of a time-dependent contact vector field $X_t$ is
given by
$$
H: = - \lambda(X_t).
$$
We denote by $X_H$ the contact vector field whose associated contact Hamiltonian is given by $H = H(t,x)$, and its flow by
$ \psi^t_H $.
\item When $\psi = \psi^1_H$, we say $H$ generates $\psi$ and write $H \mapsto \psi$.
\item We write $\phi^t_H := \psi^t_H \circ (\psi^1_H)^{-1}$. (Warning: \emph{Do not get confused with 
the common notation $\phi_H^t$ for the symplectic Hamiltonian flow}.)
\item {(Gauge transformation $\Phi_H$)} For given $H$, we call
a gauge transformation the one-to-one correspondence
$$
\Phi_H: \CL(\psi_H^1(R_0),R_1) \to \CL(R_0,R_1)
$$
defined by $(\Phi_H)^{-1}(\gamma) =: \overline \gamma$ for $\gamma \in\CL(R_0,R_1)$  
i.e., $\Phi_H(\overline \gamma) = \gamma$.
\item {($\gamma$ and $\overline \gamma$)} Throughout the paper, we consistently denote by $\overline \gamma$ the gauge-transformed
Reeb chord of $(\psi_H^1(R_0),R_1)$ of $\gamma$
when $\gamma$  is a Hamiltonian chord of $(R_0,R_1)$.
\item We will try to consistently use the following notations whenever appropriate:
\begin{itemize}
\item $(q,p,z)$ a point of $J^1B$ or the canonical coordinates thereof
\item $x = (q,p)$ a point of $T^*B$, \
\item $y = (x,z)$ a point in $J^1B$.
\end{itemize}
\item {(Reeb vector field)} We denote by $R_\lambda$ the Reeb vector field for the contact form $\lambda$
and its flow by $\phi^t_{R_\lambda}$.
\item {($J'$)} For given $J = \{J_t\}$, we denote by $J'$ defined by
$$
J' = \{J_t'\}_{0 \leq t \leq 1}, \quad J_t' := (\phi^t_H)^*J_t = (d \phi^t_H)^{-1} J_t (d \phi^t_H).
$$
\item $\gamma$, $\overline \gamma$; $\gamma$ is a translated Reeb chord and
$\overline \gamma$ the associated Reeb  chord.
\end{enumerate}

\section{Preliminaries}

\subsection{Nondegeneracy of Reeb chords}
\label{subsec:reeb-chords}

In this subsection, we consider the case of Reeb chords which corresponds to the
case of contact Hamiltonian $H \equiv -1$ in our sign convention.

Let  $(R_0, R_1)$ be a pair of Legendrian submanifolds. 
Consider the boundary value problem
\be\label{eq:Reeb-problem-bdy}
\left\{
\begin{array}{l}
\dot {\overline \gamma} (t) = T R_\lambda (\overline \gamma (t)), \\
\overline \gamma (0) \in R_0, \quad \overline \gamma (1) \in R_1
\end{array} \right.
\ee
for $\overline \gamma : [0,1] \rightarrow M$.
We denote by $(\overline \gamma, T)$ a solution of \eqref{eq:Reeb-problem-bdy} and denote by
$ \mathfrak{Reeb}(R_0, R_1) $ the space of such solutions.

\begin{defn}
We say a Reeb chord $(\overline \gamma, T)$ of $(R_0, R_1)$ is \emph{nondegenerate} if the linearization map
$ \Psi_{\overline \gamma} = d\phi^T_{R_\lambda} : \xi_{\overline \gamma (0)} \rightarrow \xi_{\overline \gamma (1)} $ satisfies
$$ \Psi_{\overline \gamma} (T_{\overline \gamma (0)} R_0) \pitchfork T_{\overline \gamma (1)} R_1 \quad {\rm in} \quad \xi_{\overline \gamma (1)}$$
or equivalently
$$ \Psi_{\overline \gamma} (T_{\overline \gamma (0)} R_0) \pitchfork T_{\overline \gamma (1)} Z_{R_1} \qquad {\rm in} \quad T_{\overline \gamma (1)} M. $$
\end{defn}
Here $Z_{R_1}$ is the Reeb trace of $R_1$, i.e.
$$ 
Z_{R_1} = \bigcup_{t \in \R} \phi^t_{R_\lambda} (R_1). 
$$
(When $T = 0$, it is well-known that the constant loop is nondegenerate in the Morse-Bott sense.
See \cite{oh:contacton-transversality} for the details of its proof.)

\subsection{Some contact Hamiltonian calculus}
\label{subsec:contact-Hamiltonian}

In this section, we organize some useful results concerning the contact Hamiltonian dynamics
which will enable us to systematically study the change of spectral invariants of Legendrian submanifolds
under the Legendrian isotopy. Majority of the results in this section are widely known to experts.
(See \cite{mueller-spaeth2}, for example.) We also refer readers to
\cite{BCT} for a nice exposition on contact Hamiltonian mechanics to get a grasp on it.

Let $(M,\xi)$ be a co-oriented contact manifold and let $\lambda$ be a
contact form with $\xi = \ker \lambda$. Denote by $\Cont(M,\xi)$ (resp. $\Cont_0(M,\xi)$)
the set of contact diffeomorphisms (resp. the identity component thereof).
We denote by $R_\lambda$ the Reeb vector field of $\lambda$.

\begin{defn} For given coorientation preserving contact diffeomorphism $\psi$ of $(M,\xi)$
we call the function $g$ appearing in
$$
\psi^*\lambda = e^g \lambda
$$
the \emph{conformal exponent} for $\psi$ and denote it by $g = g_\psi$.
\end{defn}

The following lemma is a straightforward consequence of the identity
$$
(\phi \psi)^* \lambda = \psi^* \phi^* \lambda.
$$

\begin{lem}
Let $\lambda$ be given and denote by $g_\psi$ the function $g$ appearing above associated to $\psi$. Then
\begin{enumerate}
\item $g_{\phi \psi} = g_\phi \circ \psi + g_\psi$ for any $\phi, \psi \in \Cont (M,\xi)$,
\item $g_{\psi^{-1}} = - g_\psi \circ \psi^{-1}$ for any $\psi \in \Cont (M,\xi)$.
\end{enumerate}
\end{lem}

\begin{defn} A vector field $X$ on $(M,\xi)$ is called \emph{contact} if
there exists a smooth function $f: M \to \R$ such that
$$
\CL_X \lambda = f \lambda.
$$
The associated function $H$ defined by
\be\label{eq:contact-Hamiltonian}
H = - \lambda(X)
\ee
is called the \emph{contact Hamiltonian} of $X$. We also call $X$ the
contact Hamiltonian vector field associated to $H$.
\end{defn}

A straightforward calculation shows
$$
f = - R_\lambda[H].
$$
For given general function $H$, the associated contact Hamiltonian vector field
$X_H$ has decomposition
$$
X_H = X_H^\pi - H R_\lambda \in \xi \oplus \R \langle R_\lambda \rangle
$$
where the projection $X_H^\pi$ to $\xi$ is uniquely determined by the equation
$$
X_H^\pi \rfloor d\lambda = dH - R_\lambda[H] \lambda.
$$
(This is a special case of \cite[Lemma 2.1]{oh-wang2} with different sign convention.
 Also see \cite[Section 2]{mueller-spaeth2} for some relevant discussions.)

\begin{lem}\label{lem:XH-decompose} Solving contact Hamilton's equation $\dot x = X_H(t,x)$ is equivalent to
finding $\gamma:\R \to M$ that satisfies
\be\label{eq:XH-decompose}
(\dot \gamma - X_H(t,\gamma(t)))^\pi = 0, \quad \gamma^*\lambda + H(t,\gamma(t))\, dt = 0
\ee
\end{lem}

We also state the following lemma for the later purpose, whose proof is
a straightforward calculation. (See \cite[Lemma 2.2]{mueller-spaeth2}.)

\begin{lem}\label{lem:inverse-Hamiltonian} Let $\psi_t$ be a contact isotopy satisfying
$\psi_t^*\lambda = e^{g_t} \lambda$ and generated by the vector field $X_t$ with
its contact Hamiltonian  $H(t,x) = H_t(x)$.
\begin{enumerate}
\item Then the inverse isotopy $\psi_t^{-1}$ is generated
by the contact Hamiltonian, denoted by $\overline H$,
\be\label{eq:inverse-Hamiltonian}
\overline H(t,x) = - e^{-g_t(x)} H(t, \psi_t(x)).
\ee
\item If $\psi'_t$ is another contact isotopy with corresponding $g'_t$ and $H'_t$,
then the product $\psi_t' \psi_t$ is generated by the Hamiltonian
\be\label{eq:product-Hamiltonian}
H '\# H(t,x): = H'(t,x) + e^{g_t'((\psi_t')^{-1}(x))} H(t, (\psi_t')^{-1}(x)).
\ee
\end{enumerate}
\end{lem}
In particular, we have
\be\label{eq:barH'-sharp-H}
\overline H' \# H(t,x) = e^{-g_t'}(H(t,\psi'(t)) - H'(t,\psi'(t))).
\ee
We remark that these formulae are reduced to the standard formulae
in the Hamiltonian dynamics in symplectic geometry,
if $\psi_t, \, \psi_t'$ are $\lambda$-strict contactomorphisms so that $g_t \equiv 1 \equiv g_t'$.

\subsection{Perturbed contact action functional $\CA_H$}
\smallskip

To obtain a good variational problem, we need to study the action functional
defined by the first-named author in \cite{oh:perturbed-contacton-bdy}
which is called
the \emph{perturbed action functional}. In fact, \cite{oh:perturbed-contacton-bdy} studies the action functional on general contact manifold $(M,\xi)$, while in the present paper
we focus on the 1-jet bundles $M = J^1B$ with the standard contact structure
$$
\xi = \ker \lambda = \ker (dz - pdq)
$$
where $B$ is a closed manifold of dimension $n$. From now on we assume that $H$ is a
\emph{compactly supported} Hamiltonian function on $J^1B$, which will be sufficient for
the purpose of the present paper.
We denote by
$$
\CH := C^\infty_c ([0,1] \times J^1B, \R)
$$
the space of compactly supported Hamiltonians. For each $r \in \R_+$, we define
\be \label{eq:CKR}
\CH_r = \{H \in C^\infty([0,1]\times J^1B,\R) \mid \supp H \subset D_r(J^1B)\}
\ee
which provides a natural filtration of the space $\CH $. Then we have
$$
\CH =  \bigcup_{r \in \R_+} \CH_r
$$
and equip the union $\cup_r \CH_r$ with the direct limit topology of $\{\CH_r\}_{r > 0}$.

\begin{defn}[Perturbed action functional]
Let $H = H(t,y)$ be a contact Hamiltonian and a pair $(R_0,R_1)$ of Legendrian submanifolds of $J^1B$.
Consider the free path space
$$
\CL := C^{\infty}([0,1];J^1B) = \{ \gamma: [0,1] \rightarrow J^1B \}.
$$
and a path space
$$
\CL(R_0,R_1) = \CL(J^1B;R_0,R_1) := \{ \gamma \in \CL \mid \gamma(0) \in R_0, \; \gamma(1) \in R_1 \}.
$$
We define a functional $\CA_H: \CL(R_0,R_1) \rightarrow \R$ given by
\be\label{eq:perturbed-action2}
\CA_H(\gamma) := \int_0^1 e^{g_{(\phi^t_H)^{-1}} (\gamma(t))} \gamma^* \lambda_H
= \int_0^1 e^{g_{(\phi^t_H)^{-1}} (\gamma(t))} (\lambda(\dot \gamma (t)) + H_t (\gamma (t))) dt
\ee
Here $\lambda_H := \lambda + H dt$ with slight abuse of notation as in \cite{oh:perturbed-contacton-bdy}.
\end{defn}

Note that, when $H = 0$, we have
$$\CA_0 (\gamma) = \CA (\gamma) = \int_0^1 \gamma^* \lambda $$
which is the standard contact action functional for the Reeb dynamics.

The following lemma connects the perturbed action functional with the unperturbed (standard) one.

\begin{lem}[Lemma 2.2, \cite{oh:perturbed-contacton-bdy}]\label{lem:CAHu=CAw}
For each given path $\gamma \in \CL(R_0,R_1)$, consider the path $\overline \gamma \in \CL(\psi^1_H(R_0), R_1)$ defined by
$$ \overline \gamma (t) := (\phi^t_H)^{-1}(\gamma(t)) $$
where $\phi^t_H := \psi^t_H \circ (\psi^1_H)^{-1} $.
Then we have
$$
\CA_H(\gamma) = \CA(\overline \gamma).
$$
\end{lem}

We now quote the following first variation formula of the action functional
\eqref{eq:perturbed-action2} on the free space $\CL$ from \cite{oh:perturbed-contacton-bdy}.

\begin{prop}[Proposition 2.3, \cite{oh:perturbed-contacton-bdy}]
For any vector field $\eta$ along $\gamma \in \CL$, we have
\bea\label{eq:perturbed-first-var}
\delta \CA_H (\gamma)(\eta)
&=& \int_0^1 d\lambda\left((d \phi^t_H)^{-1}(\eta(t)), (d \phi^t_H)^{-1}(\dot \gamma - X_H(t,\gamma(t)))\right) \, dt \nonumber \\
& & + \lambda(\eta(1)) - e^{g_{(\psi^1_H}(\gamma(0))} \lambda(\eta(0))
\eea
\end{prop}

An immediate corollary of this first variation formula shows that the Legendrian boundary condition is a natural boundary condition for the action functional $\CA_H$ in that it kills the boundary contribution in the first variation:

\begin{cor}
Let $(R_0,R_1)$ be a Legendrian pair. Then we have
$$
\delta \CA_H (\gamma)(\eta)
= \int_0^1 d\lambda(d(\phi^t_H)^{-1}(\eta(t)), d(\phi^t_H)^{-1}(\dot \gamma - X_H(t,\gamma(t)))) dt
$$
on $\CL(R_0,R_1)$.
\end{cor}
Motivated by this first variation formula, we consider the following action functional.

\begin{defn}[Effective action functional] We define $\widetilde \CA_H: \CL(J^1B)\to \R$ to be
\bea\label{eq:tildeCAH}
\widetilde{\mathcal{A}}_H(\gamma) & = & - \CA_H (\gamma) + z(\gamma(1)) \nonumber \\
& = & - \int_0^1 e^{g_{(\phi_H^t)^{-1}} (\gamma (t))} (\lambda_{\gamma(t)} (\dot \gamma (t)) + H(\gamma(t))) dt 
+ z(\gamma(1))
\eea
\end{defn}

The following provides the relationship between 
$\widetilde \CA_H$ and the classical action functional $\CA_H^{\text{\rm cl}}$.
\begin{prop}\label{prop:lifted-action} 
Suppose that $H$ is lifted from $T^*B$, i.e., $H$ has the form
$H(t, q,p, z) = H(t,q,p)$. Denote by $\varphi_H^t: T^*B \to T^*B$
the symplectic Hamiltonian flow generated by $H$.
 Let $L$ be the Lagrangian submanifold given by
$L = \varphi_H^1(o_{T^*B})$. We set
$$
\ell(t) = \pi_{\text{\rm cot}}(\gamma(t))
$$
for a curve $\gamma$ in $J^1B$ with $\gamma(1) \in o_{J^1B}$. Then we have
\be\label{eq:RL}
\widetilde \CA_{H}(\gamma) 
= \CA_H^{\text{\rm cl}}(\ell) + h(\ell(0))
\ee
where $h: L \to \R$ is the canonical Liouville primitive such that
$$
\psi_H^1(o_{J^1B}) = \{(q,p,z) \mid (q,p) \in L, \, z = h(q,p) \}.
$$
\end{prop}
\begin{proof}
By Lemma \ref{lem:CAHu=CAw}, we can rewrite
$$
\widetilde \CA_H(\gamma) = - \int \overline \gamma^*\lambda 
+z(\gamma(1)).
$$
Furthermore  we have
\beastar
- \int \overline \gamma^*\lambda
& = & - \int \overline \gamma^*(dz - \pi_{\text{\rm cot}}^*\theta)\\
& = & \int (\overline \gamma^*(\pi_{\text{\rm cot}}
^*\theta)) 
- \int  \overline \gamma^*(dz)\\
& = & \int \overline \gamma^*(\pi_{\text{\rm cot}}
^*\theta) - z(\overline \gamma(1))
+ z(\overline \gamma(0))
\eeastar
Since $\phi_H^1 = \psi_H^1 (\psi_H^1)^{-1} = id$, we have 
$z(\overline \gamma(1)) = z(\gamma(1))$ and so derive
\be\label{eq:effective-action}
\widetilde \CA_H(\gamma) = \int \overline \gamma^*(\pi_{\text{\rm cot}}
^*\theta)
+ z (\overline \gamma(0)).
\ee
Denote by $\overline H$ the inverse Hamiltonian \eqref{eq:inverse-Hamiltonian}
of $H$ on $J^1B$.
Then we evaluate
\beastar
\overline \gamma^*(\pi_{\text{\rm cot}}
^*\theta)(\del_t)
& = & \theta\left(\frac{d}{dt} \pi_{\text{\rm cot}}
(\overline \gamma(t))\right)\\
& = & \theta\left(d(\pi_{\text{\rm cot}}
)\left(d(\phi_H^t)^{-1}(\dot \gamma) 
+ X_{\overline H}(\overline \gamma(t))\right)\right) \\
& = & \theta\left(d(\pi_{\text{\rm cot}}
 \circ \phi_H^t)^{-1})\left(\dot \gamma 
+ d(\phi_H^t)(X_{\overline H}((\phi_H^t)^{-1})(\gamma(t)))\right)\right).
\eeastar
But a direct evaluation using the formula
\eqref{eq:inverse-Hamiltonian} proves
$$
\lambda((\phi_H^t)_*X_{\overline H})  = -H
$$
which proves $d(\phi_H^t)(X_{\overline H}((\phi_H^t)^{-1}(\gamma(t)) 
= -X_H(\gamma(t))$. Therefore we have derived
$$
\overline \gamma^*(\pi_{\text{\rm cot}}
^*\theta)(\del_t) = \theta \left(d(\pi_{\text{\rm cot}}
 \circ d(\phi_H^t)^{-1}(\dot \gamma - X_H(\gamma))\right).
$$
In particular, if $\phi_H^t = \psi_H^t (\psi_H^1)^{-1}$ are strict  contactomorphisms, i.e., 
$(\phi_H^t)^*\lambda = \lambda$ for all $t$,
which is the case when the Hamltonian $H$ is lifted from $T^*B$, we get
$$
(\phi_H^t)^*dz - (\varphi_H^t (\varphi_H^1)^{-1})^*\theta  = \lambda
$$
and hence
$$
(\varphi_H^t (\varphi_H^1)^{-1})^*\theta = (\phi_H^t)^*dz - \lambda.
$$
We compute
$$
\theta \left(d(\pi_{\text{\rm cot}}
 \circ \phi_H^t)^{-1}(\dot \gamma
-X_H(\gamma)\right)
= (\varphi_H^t (\varphi_H^1)^{-1})_*\theta(\dot \ell - X_H(\ell)).
$$
Therefore we can rewrite
\beastar
 (\varphi_H^t (\varphi_H^1)^{-1})_*\theta((\dot \ell - X_H(\ell))
 & = &( (\phi_H^t)_*dz - \lambda)((\dot \gamma - X_H(\gamma)) \\
 & = & dz(\dot \gamma) - (\gamma^*\lambda(\del_t) + H) \\
 & = &\ell^*\theta (\del_t)- H(t,\ell(t))
 \eeastar 
By integration, we have shown
$$
\int_0^1 \overline \gamma^*(\pi_{\text{\rm cot}}^*\theta)(\del_t)
=  \CA_H^{\text{\rm cl}}(\ell).
$$
Furthermore we consider the Legendrian embedding 
(resp. Lagrangian embedding) 
$\iota_H^{\text{\rm leg}}: B \to J^1B$ 
(resp. $\iota_H^{\text{\rm lag}} : B \to T^*B$) defined by
$$
\iota_H^{\text{\rm leg}}(q) = \psi_H^1(q,0,0), 
\, \left(\text{\rm resp.} \, \iota_H^{\text{\rm lag}}(q) = \varphi_H^1(q,0)\right)
$$
where $\varphi_H^t$ is the \emph{symplectic Hamiltonian flow} of $H$.
Then we have 
$$
\iota_H^{\text{\rm lag}}(q) = \pi_{\text{\rm cot}}
\left(\iota_H^{\text{\rm leg}}(q)\right).
$$
Now we can check the function $h: L \to \R$ defined by
$$
h(x): = z\left(\iota_H^{\text{\rm leg}} \circ (\iota_H^{\text{\rm lag}})^{-1}(x)\right)
$$
is the Liouville primitive  of the Lagrangian submanifold 
$$
L: = \iota_H^{\text{\rm lag}}(B) =  \pi_{\text{\rm cot}}
(\iota_H^{\text{\rm leg}}(B))
$$
i.e., $i_L^*\theta = dh$, which also satisfies \eqref{eq:RL} by construction.
This concludes the proof.
\end{proof}

\begin{rem} It is very satisfying and amusing to see the right hand side
functional of \eqref{eq:RL} is precisely the same functional used in 
\cite{oh:jdg} which was also called the \emph{effective action functional} therein, 
and hence the consistency of the terminology in the two
cases which are related by the canonical process
given in the above proof.
\end{rem}

We now examine the relationship between the critical points of
the aforementioned constrained action functional and the contact Hamiltonian trajectories.
We represent each Reeb chord between $\psi_H^1(R_0)$ and $R_1$ by a pair
$(\overline \gamma, T)$  and $\overline \gamma:[0,|T|] \to J^1B$ is a Reeb chord with action $T$.
(Here $T \neq 0$ since we assume $\psi(R_0) \cap R_1 = \emptyset$.)
We then consider the curves of the form
\be\label{eq:translated-Ham-chord}
\gamma^\pm(t) = \phi_H^t(\overline \gamma^\pm_{T_\pm}(t))
\ee
where 
$$
(\overline \gamma^\pm,T_\pm) \in \mathfrak{Reeb}(\psi_H^1(R_0),R_1)).
$$
The following definition is introduced  in \cite{oh:shelukhin-conjecture}.
 
\begin{defn}[Translated Hamiltonian chords] \label{defn:trans-Ham-chords}
Let $(R_0,R_1)$ be a 2-component Legendrian link of $(M,\lambda)$.
\begin{enumerate}
\item
We call a curve $\gamma$ of the form \eqref{eq:translated-Ham-chord} a \emph{translated Hamiltonian chord} from $R_0$ to $R_1$ with $\gamma(0) \in R_0$.
We denote by
$$
\mathfrak X^{\text{\rm trn}}((R_0,R_1);H)
$$
the set thereof.
\item
We call the intersection $(\psi_H^1)^{-1}(Z_{R_1})\cap R_0$ the set of \emph{$\lambda$-translated Hamiltonian intersection points}.
\end{enumerate}
\end{defn}

The set $\mathfrak X^{\text{\rm trn}}((R_0,R_1);H)$ will play the role of generators of the
Floer homology associated to \eqref{eq:perturbed-contacton-bdy} which appear as
the asymptotic limits of finite energy solutions of \eqref{eq:perturbed-contacton-bdy}, and
$\mathfrak{Reeb}(\psi_H^1(R_0), R_1)$ the role generators of the Floer homology 
of its gauge transform \eqref{eq:contacton-bdy-J0} below.
(See \cite{oh:entanglement1,oh:perturbed-contacton-bdy} for the proof.)

There are three other ways of viewing the set $\mathfrak X^{\text{\rm trn}}((R_0,R_1);H)$ as follows.

\begin{lem}[Lemma 3.4 \cite{oh:shelukhin-conjecture}]\label{prop:trnhamchords-trnReebchords} We have one-one correspondences
$$
\xymatrix{
(\psi_H^1)^{-1}(Z_{R_1})\cap R_0 \ar[d]\ar[r] & \ar[l]\psi_H^1(R_0) \cap Z_{R_1} \ar[d]\\
\mathfrak X^{\text{\rm trn}}((R_0,R_1);H) \ar[u] \ar[r] & \ar[l] \mathfrak{Reeb}(\psi_H^1(R_0), R_1) \ar[u].
}
$$
\end{lem}
We refer readers \cite{oh:shelukhin-conjecture} for detailed discussion on these transformations.

\part{Analysis of perturbed contact instantons on one-jet bundles}

\section{Tameness of one-jet bundles}
\label{sec:Floer-data}

Recall from \cite{oh-wang1} the following definition on general
contact manifold.

Let $(M, \xi)$ be a contact manifold.
A \emph{contact triad} for the contact manifold $(M, \xi)$ is a triple $(M,\lambda, J)$
whose explanation is now in order. With $\lambda$ given, we have the Reeb vector field $R_\lambda$
uniquely determined by the equation $R_\lambda \rfloor d\lambda = 0, \, R_\lambda \rfloor \lambda = 1$.
Then we have decomposition $TM = \xi \oplus \R \{R_\lambda\}$. We denote by $\Pi: TM \to TM$ the associated idempotent whose image is $\xi$.
A \emph{CR almost complex structure $J$} is an endomorphism $J: TM \to TM$ satisfying $J^2 = - \Pi$
or more explicitly
$$
(J|_\xi)^2 = - id|_\xi, \quad J(R_\lambda) = 0.
$$
\begin{defn}
We say $J$ is \emph{adapted to $\lambda$} if $d\lambda(Y, J Y) \geq 0$ for all $Y \in \xi$ with equality only when $Y = 0$. The associated contact triad metric is given by
$$
g=g_\xi+\lambda\otimes\lambda.
$$
\end{defn}

\subsection{Definition of tameness of contact manifolds}

Since the one-jet bundle $J^1B$ is not compact, we need to examine the $C^0$-bound of (perturbed) contact instantons
in the study of compactness property of the moduli space thereof.
For this purpose, we introduce a general class of contact manifolds,
called \emph{tame contact manifolds} in \cite[Section 5]{oh:entanglement1}.

We first introduces a class of barrier functions which will control the $C^0$ bounds of contact instantons on noncompact contact manifolds.

\begin{defn}[Reeb-tame function]
Let $(M,\xi)$ be a contact manifold equipped with contact form $\lambda$. A function $\psi : M \rightarrow \R$ is called $\lambda$-\emph{tame} (at infinity) if
$$ \CL_{R_\lambda} d\psi = 0 $$
on $M \setminus K$ for a compact subset $K$.
\end{defn}

The following is a subcase of the notion considered in \cite{oh:entanglement1}
in which a more general notion of \emph{quasi-pseudoconvexity} is introduced.
For the purpose of the present paper, this restricted class of tame
contact manifold will be sufficient.

\begin{defn}[Contact $J$-pseudoconvexity]
Let $J$ be a $\lambda$-adapted CR almost complex structure. Let $U \subset M$ be an open subset. We call a function $\psi : U \rightarrow \R$ \emph{contact $J$-pseudoconvex} if
\bea\label{eq:J-convex}
-d(d\psi \circ J) &\geq& 0 \quad {\rm on}\; \xi, \\
R_\lambda \rfloor d(d\psi \circ J) &=& 0
\eea
on $U$. We call such a pair $(\psi,J)$ a \emph{contact pseudoconvex pair} on $U$.
\end{defn}
In the sense of \cite{oh:entanglement1}, such  contact form $\lambda$ is \emph{tame} (at infinity) in that $\lambda$
admits a contact pseudoconvex pair $(\psi,J)$ on $M \setminus K$ such that $\psi$ is also a $\lambda$-tame exhaustion function of $M \setminus K$.

\subsection{Choice of adapted CR almost complex structures on $J^1B$}
\label{subsec:lifted-Jg}

We start with the set $\JJ^c_g(T^*B)$ consisting of $\omega_0$-compatible almost complex structures
 on the cotangent bundle $T^*B$ equipped with
the canonical symplectic form $\omega_0 = -d\theta$.
This class of almost complex structures was used by Floer in \cite{floer:Witten} and
by the first-named author in \cite{oh:jdg} for the construction
of Lagrangian spectral invariants.

We now canonically lift each element $J \in \JJ^c_g(T^*B)$ to
a natural $\lambda$-adapted CR-almost complex structure
on $J^1B$ by pulling it back to
$\xi$ by the isomorphism $\xi \to T(T^*B)$ induced by the restriction to $\xi$ of
the projection $d\pi: T(J^1B) \to T(T^*B)$. In particular we denote by
\be\label{eq:tildeJg}
\widetilde J_g
\ee
the lift of the Sasakian almost complex structure $J_g$ on $T^*B$ associated to the metric $g$
on $B$ and by
\be\label{eq:metric-gtilde}
\widetilde g = d\lambda(\cdot, \widetilde J_g \cdot) + \lambda \otimes \lambda
\ee
the triad metric on $J^1B$ associated to the triad $(J^1B, \lambda, \widetilde J_g)$. (See Appendix \ref{sec:JonT*B} for the description of Sasakian
almost complex structures.)

Recall that we equip the one-jet bundle $J^1B$ with the standard contact form
$\lambda = dz - pdq$. 
\begin{defn}[Lifted $CR$-almost complex structures]\label{defn:lifted-CR-J} We call a $CR$-almost complex
structure on $\pi_{\text{\rm cot}} 
^{-1}(U) \subset J^1B$ with an open subset $U \subset T^*B$ a \emph{$T^*B$-lift} if it is lifted to $\xi$ on
$$   
\pi_{\text{\rm cot}}    
^{-1}(U)
$$
by $d\pi_{\text{\rm cot}}
$  from an $\omega_0$-compatible almost complex structure on $U \subset T^*B$.
We denote by
$$
\JJ^c_g(J^1B)
$$
the set of $\lambda$-adapted CR almost complex structures  $J$ that
is a $T^*B$-lift on $\pi_{\text{\rm cot}}^{-1}(T^*B \setminus K)$ for a compact subset $K \subset T^*B$.
\end{defn}

We first prove the following general lemma.

\begin{prop}\label{prop:harmonic} Let $J^1B$ be equipped with a $\lambda$-adapted $CR$-almost
complex structure $J$ which coincides with the
$T^*B$ lift $\widetilde J_g$ of a Sasakian almost complex structure
$J_g$ on the region $|p|_g \geq r$ for some $r> 0$.
Then both functions $z$ and $|p|^2_g$ are harmonic with respect to the triad metric $\widetilde g$
on $J^1B$ associated to the triad $(J^1B,\lambda, J)$.
Here the norm $| \cdot |_g$ is defined by the given Riemannian metric $g$ on $B$.
\end{prop}

\begin{proof}
By the property $\Image \widetilde J_g = \xi$, we have
 $$
 0 = \lambda \circ \widetilde J_g = dz\circ \widetilde J_g - \pi_{\text{\rm cot}}^* \theta \circ  \widetilde J_g
 = dz\circ \widetilde J_g - \theta \circ  J_g.
 $$
Therefore
\be\label{eq:dzJ}
- dz\circ \widetilde J_g = -\theta \circ J_g.
\ee
The following is well-known among the experts and shows that $-\theta$ is the Liouville one-form of $\omega_0$ and $\frac12 |p|_g^2$ is
a symplectization end radial function of the Liouville manifold $T^*B \setminus \{o_{T^*B}\}$.

\begin{lem}\label{lem:d|p|2}  
We have $-\theta \circ J_g = \frac12 d(|p|^2_g)$.
\end{lem}
\begin{proof} This follows from a direct calculation using the definition of
Sasakian almost complex structure $J_g$. See \cite[Section 4 \& 5]{BKO} for an
explicit calculation leading to this formula, especially the last displayed formula
in Section 5 therein. For readers' convenience, we provide the details of calculation in
Appendix \ref{sec:sasaki-J}.
\end{proof}

Then by the choice of $J$, $J$ is a $T^*B$-lift of $J_g$ on the region of
$(q,p,z)$ with $|p|_g > r$ for some $r> 0$. Therefore we have
$$
-d(dz\circ J)  =  -d( \theta \circ J_g)
= d\left(\frac12 d|p|_g^2\right) = 0
$$
provided $|p|_g > r $ for some $r > 0$.
Therefore we have
$$
(\Delta z)\, \omega_0 =- d (dz \circ J) = 0
$$
whenever $(q,p,z)$ satisfies $|p|_g > r$. This proves that the function $z$ is a harmonic function.

For Statement (2), we start with \eqref{eq:dzJ} on the region $|p|_g > r$. Then we compute
\beastar
-d(d|p|^2_g \circ J) & = & -d(d|p|^2_g \circ J_g) = 
2 d((\theta \circ J_g) \circ J_g)) \\
& = & 2d\left((dz \widetilde J_g \widetilde J_g\right) =  -2  d(dz) = 0.
\eeastar
This proves that the function $|p|_g^2$ is a harmonic function on $T^*B$.
This finishes the proof of Proposition \ref{prop:harmonic}.
\end{proof}

\begin{prop}\label{prop:tameness-of-one-jet}
In the one-jet bundle $(J^1B,dz - pdq)$, the function $\psi = \frac12|p|^2 +|z|$ satisfies the $\lambda$-tameness and contact $J$-convexity
outside a compact subset
for any $J \in \CJ^c_g(J^1B)$. In particular $(J^1B,dz - pdq)$ is a tame contact manifold.
\end{prop}

\begin{proof} We first remark that $\psi(x) = 0$ if and only if $z  = 0 = p$ and so $\psi^{-1}(0) = o_{J^1B}$ is compact.
Consider $\{z = 0\}$ on which $\psi = |p|_g^2$ which are plurisubharmonic.
If $z > 0$, we consider the function $\psi_+ := \frac12 |p|^2 + z$, and if $z \leq 0$, we consider $\psi_- : = \frac12|p|^2 - z$
both of which are positive subharmonic functions on $\{z > 0\}$ and $\{z \leq 0\}$ respectively.

We can easily check that
$$
\CL_{R_\lambda} d\psi = d(R_\lambda \rfloor dz) \equiv 0
$$
for both cases. Since both $|p|^2$ and $z$ are pluri-subharmonic functions
for $\widetilde J_g$ and $J\equiv \widetilde J_g$ outside a
compact subset  by Proposition \ref{prop:harmonic},
we have the conditions of contact $J$-convexity, that is,
\beastar
-d(d\psi \circ J) &\geq& 0 \quad {\rm on}\; \xi, \\
R_\lambda \rfloor d(d\psi \circ J) &=& 0
\eeastar
Combining the above, we have finished the proof.
\end{proof}

\section{Perturbed contact Instantons, energy and gauge transformation}
\smallskip

In this section, we follow \cite[Section 3]{oh:perturbed-contacton-bdy}.
Let $\dot \Sigma = \R \times [0,1]$,
$(R_0,R_1)$ be a pair of Legendrian submanifolds in $J^1B$ and let a contact Hamiltonian $H = H(t,y)$ be given.

\subsection{Gauge transformations}

For a given Hamiltonian $H \in \CH$,  we recall from \cite{oh:entanglement1}
that the transformation
$$
\gamma  \mapsto (\Phi_H)^{-1}(\gamma) =: \overline \gamma
$$
satisfies
$$
\Phi_H(\overline \gamma) (t): = \psi^t_H (\psi^1_H)^{-1}(\overline \gamma (t)).
$$
and defines a bijective correspondence
\be\label{eq:gauge}
\Phi_H : \CL (\psi^1_H (R_0), R_1) \to \CL (R_0, R_1).
\ee

Next we apply $\Phi_H$ $\tau$-wise to the map $u$  by
\be\label{eq:gauge-transformation}
u (\tau,t) = \phi^t_H(w(\tau,t)) = \psi_H^t (\psi_H^1)^{-1}(w(\tau,t))
\ee
or equivalently
$$
\Phi_H (w_\tau) = u_\tau, \quad w_\tau = w(\tau,\cdot), \, u_\tau = u(\tau, \cdot).
$$
We write $u = \Phi_H(w)$ by an abuse of notations and call the map \eqref{eq:gauge}
a \emph{gauge transformation}  following the term
used in \cite{oh:cag}, \cite{oh:entanglement1}. We also consider the gauge transformation for
the \emph{nonautonomous case}, i.e., for the case where $H$ depends also on another parameter $s \in [0,1]$,
$\{H^s\}_{s \in [0,1]}$. These gauge transformation enable us to go back and forth between
the perturbed contact instantons (i.e., $H \neq 0$) and the unperturbed ones (i.e., $H = 0$).

\emph{We will freely do this transformation in our convenience when we give the proofs of
many statements on the action estimate and the index calculations.}

\subsubsection{Gauge transformation of autonomous contact instantons}

Consider the Hamiltonian $H = H(t,y)$ and a $t$-dependent $\lambda$-adapted CR almost complex structure
$J = J(t,y)$, i.e., $J = \{J_t\}_{t \in [0,1]}$. We call such a pair an \emph{autonomous CI-bulk datum} (of
perturbed contact instantons).

\begin{defn} Let $\dot \Sigma = \R \times [0,1] \cong D^2 \setminus \{\pm 1\}$ with
the standard coordinates $(\tau,t)$ of $\R \times [0,1] \subset \R^2$.
 A \emph{Hamiltonian perturbed contact instanton} is a map $u:\dot \Sigma \rightarrow J^1B$
that satisfies the following system of equations
\be\label{eq:perturbed-contacton}
\overline{\del}^\pi_H u = 0, \quad d(e^{g_{H,u}}(u^* \lambda_H \circ j)) = 0
\ee
where we abbreviate
\beastar
\overline{\del}^\pi_H u & : = & (du - X_H(t,u) \otimes dt)_{J_t}^{\pi(0,1)}\\
u^*\lambda_H & : = & u^*\lambda + u^*H_t\, dt\\
g_{H,u} & :=  & g_{(\phi^t_H)^{-1}}\circ u.
\eeastar
\end{defn}

For a given such $(H,J)$, we associate another family of $\lambda$-adapted CR almost complex structures
denoted by $J' = J'(t,y)$ defined as follows.

\begin{defn}[$J'$] Let $H = H(t,y)$ be given.
For each one-parameter family $J = \{J_t\}$ of CR-almost complex structures adapted to $\lambda$,
we consider another family $J_t'$ defined by the relation
\be\label{eq:perturbed-cplx-str}
J' = \{J_t'\}_{0 \leq t \leq 1}, \quad J_t' := (\phi^t_H)^*J_t = (d \phi^t_H)^{-1} J_t (d \phi^t_H)
\ee
of $\lambda$-admissible CR almost complex structures.
\end{defn}

Now we have the following equivalence of two equations.

\begin{prop}\label{prop:equivalent-eqns}
Let $J $ and $J_t'$ be as in \eqref{eq:perturbed-cplx-str}.
Let $\dot \Sigma \cong \R \times [0,1]$ and let $g_{H,u}$ be the conformal exponent function defined as above.
Then $u$ satisfies
\be\label{eq:perturbed-contacton-bdy}
\left\{
\begin{array}{l}
(du - X_H \otimes dt)^{\pi,(0,1)}_J = 0, \quad d(e^{g_{H,u}}(u^*\lambda + H dt)\circ j) = 0 \\
u(\tau,0) \in R_0, \quad u(\tau,1) \in R_1
\end{array} \right.
\ee
with respect to $J = \{J_t\}$ if and only if $w$ satisfies
\be\label{eq:contacton-bdy-J0}
\left\{
\begin{array}{l}
\overline{\del}^\pi_{J'} w= 0, \quad d(w^*\lambda \circ j) = 0 \\
w(\tau,0) \in \psi^1_H(R_0), \quad w(\tau,1) \in R_1
\end{array} \right.
\ee
with respect to $J' = \{J_t'\}$.
\end{prop}

\subsubsection{Gauge transformation of nonautonomous contact instantons}

In this subsection, we assume that a pair of data $(H^\alpha,J^\alpha)$,
$(H^\beta,J^\beta)$ are given.

We consider the homotopy of the type
$$
(\{H^s\}, \{J^s\})_{s \in [0,1]}
$$
of $H = H(s,t,y), \, J = J(s,t,y)$ with
$H^s = H(s, \cdot)$ and $J^s = J(s, \cdot)$ satisfying
$$
(H^0,J^0) = (H^\alpha, J^\alpha), \quad (H^1,J^1) = (H^\beta, J^\beta).
$$
We call such a pair a \emph{nonautonomous CI-bulk datum}.

In this case, we consider the $s$-dependent gauge transformations
$$
\Phi_{H^s}: (\psi_{H^s}^1(R_0), R_1) \to (R_0,R_1).
$$
We fix an elongation function $\chi: \R \rightarrow [0,1]$ whose precise
expression will be given in \eqref{eq:chi}.

By elongating the parameter $s \in [0,1]$ to the parameter $\tau \in \R$
by putting $s = \chi(\tau)$, it was shown in \cite{oh:entanglement1} by
a straightforward calculation that it transforms \eqref{eq:contacton-bdy-J0} into
\be\label{eq:perturbed-contacton-bdy-HG}
\begin{cases}
(du - X_{H}(u)\, dt + X_{G}(u)\, ds)^{\pi(0,1)} = 0, \\
d\left(e^{g_{H^\chi,u}}(u^*\lambda + u^*H^\chi dt - u^*G\, d\tau) \circ j\right) = 0,\\
u(\tau,0) \in R_0,\, u(\tau,1) \in R_1.
\end{cases}
\ee
where we recall $g_{H,u}$ is the function on $\Theta$ defined as before.

\begin{rem}\label{rem:curvature-free} Appearance of the terms involving $\tau$-developing Hamiltonian
$$
\Dev_\lambda(\tau \mapsto \Psi^\chi(\tau,t)) =: G
$$
is common in the, so called \emph{curvature-free}, Hamiltonian Floer theory,
which is needed to get the equation \eqref{eq:rho-contacton-bdy} as the outcome of
the \emph{$\tau$-dependent gauge transformation $\Psi^\rho$} above.
See \cite{seidel:pi1}, \cite[Section 21.6.2]{oh:book2} for relevant discussion in the symplectic case.
\end{rem}

\subsection{Asymptotic convergence and charge vanishing}
\label{subsec:subsequence-convergence}

In this section, we summarize the asymptotic convergence result proved in \cite{oh:contacton-Legendrian-bdy}
specialized to the case when
$$
\dot \Sigma = \R \times [0,1], \quad (M,\lambda) = (J^1B, dz - pdq).
$$

\begin{defn}
Let
$w: \R \times [0,1] \to J^1B$ be any smooth map with Legendrian boundary condition $(R_0, R_1)$.
We define the \emph{total $\pi$-harmonic energy} $E^\pi(w)$
by
\be\label{eq:endenergy}
E^\pi(w) = E^\pi_{(\lambda,J;\dot\Sigma,h)}(w)
= \frac{1}{2} \int_{\dot \Sigma} |d^\pi w|^2
\ee
where the norm is taken in terms of the triad metric on $J^1B$.
\end{defn}

\begin{defn} Assume $(\lambda, \vec R)$ is nondegenerate
and  $w$ converges in $C^\infty$-topology as $\tau \to \infty$.
 We associate two
natural asymptotic invariants at each puncture defined as
\bea
T & := & \frac{1}{2}\int_{[0,\infty) \times [0,1]} |d^\pi w|^2
+ \int_{\{0\}\times [0,1]}(w|_{\{0\}\times [0,1]})^*\lambda\label{eq:TQ-T}\\
Q & : = & \lim_{r \to \infty} \int_{\{r\}\times [0,1]}((w|_{\{0\}\times [0,1] })^*\lambda\circ j).\label{eq:TQ-Q}
\eea
(Here we only look at positive punctures. The case of negative punctures is similar.)
We call $T$ the \emph{asymptotic contact action}
and $Q$ the \emph{asymptotic contact charge} of the contact instanton $w$ at the given puncture.
\end{defn}

It follows (see \cite[Remark 6.4]{oh:contacton-Legendrian-bdy}) that
$$
T = \frac{1}{2}\int_{[s,\infty) \times [0,1]} |d^\pi w|^2
+ \int_{\{s\}\times [0,1]}(w|_{\{s\}\times [0,1]})^*\lambda, \quad
\text{for any } s\geq 0
$$
does not depend on $s$ whose common value is nothing but $T$.

The following is a special case applied to
the one-jet bundle $J^1B$ proved in \cite{oh:contacton-Legendrian-bdy}.
(See also \cite{oh-yso:index}.)

\begin{thm}[See Corollary 1.9 \cite{oh-yso:index}]\label{thm:charge-vanishing}
Assume $(\lambda,\vec R)$ are nondegenerate.
Suppose that $w(\tau, \cdot)$ satisfies
\eqref{eq:contacton-bdy-J0} and
converges as $\tau \to \infty$ in the strip-like coordinate
at a puncture $p \in \del \dot \Sigma$ with associated Legendrian pair $(R,R')$.
Then its asymptotic charge $Q$ vanishes and the convergence is exponentially fast.
\end{thm}

\begin{rem}
There are similar results for perturbed contact instantons. (See \cite[Section 8]{oh:perturbed-contacton-bdy}.) However it is sufficient to use the above unperturbed results through the gauge transformation for our purpose
in the present paper.
\end{rem}

\subsection{Off-shell energy of contact instantons}
\label{subsec:lambda-energy}

Now we borrow the discussion from
\cite{oh:contacton}, \cite{oh:entanglement1} applied to
the special case
$$
\dot \Sigma = \R \times [0,1] \cong D^2 \setminus \{\pm 1\}
$$
and define the off-shell energy of contact instantons $E(w)$ that
will have two components, one the $\pi$-energy and the other
the vertical energy or $\lambda$-energy.

\subsubsection{The $\pi$-energy $E^\pi_{J,H}$}

We start with the $\pi$-energy of perturbed contact instanton.
\begin{defn}[The $\pi$-energy of perturbed contact instanton]
Let $u:\R \times [0,1] \rightarrow J^1B$ be any smooth map. We define
$$ E^\pi_{J,H} (u) := \frac{1}{2} \int e^{g_{H,u}} |( d^\pi u - X^\pi_H (u) \otimes dt)^\pi |^2_J. $$
\end{defn}

Then we have the following energy identity between the maps satisfying \eqref{eq:perturbed-contacton-bdy} and those satisfying \eqref{eq:contacton-bdy-J0}, when $J' = \{J'_t\}$ is the one given by \eqref{eq:perturbed-cplx-str}.

\begin{prop}[Proposition 3.8, \cite{oh:perturbed-contacton-bdy}]
Let $J' = \{J_t'\}$ be as in \eqref{eq:perturbed-cplx-str}. For any smooth map $w:\R \times [0,1] \rightarrow J^1B$,
let $u$ be as above. Then
\be\label{eq:perturbed-energy}
E^\pi_{J,H} (u) = E^\pi_{J'} (w).
\ee
\end{prop}

We have the crucial action identity for the energy which provides the gradient structure of the perturbed contact instanton equation.

\begin{thm}[Theorem 3.10, \cite{oh:perturbed-contacton-bdy}]
Let $\phi^t_H = \psi^t_H \circ (\psi^1_H)^{-1}$ as above. Let $u$ be any finite energy solution of \eqref{eq:perturbed-contacton-bdy} associated to the pair $(H,J)$ as in \eqref{eq:perturbed-cplx-str} with the asymptotic limits
$$
\gamma_\pm (t) := \lim_{\tau \rightarrow \pm\infty} u(\tau,t).
$$
Let $w$ be the map defined as above and consider the paths given by
$$
\overline{\gamma}_\pm (t) = (\phi^t_H)^{-1} (\gamma_\pm (t)).
$$
Then $\overline{\gamma}_\pm $ are Reeb chords from $\psi^1_H (R_0)$ to $R_1$ and satisfy
\be\label{eq:action-identity}
E^\pi_{J,H} (u) = \CA_H (\gamma_+) - \CA_H (\gamma_-)=\CA (\overline{\gamma}_+) - \CA (\overline{\gamma}_-).
\ee
\end{thm}

\subsubsection{The $\lambda$-energy $E^\lambda(u)$}

Next we borrow the presentation of $\lambda$-energy
from \cite[Section 5]{oh:contacton}, \cite[Section 11]{oh:entanglement1}
specialized to the current case of
one-jet bundles and for the maps defined on $\R \times [0,1]$.

As mentioned in \cite[Section 11]{oh:entanglement1},  the Riemann surfaces that
are relevant to the purposes of the present paper are of the following three types:
\begin{situ}[Charge vanishing]\label{situ:charge-vanish}
\begin{enumerate}
\item
First, we mention that the \emph{starting} Riemann surface will be an open Riemann surface
$$
\dot \Sigma \cong \R \times [0,1]
$$
together with an contact instanton with Legendrian pair boundary condition $(R_0,R_1)$.
\item $\C$ which will appear in the bubbling analysis at an interior point of $\dot \Sigma$,
\item $\H = \{ z \in \C \mid \Im z \geq 0\}$ which will appear in the bubbling analysis at a boundary point of $\dot \Sigma$.
\end{enumerate}
\end{situ}
An upshot is that \emph{the asymptotic charges vanish in all these three cases.}

This being said, we follow the procedure exercised in \cite{oh:contacton}
for the closed string case. We introduce the following class of test functions.
Especially the automatic charge vanishing in our current circumstance also enables us
to define the vertical part of energy, called the $\lambda$-energy whose definition is in order.

\begin{defn}\label{defn:CC} We define
\be
\CC = \left\{\varphi: \R \to \R_{\geq 0} \, \Big| \, \supp \varphi \, \text{is compact}, \, \int_\R \varphi = 1\right\}
\ee
\end{defn}

Then on the given strip-like neighborhood $\pm [R,\infty) \times [0,1] \cong
D_\delta(p) \setminus \{p\}$ for sufficiently large fixed $R> 0 $, we can write
$$
w^*\lambda \circ j = df
$$
for some function $f$.
\begin{defn}[Contact instanton potential] We call the above function $f$
the \emph{contact instanton potential} of the contact instanton charge form $w^*\lambda \circ j$ on $D_\delta(p) \setminus \{p\}$.
\end{defn}

By the $\tau$-translation, we may assume
$f$ is defined on $[0,\infty) \times S^1 \to \R$. \emph{Using the vanishing of
asymptotic charge}, we can explicitly write the potential as
\be\label{eq:f-defn}
f(z) = \int_{+\infty}^z w^*\lambda \circ j
\ee
where the integral is over any path from $\infty$ to $z$ along a path in
$[0,\infty) \times [0,1]$. By the closedness of $w^*\lambda \circ j$ on
$[0,\infty) \times [0,1]$, the integral is well-defined and satisfies
$w^*\lambda \circ j = df$. (Compare this with  \cite[Formula above (5.5)]{oh:contacton}
where the general case with nontrivial charge is considered.)

We denote by $\psi$ the function determined by
\be\label{eq:psi}
\psi' = \varphi, \quad \psi(-\infty) = 0, \, \psi(\infty) = 1.
\ee
\begin{defn}\label{defn:CC-energy} Let $w$ satisfy $d(w^*\lambda \circ j) = 0$. Then we define
\beastar
E_{\CC}(j,w;p) & = & \sup_{\varphi \in \CC} \int_{D_\delta(p) \setminus \{p\}} df\circ j \wedge d(\psi(f)) \\
& = &\sup_{\varphi \in \CC} \int_{D_\delta(p) \setminus \{p\}}  (- w^*\lambda ) \wedge d(\psi(f)).
\eeastar
\end{defn}

We note that
$$
df \circ j \wedge d(\psi(f)) = \psi'(f) df\circ j \wedge df = \varphi(f) df\circ j  \wedge df \geq 0
$$
since
$$
df\circ j  \wedge df = |df|^2\, d\tau \wedge dt.
$$
Therefore we can rewrite $E_{\CC}(j,w;p)$ into
$$
E_{\CC}(j,w;p) = \sup_{\varphi \in \CC} \int_{D_\delta(p) \setminus \{p\}} \varphi(f) df \circ j \wedge df.
$$
The following proposition shows that  the definition of $E_{\CC}(j,w;p)$ does not
depend on the constant shift in the choice of $f$.

\begin{prop}[Proposition 11.6 \cite{oh:entanglement1}]\label{prop:a-independent}
For a given smooth map $w$ satisfying $d(w^*\lambda \circ j) = 0$,
we have $E_{\CC;f}(w) = E_{\CC,g}(w)$ for any pair $(f,g)$ with
$$
df = w^*\lambda\circ j = dg
$$
on $D^2_\delta(p) \setminus \{p\}$.
\end{prop}

This proposition enables us to introduce the following vertical energy where we write $E_\pm^\lambda: = E_{\pm \infty}^\lambda$
on $\R \times [0,1] \cong D^2 \setminus \{\pm 1\}$.

\begin{defn}[Vertical energy]
We define the \emph{vertical energy}, denoted by $E^\perp(w)$, to be the sum
$$
E^\perp(w) = E^\lambda_+(w) + E^\lambda_-(w)
$$
\end{defn}

Now we define the final form of the off-shell energy.
\begin{defn}[Total energy]\label{defn:total-enerty}
Let $w:\dot \Sigma \to Q$ be any smooth map. We define the \emph{total energy} to be
the sum
\be\label{eq:final-total-energy}
E(w) = E^\pi(w) + E^\perp(w).
\ee
\end{defn}

\begin{rem}[Uniform $C^1$ bound]
The upshot  is that the Sachs-Uhlenbeck \cite{sachs-uhlen}, Gromov \cite{gromov:pseudo} and Hofer
\cite{hofer:symplectization} style bubbling-off analysis
can be carried out with this choice of energy. (See \cite{oh:contacton,oh:entanglement1} for
the details of this bubbling-off analysis.) In particular
\emph{all moduli spaces of finite energy perturbed contact instantons we consider in the present paper will
have uniform $C^1$-bounds inside each given moduli spaces.}
\end{rem}

\section{Maximum principle}
\smallskip

In this section, we study $C^0$-bounds of contact instantons and their
perturbed ones. Since we assume that $H$ is compactly supported,
it is enough to consider the case of unperturbed contact instantons
for the adapted CR almost complex structure outside a compact subset.

The upshot of our consideration of tame contact manifold is the
following $C^0$ bounds for contact instantons.

\begin{thm}[$C^0$-bound of (unperturbed) contact instantons]
\label{thm:max-principle}
Let $B$ be a compact manifold, $H$ be a contact Hamiltonian with $\psi^1_H (o_{J^1B}) \pitchfork Z = o_{T^*B}\times \R$ and
$J' \in \widetilde{\CJ}^c$. Let $(\overline{\gamma}_\pm,T_\pm)$ be two Reeb chords with $T_\pm \neq 0$. Suppose $w(\tau,\cdot) \to \overline \gamma^\pm$ as $\tau \to \pm\infty$ respectively.

We take $J'$ so that $J'_t \equiv J_0$ is $t$-independent.
Then there exists a constant $r = r(H,J_0)$ such that for any contact instanton
$w : \R \times [0,1] \rightarrow J^1B$ satisfying
\be\label{eq:contacton-zero-bdy}
\begin{cases}
\overline{\del}^\pi w = 0, \quad d(w^*\lambda \circ j) = 0 \\
w(\tau,0) \in \psi^1_H(o_{J^1B}), \quad w(\tau,1) \in o_{J^1B}\\
\lim_{\tau \rightarrow \pm\infty} w(\tau,t) = \overline{\gamma}_\pm (T_\pm t),
\end{cases}
\ee
we have
$$
{\rm Image}\, w \subset D_r(J^1B)
$$
\end{thm}

This is a special case of \cite[Theorem 5.11]{oh:entanglement1}.
For readers's convenience, we provide its proof which is much simpler
for $(J^1B,dz - pdq)$ than for the general case.

We start with the following which is an immediate consequence of Proposition \ref{prop:harmonic}.
\begin{lem}\label{lem:harmonic-subharmonic}
Let $w: \dot \Sigma \to J^1B$ be any contact instanton in $(J^1B, J')$ associated to $J \in \CJ^c_g(J^1B)$. Then  whenever $|p\circ w(z)|_g > r$, the following hold:
\begin{enumerate}
\item the function $z\circ w$ is harmonic.
\item $|p \circ w|^2_g$ is subharmonic.
\end{enumerate}
Here the norm $| \cdot |$ is defined by the given Riemannian metric $g$ on $B$.
\end{lem}
\begin{proof} By the choice of $J$, the definition the associated $J'$ and
the assumption that $H$ is compactly supported, we have $J' = J = \widetilde J_g$
if $|p(w(z)|_g > r$ for some sufficiently large $r > 0$ near the point $w(z)$. We decompose
$$
w = (\pi_{\text{\rm cot}} \circ w, z \circ w)
$$
and write $v: = \pi_{\text{\rm cot}} \circ w :\dot \Sigma \to T^*B$. Then we have 
$\delbar_{J_g}v =0$ thereon. We derive
\bea\label{eq:Deltazw}
-d(d(z \circ w) \circ j) & = & -d(dz dw \circ j) = -d(dz (d^\pi w + w^*\lambda) \circ j) \nonumber\\
& = & -d (dz J' d^\pi w) - d(w^*\lambda \circ j) \nonumber \\
& = & -d (dz \widetilde J_g\, d^\pi w) - d(w^*\lambda \circ j)
\eea
Since $- dz \widetilde J_g = - \theta \circ \widetilde J_g = \frac12 d|p|_g^2$, we have
\be\label{eq:d|pv|2}
- dz J' d^\pi w = - dz \widetilde J_g\,  d^\pi w = \frac12 d|p|_g^2\,  dv
= \frac12d( d|p\circ v|_g^2).
\ee
Therefore we obtain $ -d (dz \widetilde J_g d^\pi w) =0$.
On the other hand, we have $d(w^*\lambda\circ j) = 0$ by the defining equation of
the contact instanton. Substituting these two into \eqref{eq:Deltazw}, we have proved $\Delta \, (z\circ w) = 0$
which is Statement (1).

For Statement (2), we start with \eqref{eq:d|pv|2}. Then we first compute
\beastar
-d(d|p\circ w|^2 \circ j) & = & -d(d|p\circ v|_g^2 \circ j) = 2 d\left((dz J' d^\pi w)\circ j\right)\\
& = & 2 d\left((dz J' J' d^\pi w)\right) = - 2 d\left((dz d^\pi w)\right).
\eeastar
Then by definition, since $\lambda = dz - \pi^*\theta$, we have
$$
dz(d^\pi w) = \pi^*\theta(d^\pi w) = \theta(dv)
= v^*\theta.
$$
Therefore we have derived
$$
-d(d|p\circ w|^2 \circ j) = -v^*d\theta = v^*\omega_0.
$$
This proves $\Delta |p\circ w|^2_g \geq 0$ thanks to $\delbar_{J_g} v = 0$, i.e., $|p\circ w|^2_g$
 is a subharmonic function.
\end{proof}

\begin{proof}[Wrap-up of the proof of Theorem \ref{thm:max-principle}]
By the asymptotic condition
$$
\lim_{\tau \rightarrow \pm\infty} w(\tau,t) = \overline{\gamma}_\pm (T_\pm t),
$$
it follows from Proposition \ref{prop:harmonic} that
$(J^1B, dz - pdq)$, $(\psi, J')$ and $(\psi^1_H (o_{J^1B}), o_{J^1B})$ satisfy the required conditions for
applying \cite[Theorem 5.11]{oh:entanglement1} for $\psi = |p|^2+ |z|$ and hence
$$
{\rm Image} \, w \subset ((|p|^2)^{-1}[-r_1, r_1]) \cap (z^{-1}(-r_2,r_2))
$$
for some $r_1,r_2 > 0$. By compactness of $o_{J^1B}$ and the nondegeneracy condition
$\psi^1_H (o_{J^1B}) \pitchfork Z_{o_{J^1B}}$ where $Z_{o_{J^1B}}: = o_{T^*B}\times \R$ is the Reeb trace of $o_{J^1B}$,
we have finitely many Reeb chords from $\psi^1_H (o_{J^1B})$ to $o_{J^1B}$. Therefore we can choose $r > 0$
independent of Reeb chords $\overline \gamma$ and
$$
((|p|^2)^{-1}[-r_1, r_1]) \cap (z^{-1}(-r_2,r_2)) \subset D_r(J^1B).
$$
This finishes the proof.
\end{proof}

Applying the gauge transformation $u(\tau, t) = \phi^t_H (w(\tau,t))$ and using the properties
\be\label{eq:J-J'}
J = (\phi^t_H)_\ast J', \quad J' = J,
\ee
we have $J =J' = \widetilde J_g$ outside a compact subset by the compact support 
hypothesis of $H$. Therefore
we have the $C^0$-bounds of perturbed contact instantons by the maximum principle.

\begin{cor}
Under the same hypotheses as in Theorem \ref{thm:max-principle},
there exists a constant $r = r(H,J)$ such that for any perturbed contact instanton
$u : \R \times [0,1] \rightarrow J^1B$ satisfying
\be\label{eq:perturbed-contacton-bdy-oJ1B}
\begin{cases}
(du - X_H \otimes dt)^{\pi(0,1)} = 0, \quad d(e^{g_{H,u}}(u^*\lambda + H dt)\circ j) = 0 \\
u(\tau,0),\; u(\tau,1) \in o_{J^1B}
\end{cases}
\ee
we have
$$
{\rm Image} \, u \subset D_r(J^1B).
$$
\end{cor}

\begin{proof}
Since ${\rm Image} \, u \subset D_{r'}(J^1B)$ for some $r' > 0$ by above corollary, we have
${\rm Image} \, u \subset \phi^t_H (D_{r'}(J^1B))$. Since it is compact, we have $r > 0$ such that
$$ \phi^t_H (D_{r'}(J^1B)) \subset D_r(J^1B). $$
This finishes the proof.
\end{proof}

\section{Moduli space of finite energy (perturbed) contact instantons}
\label{sec:moduli-space}

In this section, we introduce two main generic transversality results for the (perturbed)
contact instanton moduli spaces by following \cite{oh:contacton-gluing}. 
Since we will be back and forth
between the perturbed and unperturbed contact instantons via the gauge transformation, 
we will distinguish
them by putting `overline' on the unperturbed ones as in \cite{oh:entanglement1}.

Now we consider general \emph{compactly supported} contact Hamiltonian 
$H : J^1B \rightarrow \R$ such that
$$
\psi_H^1(o_{J^1B}) \pitchfork Z
$$
with $Z = Z_{o_{J^1B}}$, which is equivalent to saying that
any Reeb chord from $\psi_H^1(o_{J^1B})$ to $o_{J^1B}$ is nondegenerate.
(We recall the case of  nondegeneracy of Reeb chords for the general pair $(R_0,R_1)$
summarized in Subsection \ref{subsec:reeb-chords}.)

\subsection{Boundary map moduli space}

We denote
$$
\Theta = \R \times [0,1].
$$
\begin{thm}\label{thm:decompose} Denote by $\widetilde \CM(H,J)$ the set of moduli space of solutions of
\eqref{eq:perturbed-contacton-bdy} with finite energy $E(u)< \infty$. Then we have
the decomposition
$$
\widetilde \CM(H,J) = \bigcup_{\gamma^\pm} \widetilde \CM (H,J; \gamma^-, \gamma^+).
$$
where $\widetilde \CM (H,J; \gamma^-, \gamma^+)$ is
the moduli space of solutions of \eqref{eq:perturbed-contacton-bdy} with asymptotic conditions
$$
\lim_{\tau \rightarrow \pm \infty} u(\tau,t) = \gamma^\pm(t)
$$
with  $\gamma^\pm \in \mathfrak X(o_{J^1B},o_{J^1B})$.
\end{thm}
\begin{proof}
This is a consequence of  combination of the results
\cite{oh:contacton-Legendrian-bdy}, \cite{oh:entanglement1} and \cite{oh:perturbed-contacton-bdy}.
More specifically, it follows from the results of asymptotic
convergence and of vanishing charge and the Gromov-Floer-Hofer style
compactification via the bubbling-off analysis similar to the study of
the moduli space of Hamiltonian-perturbed Floer equations.

The latter in turn relies on some energy estimates.
More specifically, we have the following energy identity.

\begin{prop}\label{eq:energy-identity}  Let $u$ satisfy
\eqref{eq:perturbed-contacton-bdy}, and set $w(\tau,t): = (\phi_H^t)^{-1}(u(\tau,t))$.
 Then
\be\label{eq:downward-action-flow}
\widetilde \CA_H(u(- \infty)) - \widetilde \CA_H(u(-\infty))
=- \int_{-\infty}^\infty \left| \left(\frac{\del w}{\del \tau}\right)^\pi \right|_{J^{\prime}}^2\, d\tau.
\ee
\end{prop}
\begin{proof} We recall the identity
$$
\CA_H(u(\tau)) = \CA(w(\tau))
$$
from Lemma \ref{lem:CAHu=CAw}. Therefore
\bea\label{eq:dCAHdtau}
\frac{d}{d \tau} \CA_H (u(\tau)) &=& \frac{d}{d \tau} \CA (w(\tau))\nonumber\\
& = & \delta \CA (w(\tau)) \bigg( \frac{\del w}{\del \tau} \bigg)
= \int_0^1 d\lambda \bigg( \frac{\del w}{\del \tau}, \frac{\del w}{\del t} \bigg) dt
\nonumber\\
&=& \int_0^1 d\lambda \bigg( \frac{\del w}{\del \tau} ^\pi, \frac{\del w}{\del t} ^\pi \bigg) dt
= \int_0^1 d\lambda \bigg( \frac{\del w}{\del \tau} ^\pi, J' \frac{\del w}{\del \tau} ^\pi \bigg) dt \nonumber \\
&=& \left| \left(\frac{\del w}{\del \tau}\right)^\pi \right|_{J^{\prime}}^2.
\eea
Since $u(\tau,1) \in o_{J^1B}$, note that we have
$$
\widetilde \CA _H (u(\tau)) = - \CA_H (u(\tau)) +z(u(\tau,1)) = - \CA_H (u(\tau))
$$
and hence
$$
\frac{d}{d \tau} \widetilde \CA _H (u(\tau)) = - \frac{d}{d \tau} \CA_H (u(\tau)).
$$
By integrating \eqref{eq:dCAHdtau} over $\R$ using the finiteness of
the $\pi$-energy, we have finished the proof.
\end{proof}

Furthermore we also need the uniform bound for the \emph{vertical energy of $u$} $E^\perp(u)$
whose precise definition is postponed till Section
\ref{subsec:Eperp-bound}
in the more general context of continuity maps. This finishes the proof of Theorem \ref{thm:decompose}.
\end{proof}

Recall that the asymptotic limit curves $\gamma^\pm$ appearing above have the form
\be\label{eq:gamma-+-}
\gamma^\pm(t) = \phi^t_H (\overline{\gamma}^\pm_{T_\pm} (t))
\ee
where
$$
(\overline{\gamma}^\pm, T_\pm) \in \mathfrak{Reeb}\left(\psi^1_H(o_{J^1B}), o_{J^1B}\right).
$$
We then denote
$$
\CM (H,J; \gamma^-, \gamma^+) := \widetilde \CM (H,J; \gamma^-, \gamma^+)/\R.
$$
Before the study of transversality for the general Hamiltonian $H = H(t,x)$ turned on, 
we first recall the case of  nondegeneracy of Reeb chords for the general pair $(R_0,R_1)$
summarized in Subsection \ref{subsec:reeb-chords}.

\begin{defn}\label{defn:off-shell}
We define the off-shell function space
$$
\CF := \CF(J^1B, H; o_{J^1B}; \gamma^-, \gamma^+)
$$
to be the set of smooth maps satisfying the boundary condition
$$
u(\tau, i) \in o_{J^1B} \qquad {\rm for} \quad i = 0,1
$$
and the asymptotic condition
$$
\lim_{\tau \rightarrow \pm \infty} u(\tau,t) = \gamma^\pm (t).
$$
\end{defn}
Now we are ready to state the generic transversality results that we need.
We will express transversality
statements in terms of the \emph{gauge-transformed} moduli spaces.

\subsection{Gauge transformation of the moduli space}

Consider the perturbed contact instanton equation
\eqref{eq:perturbed-contacton-bdy} and its associated moduli space
$$
\CM(J^1B, H; o_{J^1B}; \gamma^-, \gamma^+).
$$
After applying the gauge transformation $\Phi_H$ and considering 
\be\label{eq:choice-J'}
J'_1 = (\psi_H^1)^*J_0
\ee
as in \cite{oh:entanglement1}, we convert it to 
$$
\CM (J',\psi^1_H(o_{J^1B}),o_{J^1B}; \overline{\gamma}^-, \overline{\gamma}^+)
$$
that consist of solutions of \eqref{eq:contacton-bdy-J0}
for each associated Reeb chords $(\overline{\gamma^\pm}, T_\pm)$.
We assume that these Reeb chords are nondegenerate in the sense of
Subsection \ref{subsec:reeb-chords}.

We now recall the transversality result under the perturbation of CR almost
complex structures $J$ from \cite{oh:contacton-transversality} 
restricted to the current case of one-jet bundle.

For a given $H$, \emph{after a suitable $W^{k,p}$-completion} whose
details we omit and refer to \cite{oh:contacton-gluing}, we consider the universal section
$$
\Upsilon^{\rm univ} : \CF \times \CP(\CJ^c_g(J^1B)) \rightarrow
\Omega^{(0,1)} (w^*\xi) \oplus \Omega^2(\Theta)
$$
defined by
$$
\Upsilon^{\rm univ} (w, J')
= \left(\overline{\del}^\pi_{J'} w, d(w^*\lambda \circ j)\right)
$$
where $\Theta \cong \R \times [0,1]$.
Then we consider the universal moduli space
$$
{\CM}\left(J^1B, (\psi^1_H(o_{J^1B}),o_{J^1B}); \overline{\gamma}^-, \overline{\gamma}^+\right)
 =: (\Upsilon^{\rm univ})^{-1}(0).
$$
By considering the projection map
$$
\Pi_2 : \CF \times \CP(\CJ^c_g(J^1B) )\rightarrow \CP(\CJ^c_g(J^1B)),
$$
 we have
\beastar
&{}& \CM \left(J', (\psi^1_H(o_{J^1B}),o_{J^1B}); \overline{\gamma}^-, \overline{\gamma}^+\right)\\
& = & \Pi_2^{-1}(J') \cap
\CM\left(J^1B, \psi^1_H(o_{J^1B}),o_{J^1B}); \overline{\gamma}^-, \overline{\gamma}^+\right)
\eeastar
which is independent of $(k,p)$ with $k \geq 2$.

Now we have the following theorem.

\begin{thm}[Theorem 4.2, \cite{oh:contacton-gluing}]\label{thm:generic}
Let $0 < \ell < k - \frac{2}{p}$. Then
\begin{enumerate}
\item ${\CM}(\psi^1_H(o_{J^1B}),o_{J^1B}; \overline{\gamma}^-, \overline{\gamma}^+)$
is an infinite dimensional  manifold.
\item The projection
$$
 \Pi_2|_{(\Upsilon^{\rm univ})^{-1}(0)} : (\Upsilon^{\rm univ})^{-1}(0) \rightarrow \CP(\CJ^c_g(J^1B))
$$
is a Fredholm map (again with a suitable Banach completion mentioned above),
and its index is the same as that of $D\Upsilon(w)$ for any
$$
w \in {\CM}(J', (\psi^1_H(o_{J^1B}),o_{J^1B}; \overline{\gamma}^-, \overline{\gamma}^+).
$$
Here $D\Upsilon(w)$ will be defined in the next section.
\end{enumerate}
\end{thm}

An immediate corollary of Sard-Smale theorem is that for a generic choice of $J'$
$$
{\CM}(J', (\psi^1_H(o_{J^1B}),o_{J^1B}); \overline{\gamma}^-, \overline{\gamma}^+)
$$
is a smooth manifold, and hence $\widetilde{\CM}(H,J; \gamma^-, \gamma^+)$ is also a smooth manifold.
We denote by
$$
\CP^{\rm reg}_H(\CJ^c_g(J^1B))
$$
the set of regular values of the projection
$\Pi_2|_{(\Upsilon^{\rm univ}))^{-1}(0)}$.

\subsection{Chain map moduli space}

We also have to consider the \emph{nonautonomous version} of the
equation \eqref{eq:perturbed-contacton-bdy}, which enters in the construction of
chain maps.

Let $(J^\alpha, H^\alpha)$ and $(J^\beta, H^\beta)$ be two given generic nondegenerate pairs.
For each given path $(\{J^s\} , \{H^s\}) = \{ (J^s, H^s) \}_{s \in [0,1]}$ between them,
we take its elongated path
$$
(J^\rho, H^\rho) = \{ (J^{\rho(\tau)}, H^{\rho(\tau)}) \}_{ -\infty \leq \tau \leq \infty}
$$
and consider the non-autonomous versions of perturbed contact instanton equation
\eqref{eq:perturbed-contacton-bdy}.

We consider the 2-parameter family of contactomorphisms
$
\Psi_{s,t}: = \psi_{H^s}^t.
$
Obviously we have the $t$-developing Hamiltonian $\Dev_\lambda(t \mapsto \Psi_{s,t}) = H^s$.
We then consider the elongated two parameter family
$$
H^\rho(\tau,t,x) =  H^{\rho(\tau)}(t,x)
$$
and write the $\tau$-developing Hamiltonian
$$
G(\tau,t,x) = \Dev_\lambda(\tau \mapsto \Psi_{\tau,t}^\rho)
$$
where $\Psi_{\tau,t}^\rho = \Psi_{\rho(\tau),t}$.

Then we consider the nonautonomous perturbed contact instanton equation
\eqref{eq:perturbed-contacton-bdy-HG}.
After the gauge transformation, this nonautonomous equation becomes the following equation with
\emph{moving boundary condition}
\be\label{eq:rho-contacton-bdy}
\begin{cases}
\overline{\del}^\pi_{(J')^\rho} w
= \big( \frac{\del w}{\del \tau} + (J')^\rho \frac{\del w}{\del t} \big)^\pi = 0 \\
d(w^* \lambda \circ j) = 0 \\
w(\tau,0) \in \psi_{H^\rho} (o_{J^1B}), \quad w(\tau,1) \in o_{J^1B}.
\end{cases}
\ee

Again the finite-energy moduli space is decomposed into submoduli spaces with
specified asymptotics which consist of the pair of \emph{perturbed Hamiltonian chords}
$$
\gamma^\pm \in \mathfrak{X}(H;o_{J^1B},o_{J^1B}).
$$

Then we have the following theorems similar to the above two theorems.
Since the proofs are similar and standard, we omit the proofs.
\begin{thm}
\begin{enumerate}
\item For given $J \in \CP(\CJ^c_g(J^1B))$ and $H^\alpha, H^\beta \in \CH^{\rm reg}$
there exists a residual set $\CP^{\rm reg}_J(H^\alpha,H^\beta)$ of
$$
\mathcal{P}(H^\alpha, H^\beta) := \left\{\{ H^s \}_{0 \leq s \leq 1} \mid H^0 = H^\alpha \; H^1 = H^\beta \right\}
$$
such that all the solutions of \eqref{eq:perturbed-contacton-bdy-HG} with constant $J$ are regular.
\item For given $H \in C^\infty_0 (\R \times J^1B)$ and $J^\alpha, J^\beta \in \CJ^c_g(J^1B)$ there exists a residual set
$\mathcal{P}^{\rm reg}_H(J^\alpha, J^\beta)$ of $\mathcal{P}(J^\alpha, J^\beta)$,
which is defined similarly, such that
all the solutions of \eqref{eq:perturbed-contacton-bdy-HG} with constant $H$ are regular.
\end{enumerate}
\end{thm}
We remark that $G\equiv 0$ when $H^s$ does not depend on $s$ as in the case (2) above.

After gauge transformation, the same kind of transversality hold for the moduli space of
solutions of \eqref{eq:rho-contacton-bdy}
with asymptotic conditions
$$
\lim_{\tau \rightarrow -\infty} w(\tau) = \overline{\gamma}^\alpha, \quad
\lim_{\tau \rightarrow \infty} w(\tau) = \overline{\gamma}^\beta
$$
which we denote by
$$
\CM^\rho (\{J^{\prime s}\}; \overline{\gamma}^\alpha, \overline{\gamma}^\beta).
$$
Here we have
\beastar
(\overline{\gamma}^\alpha, T^\alpha) & \in & \mathfrak{Reeb}\left (\psi_{H^\alpha} (o_{J^1B}), o_{J^1B}\right)
\\
(\overline{\gamma}^\beta, T^\beta)  &  \in & \mathfrak{Reeb}\left (\psi_{H^\beta} (o_{J^1B}), o_{J^1B}\right).
\eeastar

\section{Fredholm theory and index formula}

Now we compute $\operatorname{Index} u = \operatorname{Index} w$ in terms of
explicitly defined Maslov indices by specializing the Fredholm theory laid out \cite{oh:contacton,oh:contacton-transversality}. The derivation of the formula for the linearization and its Fredholm theory
are given in \cite{oh:contacton} for the closed string case and \cite{oh:contacton-transversality}
for the open string case in general.
In this section, we will just consider the unperturbed case on the one-jet bundle
which will be enough for the purpose of the present paper after relevant gauge transformations. In this section, we always assume that the pair
$(\psi_H^1(o_{J^1B}),o_{J^1B})$ is nondegenerate with respect to 
the contact form $\lambda = dz - pdq$.

Consider the contact instanton operator
$$
\Upsilon (w) = \left(\overline{\partial}^\pi w, d(w^*\lambda\circ j)\right)
$$
which is a section of the vector bundle
$$
\CC\CD_{k-1,p} \rightarrow \overline{\CW}^{k,p}
$$
as in the preceding section. Here we adopt the notations:
\begin{itemize}
\item  $\CC\CD$ the vector bundle whose fiber at $w$
is given by
$$
\CC\CD_w: = \Omega^{(0,1)} (w^*\xi) \oplus \Omega^2(\Theta)
$$
following the notation from  \cite{oh:contacton-transversality} where `CD' stands for `codomain'. And $\CC\CD_{k-1,p}$ is the completion given by
$$
\CC\CD_{k-1,p} = \Omega^{(0,1)}_{k-1,p} (w^*\xi)
\oplus \Omega^2_{k-2,p}(\Theta).
$$
\item $\CW^{k,p}$ is the $W^{k,p}$-completion of $\CF$  given in
Definition \ref{defn:off-shell} and
$\overline W^{k,p}$ is the completion of $\overline{\CF}$, the gauge transformation of $\CF$.
\end{itemize}
(We refer interested readers to \cite{oh:contacton-transversality} for the
details of the generic transversality results which we use in the
present paper.)

We decompose $ \Upsilon = (\Upsilon_1, \Upsilon_2) $ where
$$
\Upsilon_1 : \overline{\CW}^{k,p} \rightarrow \Omega^{(0,1)}_{k-1,p}(w^*\xi); \quad
\Upsilon_1(w) = \overline{\del}^\pi w
$$
and
$$
\Upsilon_2 : \overline{\CW}^{k,p} \rightarrow \Omega^2_{k-2,p}(\R \times [0,1]); \quad
\Upsilon_2 (w) = d(w^*\lambda \circ j) $$
In this decomposition we have the linearization
$$ D\Upsilon (w) : \Omega^o_{k,p}(w^*T(J^1B);TR,To_{J^1B}) \rightarrow
\CC\CD_{k-1,p}
 $$
where $R = \psi_H^1(o_{J^1B})$.
In this decomposition we can express $D\Upsilon (w)$ as the matrix form (see \cite[(11.3)]{oh:contacton})
$$
\left(
\begin{array}{cc}
\overline{\del}^{\nabla^\pi}+B^{(0,1)}+T^{\pi,(0,1)}_{d w} & \frac{1}{2} \lambda(\cdot)(\mathcal{L}_{R_\lambda}J)J\partial^\pi w \\
d(((\cdot)\rfloor d\lambda)\circ j) & -\Delta(\lambda(\cdot))dA
\end{array} \right)
$$
Note that the desired index of $w$ is exactly $\operatorname{Index} D\Upsilon(w)$.

\begin{prop}[Proposition 11.2 \cite{oh:contacton} ]\label{prop:diagonalize}
The operator $D\Upsilon(w)$ is homotopic to the operator
$$
\left(
\begin{array}{cc}
\overline{\del}^{\nabla^\pi} + B^{(0,1)} + T^{\pi,(0,1)}_{d w} & 0 \\
0 & -\Delta (\lambda (\cdot)) dA
\end{array} \right) =: L_{w}
$$
via the homotopy
$$ s \in [0,1] \mapsto
\left(
\begin{array}{cc}
\overline{\del}^{\nabla^\pi} + B^{(0,1)} + T^{\pi,(0,1)}_{d w} & \frac{s}{2} \lambda (\cdot)(\mathcal{L}_{R_\lambda}J) J \del^\pi w \\
s d(((\cdot) \rfloor d\lambda) \circ j) & -\Delta (\lambda (\cdot)) dA
\end{array} \right) =: L_s
$$
which is a continuous family of operators. This family has the same Legendrian boundary conditions which is elliptic, so it is Fredholm for all $s$. Moreover, this family preserves the index. Therefore we have
$$
{\rm Index} D\Upsilon (w)
= {\rm Index} \left(\overline{\del}^{\nabla^\pi} + B^{(0,1)} + T^{\pi,(0,1)}_{d w}\right) + {\rm Index}(-\Delta). $$
\end{prop}

\begin{rem} Although it has not been explicitly mentioned in 
\cite{oh:contacton}, it is also important that the off-diagonal terms
exponentially decay in the strip-like coordinate $(\tau,t)$ near the punctures
as $|\tau| \to \infty$.
\end{rem}

It has been shown in \cite[Lemma 10.1]{oh-yso:index}
that  ${\rm Index}(-\Delta) = 0$.

Now we will provide an explicit formula for the index of the operator
$$
D\Upsilon_1(w)= \overline{\del}^{\nabla^\pi} + B^{(0,1)} + T^{\pi,(0,1)}_{d w}
$$
acted upon the $W^{k,p}$-completion $\Omega^o_{k,p}(w^*\xi; TR, To_{J^1B})$.
With the standard coordinates $(\tau,t) \in \R \times [0,1] \subset \R^2$, we have
$$
2\Upsilon_2 (w)(\del_\tau) = 2\overline{\del}^\pi (w)(\del_\tau)
= \left( \frac{\del w}{\del \tau} + J' \frac{\del w}{\del t} \right)^\pi.
$$
By an abuse of notation, we also denote  by $D\Upsilon_1$ the linearization of this $w^*T(J^1B)$-valued operator as usual.

To closely study the linearization $D\Upsilon_1$
we use the canonical trivialization $\Phi$ in
\eqref{eq:canonical-trivialization} again. Then we have the push forward operator
$$
\Phi_\ast D\Upsilon_1|_{w^*\xi} : W^{1,2}(\Sigma,\R^{2n};\Lambda^\Phi(\tau),\R^n)
\rightarrow L^2(\Theta,\R^{2n})
$$
where $\Lambda^\Phi(\tau) = \Phi(T_{w(\tau,0)}R)$ in $\R^{2n}$.
Then by a straightforward computation, one can check that this operator becomes an operator of Cauchy-Riemann type
\be\label{eq:trivialized-cr-operator}
\begin{cases}
\overline{\del}_{J',T}Y := \frac{\del Y}{\del \tau} + J' \frac{\del Y}{\del t} + TY \\
Y(\tau,0) \in \Lambda^\Phi(\tau), Y(\tau,1) \in \R^n
\end{cases}
\ee
for $Y:\Theta \rightarrow \R^{2n}$.

Let $J'_\pm:[0,1] \rightarrow {\rm End}(\R^{2n})$ be defined by the equation
$$
\lim_{\tau\rightarrow\pm\infty}\sup_{0\leq t\leq 1}\|J(\tau,t)-J_\pm(t)\| = 0
$$
and $T_\pm$ is defined similarly. Also let $\Psi_\pm:[0,1]\times Sp(2n)$ be defined by the equations
$$
\frac{\partial \Psi_\pm}{\partial t} - J'_\pm(t)T_\pm(t)\Psi_\pm = 0, \qquad \Psi_\pm(0) = id
$$
Now we can apply the theorem from \cite{rs:spectral}.

\begin{lem}[\cite{rs:spectral}]
The Fredholm operator $\bar{\partial}_{J',T}:W^{1,2}_\Phi \rightarrow L^2$ has the index given by
$$
{\rm Index}\bar{\partial}_{J',T} = -\mu(Gr(\Psi^-),\Lambda^-) + \mu(Gr(\Psi^+),\Lambda^+) + \mu(\Delta,\Lambda)
$$
where $\Lambda = \Lambda_0\oplus  \Lambda_1$, $\Lambda^\pm$ is an asymptotic limit of $\Lambda(\tau)$ and $\Delta$ is the diagonal in $\R^{2n}\oplus\R^{2n}$.
\end{lem}

In our case, we have
$$
\Psi_\pm(t) = \Phi\circ T\phi_{R_\lambda}^{T(1-t)}\circ \Phi^{-1} \equiv id, \quad
\Lambda_0 (\tau) = \Lambda^\Phi (\tau), \quad \Lambda_1 (\tau) \equiv \R^n.$$
Therefore
$$\mu(Gr(\Psi^\pm),\Lambda^\pm) = 0 $$
since $Gr(\Psi^\pm)$ and $\Lambda^\pm$ are both constant paths, and hence we have
$$
\operatorname{Index} w = \mu(\Delta,\Lambda^\Phi(\tau)\oplus\R^n) = \mu(\R^n,\Lambda^\Phi(\tau))
$$

In the classical Floer theory, the equivalence of CR-equation with boundary condition $\phi^1_H(o_{T^*B})$ and $o_{T^*B}$ and perturbed CR-equation with boundary condition $o_{T^*B}$ and $o_{T^*B}$ gives that
$$
\mu(\R^n,\Lambda^\Phi(\tau)) = -\mu(\Lambda^\Phi_K(-\infty,t),\R^n) + \mu(\Lambda^\Phi_K(+\infty,t),\R^n)
$$
where we obtain
$$
\Lambda^\Phi_K(\pm\infty,t) = \Phi\circ \psi^t_K\circ\Phi^{-1}(\R) = B_\Phi(\R^n)
$$
for $K : T^*B \rightarrow \R$. (See the proof of \cite[Theorem 5.1]{oh:jdg} for the right-hand-side index formula.) Similarly, in our case, we have
 $$
 \mu(\Lambda^\Phi(\pm\infty,t),\R^n) = -\mu(u(\pm\infty)) = -\mu(\gamma^{\pm}(t)).
 $$

\section{Canonical grading}

We recall the set of translated Hamiltonian chords
$$
\mathfrak{X}(H;o_{J^1B},o_{J^1B})
$$
consists of the paths of the form
\be\label{eq:translated-Ham-chords}
\gamma(t) = \psi_H^t (\overline \gamma(T t)), \quad t \in [0,1]
\ee
where $\overline \gamma$ is a Reeb chord
$\overline \gamma:[0,T] \to J^1B$ with period $T$.
In this regard, we may denote the same set by
$$
\Phi_H ( \mathfrak{Reeb} (\psi_H^1(o_{J^1B}), o_{J^1B}) ).
$$
By the assumption $\psi_H^1(o_{J^1B}) \pitchfork Z$ and compactness of $\psi_H^1(o_{J^1B})$
it is easy to check that there are only finitely many elements in
$\mathfrak{X}(H;o_{J^1B},o_{J^1B})$.

We now consider the perturbed contact instanton equation \eqref{eq:perturbed-contacton-bdy}.
The rest of this section will be occupied by the proofs of the following index formulae.

\begin{thm}[Compare with \cite{oh:jdg}]\label{thm:grading}
For each $\gamma\in \mathfrak{X}(H;o_{J^1B},o_{J^1B})$, there exists a canonically assigned
Maslov index that has the values in ${1\over 2}\Z$.  We
denote this map by
$$
\mu:\mathfrak{X}(H;o_{J^1B},o_{J^1B})\rightarrow \frac{1}{2}\Z.
$$
Furthermore, $\mu$ satisfies the following properties:
\begin{enumerate}
\item  For
each solution $u$ of \eqref{eq:perturbed-contacton-bdy}
 with $u(-\infty )=\gamma^-$,
$u(+\infty )=\gamma^+$, we have
the Fredholm index of $u$ given by
$$
\operatorname{Index} u=\mu(\gamma^- )-\mu(\gamma^+ ).
$$
\item
Consider the time-independent contact Hamiltonian $F=f \circ\pi_{J^1B}, \, f\in
C^\infty(B) $ where $f$ is a Morse function on $B$.
Let $y \in \widetilde{\operatorname{Graph}} (df)\cap o_{J^1B}$ and so
$b = \pi (y) \in \operatorname{Crit}(f)$.  Denote by
$\gamma_{b} \in \mathfrak{X}(H;o_{J^1B},o_{J^1B})$ with $\gamma_{b} (0) \in  o_{J^1B}$,
$\gamma_{b} (1) = b \in o_{J^1B}$.  Then we have
$$
\mu(\gamma_{b}) = \frac{1}{2}\dim B - \mu_f (b) = \frac{n}{2} - \mu_f (b)
$$
where $\widetilde{\operatorname{Graph}}(df) = \{ (b, df(b), f(b)) \in J^1B \mid b \in B\}$ and
$\mu_f$ is the Morse index of $f$ at $b$ on $B$.
\end{enumerate}
\end{thm}

\subsection{Canonical Lagrangian splittings of contact distribution $\xi$}

Similarly to \cite{oh:jdg}, we first construct a certain canonical class of symplectic trivializations
$$
\Phi:\gamma^*\xi \rightarrow [0,1]\times\C^n
$$
of the contact distribution $\xi \subset TJ^1B$ 
for each given translated Hamiltonian chord $\gamma\in\mathfrak{X}(H;o_{J^1B},o_{J^1B})$.

\subsubsection{Lagrangian splitting in canonical coordinates}

We first recall the splitting 
\be\label{eq:TJ1B}
T(J^1B) = \xi \oplus \R \langle R_\lambda \rangle 
= \xi \oplus \R \left \langle \frac{\del}{\del z} 
\right\rangle.
\ee
Moreover with respect to any canonical coordinates 
$(q_1, \ldots, q_n, p_1 \ldots, p_n)$
we also have the Lagrangian splitting
\be\label{eq:Lag-splitting}
\xi = \span \left\{\frac{\del}{\del p_i}\right\}_{1 \leq i \leq n} \oplus
\span \left \{\frac{D}{\del q_i} \right\}_{ 1 \leq i \leq n}
\ee
where $\frac{D}{\del q_i} = \frac{\del}{\del q_i} + p_i \frac{\del}{\del z}$. 
Recalling the natural diagram
$$
\xymatrix{& \ar[dl]_{ \pi_{\text{\rm cot}}}  J^1B  \ar[dr]^{\pi_{\text{\rm front}}
} & \\
T^*B &  & B \times \R,
}
$$
we have 
\beastar
\span \left\{\frac{\del}{\del p_i}\right\}_{1 \leq i \leq n} & = &\ker d\pi_{\text{\rm front}}
 \\
\span \left \{\frac{D}{\del q_i} \right\}_{ 1 \leq i \leq n} 
& = & (d    \pi_{\text{\rm cot}}|_\xi)^{-1} \span \left \{\frac{\del}{\del q_i} \right\}_{ 1 \leq i \leq n}.
\eeastar
Together with the splitting
\eqref{eq:TJ1B}, we obtain the splitting
$$
T(J^1B) = \xi \oplus \span \{ R_\lambda \}.
$$
We summarize the above discussion into the following.
\begin{prop} The splitting \eqref{eq:TJ1B} depends only on the choice of
contact form $\lambda = dz - pdq$ independent of the choice of canonical coordinates 
of $T^*B$.
\end{prop}
\subsubsection{Lagrangian splitting in the Sasaki metric of $J^1B$}

When we equip $B$ with a Riemannian metric $g$ in addition, it naturally induces the associated
Sasaki metric on $TB$ and $T^*B$ so that the naturally orthogonal splitting
$$
T(T^*B) = H_g \oplus V
$$
induced by the Levi-Civita connection of $g$. (See Appendix \ref{sec:JonT*B} and \ref{sec:sasaki-J} for some
basic facts on the Sasaki metric and almost complex structure on $T^*B$.)

It follows that the splitting, also written as
$$
\xi = H_g \oplus V
$$
with slight abuse of notation, is a Lagrangian splitting and bundle isomorphic to 
$\pi^*T(T^*B)$ via $(d    \pi_{\text{\rm cot}})|_\xi$ which preserves the Lagrangian splittings 
of $\xi$ and  of $T(T^*B) \cong H_g \oplus V$.

\begin{defn}[$\CT_g$ and $\CT$] Let $\gamma\in\mathfrak{X}(H;o_{J^1B},o_{J^1B})$, 
consider the class of symplectic trivializations
\be\label{eq:canonical-trivialization}
\Phi: \gamma^*\xi \rightarrow [0,1]\times\R^n\oplus(\R^n)^*\cong [0,1]\times\C^n
\ee
that satisfies
$$
\Phi(H_{g;\gamma(t)}) \equiv \R^n, \quad \Phi(V_{\gamma(t)}) \equiv (\R^n)^*\cong i\R^n
 $$
for all $t\in [0,1]$; we denote the class by $\CT_g$. We then consider the union
$$
\CT = \bigcup_{g \in \text{Riem}(B)} \CT_g.
$$
\end{defn}
It is easy to check that each $\CT_g$ is contractible and so is $\CT$.
Under any of  this trivialization, we will have
$$ 
\Phi(T_{\gamma(0)}o_{J^1B}) = \Phi(T_{\gamma(1)}o_{J^1B}) = \R^n 
$$
and the map $B_\Phi:[0,1] \rightarrow {\rm Sp}(2n)$ by
$$ 
B_\Phi(t) := \Phi\circ d\psi^t_H\circ\Phi^{-1}:\C^n\cong\{0\}\times\C^n \rightarrow \{t\}\times\C^n\cong\C^n
 $$
Following the definition of \cite{rs:maslovindex}, we now consider the Maslov index
$$ 
\mu({\rm Gr}(B_\Phi),\R^n\oplus\R^n) 
$$
which is the same as $\mu(B_\Phi(\R^n),\R^n)$ \cite{rs:maslovindex}.

\subsection{Calculation of Maslov index: Proof of Theorem \ref{thm:grading}}

Note that this Maslov index is independent of the choice of trivializations $\Phi\in\CT$
by the contractibility thereof .  (See \cite[Lemma 5.8]{oh:jdg} in the context of $T^*B$).

For the proof of Statement (1) of Theorem \ref{thm:grading}, we just mention that 
it is an immediate consequence of \cite[Theorem 2.4]{rs:maslovindex}, 
and so we will just prove the second statement. 

Recall that we have a 1-1 correspondence between the two moduli spaces
$$
\widetilde \CM (H,J;o_{J^1B}, o_{J^1B}), \quad
\widetilde \CM (J',(\psi^1_H (o_{J^1B})),o_{J^1B}))
$$ 
by the gauge transformation. We will consistently denote by $u$ an element  in $
 \widetilde \CM (H,J)$ and by $w$ the corresponding element
in $\widetilde \CM (J',(\psi^1_H (o_{J^1B})), o_{J^1B})$ in the following discussion.
 
Note that the contact Hamiltonian flow of the Hamiltonian $F = f\circ\pi_{J^1B}$ is just given by
the fiberwise translation
$$ 
\psi^t_F(q,p,z) = (q,p + t\, df(q), z + t\, f(q)) 
$$
\begin{lem}\label{lem:Reeb-chords-f}
Let  $\overline{\gamma}_{b}$  be the Reeb chord from $\psi^1_F(o_{J^1B})$ 
to $o_{J^1B}$ with its final point 
$$
\overline{\gamma} (1) = (b,0,0) \in o_{J^1B}.
$$
Then the corresponding $\gamma_{b} \in \mathfrak{X}(H;o_{J^1B},o_{J^1B})$ is of the form
\beastar
\gamma_{b}(t) & =  & \psi^t_F\circ(\psi^1_F)^{-1}\overline{\gamma}(t) \\
& = & (b, (t-1)df(b), f(b)(1-t) + f(b)(t-1)) = (b,0,0).
 \eeastar
In particular $\gamma_{b}$ is a constant path. 
\end{lem}
\begin{proof} By definition,  the Reeb chord $\overline{\gamma}_{b}$ must be of the form
$$ 
\overline{\gamma}_{b}(t) = (b, 0, (1-t)\, f(b)).
$$
Since its starting point $\overline\gamma_b(0)$ is contained in $\psi_H^1(o_{J^1B})$,
we also have
$$
(b, 0, (1-t)\, f(b)) = \psi_H^1(q_0,0) = (q_0, df(q_0), f(q_0)) 
$$
for some $q_0 \in B$. Therefore we obtain $q_0 = b, \, df(q_0) = 0, \, f(b) = f(q_0)$.
In particular $b = q_0$ is a critical point of $f$. This finishes the proof.
\end{proof} 

Now we examine the path of Lagrangian subspaces
$$
t \mapsto d_{\gamma_b(t)} \psi^t_F \left(d(\psi_F^1)^{-1}(To_{J^1B})\right)
$$
in the given trivialization of $\gamma^*\xi$. We compute
$$
(d\psi^t_F)_{(q,p,z)} =
\left(
\begin{array}{ccc}
I & 0 & 0 \\
td^2f(q) & I & 0 \\
tdf(q) & 0 & 1
\end{array} \right)
$$
Then we compute the restriction of $d\psi^t_F$ to $\xi$. In terms of the splitting
$$
\xi_{(q,p,z)} = H \oplus V,
$$
we have the matrix representation of $d\psi^t_F|_\xi$ given by
$$
d\psi^t_F|_\xi =
\left(
\begin{array}{cc}
I & 0 \\
t\, d^2f(q) & I
\end{array} \right)
$$
Since $f$ is a Morse function, $d^2f(q)$ is nonsingular at every $q \in {\rm Crit}(f)$
and then $d\psi^t_F\cdot\R^n \cap \R^n \neq 0$ if and only if $t=0$. Therefore
it follows from \cite{rs:maslovindex} that
\beastar
\mu(\gamma_{b}) &= &\mu(d\psi^t_F\cdot\R^n,\R^n) 
= \frac{1}{2}{\rm sign}\Gamma(d\psi^t_F\cdot\R^n,\R^n,0)
\\
& = &\frac{1}{2}{\rm sign}d^2f(b) = \frac{1}{2}\dim B - \mu_f(b)
\eeastar
This finishes the proof of Theorem \ref{thm:grading}.
\qed

Now we have the following corollary from the dimension formula Theorem \ref{thm:generic}.
\begin{cor}
The dimension of a manifold $\widetilde \CM (H,J; \gamma^-, \gamma^+)$ is the same as
$$
{\rm Index} D\Upsilon (w) = \mu(\gamma^-) - \mu(\gamma^+).
$$
\end{cor}

\begin{rem}
We can also compute ${\rm Index}D\Upsilon (w)$ by using the method of graded anchors introduced in \cite{oh-yso:index}.
More explicitly, for a contact instanton $w$ with Legendrian boundary conditions $(\psi_H^1(o_{J^1B}), o_{J^1B})$
we can canonically choose a graded anchored Legendrian pairs
$$
((\psi_H^1(o_{J^1B}), \ell_0, \alpha_0), (o_{J^1B}, \ell_1, \alpha_1))
$$
where we put $q_0: = w(0,0)$ and
$$
\begin{cases}
\ell_1 (t) \equiv q_0,\,  & \alpha_1 (t) \equiv T_{q_0} o_{J^1B}, \\
\ell_0 (t) = \psi^t_H (q_0),\,  & \alpha_0 (t) := d\psi^t_H (T_{q_0} o_{J^1B}).
\end{cases}
$$
Then $w$ is \emph{admissible} to the anchored Legendrian pair.
The admissibility is defined in \cite[Definition 8.2]{oh-yso:index}.
Then, by \cite[Theorem 10.3]{oh-yso:index},
$$
 {\rm Index}D\Upsilon (w) = n - \mu_{\rm anc}([w^+_{01}, \overline{\gamma}^+]; \alpha_{0})
- \mu_{\rm anc}([w^+_{10}, \overline{\gamma}^-]; \alpha_{0}).
$$
(Here we computed the index by putting the (positive) strip-like coordinate
$[0,\infty) \times [0,1]$.)
Moreover by following the computation and reduction in \cite[Appendix]{oh:cag} we can check that
$$ \mu_{\rm anc}([w^+_{01}, \overline{\gamma}^+]; \alpha_{0}) = \frac{n}{2} + \mu(\gamma^+) $$
and similarly
$$ \mu_{\rm anc}([w^+_{10}, \overline{\gamma}^-]; \alpha_{0}) = \frac{n}{2}  - \mu(\gamma^-) $$
and therefore we have the same result as Theorem \ref{thm:grading} (1).
\end{rem}

\part{Legendrian contact instanton cohomology and spectral invariants}

\section{Legendrian contact instanton cohomology of $H$}
\smallskip

With all the above preparation we are now ready to construct
the Legendrian Floer cohomology via (perturbed) contact instantons.
To avoid the orientation issue, we will use the $\Z_2$-coefficients.
From now on we consider only generic pairs $(H,J)$ as in the previous section. Moreover by the dimensional reason,
we may and will assume
\be\label{eq:cap-empty}
\psi_H^1(o_{J^1B}) \cap o_{J^1B} = \emptyset
\ee
by perturbing $H$ in a $C^\infty$-small way.

\subsection{Contact instanton complex and its boundary map}
\label{subsec:bdymap}

We form a $\frac{1}{2}\Z$-graded free $\Z_2$-module
$$
CI^*(H,J:B) = \Z_2 \{\mathfrak{X}(H;o_{J^1B},o_{J^1B})\}.
$$

Recall that for each $\gamma^-, \gamma^+ \in \mathfrak{X}(H;o_{J^1B},o_{J^1B})$
satisfying  $\mu(\gamma^+) - \mu(\gamma^-) = 1$,
$\widetilde \CM (H,J; \gamma^-, \gamma^+)$ is an 1-dimensional manifold
and hence the quotient
$$
\CM (H,J; \gamma^-, \gamma^+) := \widetilde \CM (H,J; \gamma^-, \gamma^+) / \R
$$
is a compact 0-dimensional manifold by the transversality condition $ \psi_H^1(o_{J^1B}) \pitchfork Z $.
We define
$$
n_{(H,J)} (\gamma^- , \gamma^+ ) := \#_{\Z_2}\left(\CM (\gamma^-, \gamma^+)\right)
$$
for such a pair $(\gamma^-, \gamma^+)$, and
a homomorphism
$$
\delta_{(H,J)} : CI^*(H,J: B) \rightarrow CI^*(H,J: B)
$$
given by
$$
\delta_{(H,J)} ( \gamma^+ ) = \sum_\beta n_{(H,J)} (\gamma^+ , \gamma^- ) \gamma^-.
$$
By definition, $\delta_{(H,J)}$ has degree $+1$ with
respect to the grading given above.

\begin{rem} We attract readers' attention that we put the input at $+\infty$ and the output
at $-\infty$ which may be considered as the \emph{cohomological convention} in our
convention of the sign put in the action functional. The upshot is the inequality
$$
\widetilde \CA_H(\text{\rm ``output''}) \geq
\widetilde \CA_H(\text{\rm ``input''}).
$$
\end{rem}

Now we prove that $\delta_{(H,J)}$ satisfies
$$
\delta_{(H,J)} \circ \delta_{(H,J)} = 0
$$
by using the property $ T_\lambda (J^1B; o_{J^1B}) = \infty $
as in \cite[Section 13]{oh:contacton-gluing} 
so that 
$$
(CI^*(H,J; B), \delta_{(H,J)})
$$
becomes a graded complex.
In this regard, we borrow the following theorem from \cite{oh:entanglement1}.

\begin{thm}[Theorem 1.3 \cite{oh:entanglement1}] \label{thm:bdymap}
Suppose $(M,\xi)$ is tame and $R \subset M$ is a compact Legendrian submanifold. Let $\lambda$ be a tame contact form such that
\begin{itemize}
\item $\psi = \psi^1_H$ and $ \| H \| < T_\lambda (M,R) $.
\item the pair $(\psi(R), R)$ is transversal in the sense that $\psi (R) \pitchfork Z_R$.
\end{itemize}
Let $J$ be a $\lambda$-adapted almost complex structure. Then
$$
\delta_{(H,J)} \circ \delta_{(H,J)} = 0.
$$
Furthermore for two different choices of such $J$ or of $H$, the complex are chain-homotopic to each other.
\end{thm}

Now we are ready to define the \emph{perturbed} contact instanton
cohomology associated to the cochain complex
$ (CI^\ast (H,J:B), \delta_{(H,J)} ) $.

\begin{defn} For each regular parameter $(H,J)$, we define
$$
HI^* (H,J;B)=\mbox{ Ker }\delta_{(H,J)} / \mbox{Im}\delta_{(H,J)}
$$
and call it the \emph{(perturbed) contact instanton Floer cohomology} of $(H,J)$ on $B$.
\end{defn}

Moreover, the following has been proven in \cite{oh:entanglement1}, \cite{oh:contacton-gluing}.

\begin{thm}[Theorem 10.6 \cite{oh:entanglement1} \& Corollary 11.11 \cite{oh:contacton-gluing}]
Consider the case $(M,R) = (J^1B,o_{J^1B})$. Then there is a natural 
isomorphism 
$ H^* (B,\Z_2)$ to $HI^*(o_{J^1B},\Z_2)$ induced by
the correspondence between $\text{\rm Crit}(f)$ and
$\frak{Reeb}(f\circ \pi;o_{J^1B}, o_{J^1B})$  appearing in Theorem \ref{thm:grading}.
\end{thm}

We would now like to continue this isomorphism to $HI^*(H,J;o_{J^1B})$ by
the contact instanton counterpart of Floer's continuation map. 
For later purpose, we will need to explicitly write the following two kinds of chain maps,
one over the change of $J$ and the other over that of $H$:
\begin{enumerate}
\item For a fixed $H$ and generic $J^\alpha$, $J^\beta$ in the \emph{perturbed}
sense of Section \ref{sec:moduli-space}, we define
$$
h_{\beta\alpha;\{J^s\}}: CI^\ast (H, J^\beta ; B)
\rightarrow CI^\ast (H,J^\alpha; B)
$$
\item For a fixed $J$ and generic $H^\alpha$ and $H^\beta$, we define
$$
h_{\beta\alpha;\{H^s\}}: CI^\ast (H^\beta,J ; B)
\rightarrow CI^\ast (H^\alpha,J; B)
$$
\end{enumerate}

The construction of the chain homotopy map is entirely analogous to that of the case of
Lagrangian Floer homology as done in \cite{oh:jdg}. Therefore we will be brief just by
indicating the modifications needed to handle the current case of contact instantons.

\subsection{Construction of the chain map over $\{J^s\}$}

In this subsection, we fix a nondegenerate $H$ and vary CR almost complex structures from
$J^\alpha$ to $J^\beta$ for a given $H$-generic $J^\alpha, \, J^\beta$ with
$J^\alpha = \{J^\alpha_t\}_{t \in [0,1]}$ and 
$J^\beta = \{J^\beta_t\}_{t \in [0,1]}$.
Denote by $J = \{J^s\}_{s \in [0,1]}$ a smooth path satisfying $J^0 = J^\alpha$ and
$J^1 = J^\beta$. Let $\rho:\R \to [0,1]$ the standard
elongation function of the type given satisfying
$$
\rho(\tau) = \begin{cases} 1 \quad & \tau \geq 1 \\
0 \quad & \tau \leq 0.
\end{cases}
$$
We denote by $J^\rho = \{J^{\rho(\tau)}\}_{\tau \in \R}$ the associated elongated family.

For given pairs of translated Hamiltonian chords
$$
\gamma^\alpha, \, \gamma^\beta \in  \mathfrak X(H; o_{J^1B}, o_{J^1B}),
$$
we consider the associated moduli space
$\CM (H,J^\rho; {\gamma}^\alpha, \gamma^\beta)$
of solutions for  \eqref{eq:rho-contacton-bdy}
of its virtual dimension 0 or 1.

Its dimension can be calculated by linearizing the equation similarly as in the case of
boundary maps.

\begin{prop} We have
$$
\dim \CM (H, J^\rho; \gamma^\alpha, \gamma^\beta)
= \mu(\overline \gamma^\alpha) - \mu(\overline \gamma^\beta).
$$
\end{prop}
\begin{proof}
We consider the \emph{nonautonomous} operator
$$
\Upsilon^\rho_{\{J^s\}}(u) = \left(\Upsilon_{(H,J^\rho\},1} (u),
\Upsilon_{(H,J^\rho),2} (u)\right).
$$
Note that the second component $\Upsilon_{H,2} (u)$ is independent of $\{J^s\}$.

Similar to the proof of Theorem \ref{thm:grading}, we take the gauge transformation $w$ of $u$
and its linearization$D\overline{\Upsilon}^\rho_{\{J^{\prime s}\}}(w)$ which is homotopic to the diagonal operator
$$
D\Upsilon^\rho_{\{J^{\prime s}\}}(w)(Y)
= \left(\overline{\del}^{\nabla^\pi}_{J^{\prime\rho}} + B^{(0,1)}_{J^{\prime\rho}}
+ T^{\pi,(0,1)}_{J^{\prime\rho}}.
-\Delta \right)
$$
We know that ${\rm Index}(-\Delta) = 0$ as in the proof of Theorem \ref{thm:grading} and
$D\Upsilon^\rho_1|_\xi$ is the parameterized CR-type equation with Lagrangian boundary condition. The
index of $D\Upsilon^\rho_{\{J^{\prime s}\}}(w)$ provides the dimension of
$\CM(J^{\prime \rho};\overline \gamma^\alpha,\overline \gamma^\beta)$
which is the same as $\CM(H,J^\rho; \gamma^\alpha,\gamma^\beta)$.
This finishes the proof.
\end{proof}

Now we define the homomorphism
$$
h_{\beta\alpha;J^\rho} : CI^* (H,J^\beta; B) \rightarrow CI^* (H,J^\alpha; B)
$$
by
$$
h_{\beta\alpha;J^\rho}(\gamma^\beta)
= \sum_{\gamma^\alpha;\mu_{J^\beta} (\gamma^\beta) = \mu_{J^\alpha} (\gamma^\alpha)} \#(\CM(J^{\prime \rho}; \overline{\gamma}^\alpha, \overline{\gamma}^\beta)) \gamma^\alpha.
$$
The chain map property of this map is proved in \cite{oh:entanglement1,oh:contacton-gluing}.
This finishes the construction of the chain map.

\subsection{Construction of chain map over $\{H^s\}$}

Similarly, for fixed $J$ and given generic $H^\alpha, H^\beta$, let
$$
\{H^s\}_{s \in [0,1]} \in  \mathcal{P}^{\text{\rm reg}}_J(H^\alpha, H^\beta)
\subset \mathcal{P}(H^\alpha, H^\beta),
$$
and let $\rho: \R \to [0,1]$ be an elongation function. Then we define
the chain map moduli space
$\CM (H^\rho,J; \gamma^\alpha, \gamma^\beta)$ as in \eqref{eq:rho-contacton-bdy} with fixed $J$.

Again we consider the elongated 2-parameter family of contactomorphisms
$$
\Psi_{\tau,t}^\rho: = \psi_{H^{\rho(\tau)}}^t
$$
and its $t$-developing Hamiltonian $\Dev_\lambda(t \mapsto \Psi_{(s,t)}) = H^s$
and the $\tau$-developing Hamiltonian
$$
G(\tau,t,x) = \Dev_\lambda(\tau \mapsto \Psi_{\rho(\tau),t}).
$$

Then we consider the 2-parameter perturbed contact instanton equation \eqref{eq:perturbed-contacton-bdy-HG}
and define the homomorphism $h_{\beta\alpha;\{H^s\}}$  by
$$
h_{\beta\alpha;\{H^s\}}(\gamma^\beta)
= \sum_{\gamma^\alpha} \#_{\Z_2}(\CM(H^\rho,J; \overline{\gamma}^\alpha, \overline{\gamma}^\beta)) \gamma^\alpha
$$
with $\mu_{H^\beta} (\gamma^\beta) = \mu_{H^\alpha} (\gamma^\alpha)$, which finishes the proof.

\section{Legendrian spectral invariants via contact instantons}
\label{sec:spectral}

Recall that the perturbed contact instanton Legendrian cohomology $HI^* (H,J;B)$ is 
isomorphic to the singular cohomology
$H^* (B;\Z_2)$. We denote this isomorphism by
$$
h_H^{\text{\rm PSS}} : H^* (B,\Z_2) \rightarrow HI^* (H,J;B).
$$
In this section, we carry out the mini-max theory of the
effective action functional $\widetilde \CA_H$ defined in \eqref{eq:tildeCAH}.
We first recall  that the perturbed contact instanton equation
\eqref{eq:perturbed-contacton-bdy}
 is a gradient-like flow of the perturbed action functional $\CA_H$ so that it preserves the downward filtration given by the values of $\widetilde \CA _H$.

\subsection{Definition of $\rho(H,J;a)$}

The following is an immediate corollary.

\begin{cor} Define $CI_c^* (H,J;B)$ to be
 the $\Z_2$-free module generated by
$\gamma \in \mathfrak{X} (H; o_{J^1B}, o_{J^1B})$ with $\widetilde \CA_H (\gamma) \geq c$. For each $c \in \R$,
we define $(CI_c^* (H,J;B), \delta_{(H,J)})$
forms a subcomplex of $(CI_c^* (H,J;B), \delta_{(H,J)})$.
\end{cor}

Then we define the quotient
$$
CI^*_{(c_1, c_2]} = CI^*_{c_1} / CI^*_{c_2}
$$
for $c_1 < c_2$. Then we have the short exact sequence
$$
0 \to CI^*_{c_2} \to CI^*_{c_1} \to CI^*_{(c_1, c_2]} \to 0.
$$
Then the differential $\delta_{(H,J)}$ induces the (relative) differential map
$$
\delta_{(H,J)} = \delta_{(H,J)}^{(c_1, c_2]} :CI^*_{(c_1, c_2 ]}(H,J;B) \rightarrow CI^*_{(c_1, c_2]}(H,J;B)
$$
for any $c_1 < c_2$. We define the relative cohomology groups by
$$
HI^*_{(c_1, c_2]}(H,J;B) := \mbox{Ker} \delta_{(H,J)}^{(c_1, c_2]} \Big\slash \mbox{Im} \delta_{(H,J)}^{(c_1, c_2]}.
$$
From the definition, there is a natural homomorphism
$$
j: HI^*_{(c_1, c_2]} \rightarrow HI^*_{(\mu_1, \mu_2]}
$$
when $c_1  \geq \mu_1$ and $c_2 \geq \mu_2$. In particular, there exists a natural homomorphism
\be\label{eq:j}
j_c : HI^*_{(c, \infty)} \rightarrow HI^*_{(-\infty, \infty)} = HI^*.
\ee

\begin{defn}
For given $0 \neq a \in H^*(B; \Z_2)$, we define the real number $\rho (H,J;a)$ by
$$
\rho (H,J;a) := \sup_c \left\{ c \in \R \mid h_H^{\text{\rm PSS}}(a) \in \Im  j_c \subset HI^* (H,J;B)\right \}.
$$
\end{defn}

For a generic $H$ in the sense that $\psi_H^1(o_{J^1B}) \pitchfork Z$,
it follows from the finiteness of the cardinality of $\mathfrak{X} (H; o_{J^1B}, o_{J^1B})$
that $ \rho (H,J;a) $ is well-defined, i.e., $\rho (H,J;a) \neq +\infty$.

The number $\rho (H,J;a)$ can be also realized as a mini-max value as follows. Recall that each chain
$\alpha \in CI^* (H,J;B)$ is a finite linear combination
$$
\alpha = \sum_{\gamma \in \mathfrak{X} (H;o_{J^1B},o_{J^1B})} a_\gamma [\gamma], \quad a_\gamma \in \Z_2.
$$
We define the support of $a$ by
$$
{\rm supp}\, \alpha = \{ \gamma \in \mathfrak{X} (H;o_{J^1B},o_{J^1B}) \mid a_\gamma \neq 0 \}.
$$
A cocycle $\alpha$ is a chain with $\delta_{(H,J)} (\alpha) = 0$.

Now for each cocycle $\alpha$, we define its level by
$$
\ell_H(\alpha) := \min_{\gamma \in {\rm supp}\, \alpha} \{ \widetilde \CA_H (\gamma)\}. $$
Then by definition, we have
$$
\rho (H,J;a) = \sup_{\alpha; [\alpha] = h_H^{\text{\rm pss}} (a)} \{\ell_H (\alpha)\}.
$$
From this point of view we can easily check that $\rho (H,J;a)$ is a critical value of $ \widetilde \CA_H$.

\subsection{Independence of $\rho(H,J;a)$ on $J$}

Next, we will remove the $J$-dependence of $\rho (H,J;a)$
for fixed $H$ when $J$ varies among $\CP^{\rm reg}_H(\CJ^c_g(J^1B))$.

\begin{lem} \label{lem:J-inv}
Let $J^\alpha, J^\beta \in \CP^{\rm reg}_H(\CJ^c_g(J^1(B))$.
Then we have
$$
\rho (H,J^\alpha; B; a) = \rho(H,J^\beta; B; a).
$$
\end{lem}

\begin{proof}
Recall that we can choose a regular path
$\{J ^s\} \in \CP^{\rm reg}(J^\alpha, J^\beta)$
in the \emph{perturbed} sense of Section \ref{sec:moduli-space}, and recall that we have constructed the homomorphism
$$
h_{\beta\alpha;\{J^s\}} : CI^*(H,J^\beta) \rightarrow CI^*(H, J^\alpha)
$$
defined by
$$
h_{\beta\alpha;\{J^s\}} (\gamma^\beta)
= \sum \#_{\Z_2} (\CM(J^{\prime\rho}; \overline{\gamma}^\alpha, \overline{\gamma}^\beta)) \gamma^\alpha
$$
This is an isomorphism defined in Theorem \ref{thm:bdymap}.

To see how $\rho(H,J; B; a)$ vary under the change of $J$, we need to estimate the difference
$ \widetilde \CA_H(\gamma^\beta) - \widetilde \CA_H(\gamma^\alpha) $
whenever $\#_{\Z_2} (\CM^\rho (\overline{\gamma}^\alpha, \overline{\gamma}^\beta)) \neq 0$.
Let $w$ be an element in $\CM (J^{\prime\rho}; \overline{\gamma}^\alpha, \overline{\gamma}^\beta)$ and $u$ its gauge transformation. We write
$$ \widetilde \CA_H (\gamma^\beta) - \widetilde \CA_H (\gamma^\alpha)
= \int_{-\infty}^{\infty} \frac{d}{d\tau} \widetilde \CA_H (u(\tau)) d\tau.
$$
We also have
$$
\frac{d}{d\tau} \widetilde \CA_H (u(\tau))
= -  \left|\left( \frac{\del w}{\del \tau}\right)^\pi
\right|^2_{J^{\prime\rho}} \leq 0
$$
with the exactly same calculation as \eqref{eq:downward-action-flow}. Hence we have proved that
$$ \widetilde \CA_H (\gamma^\beta) \leq \widetilde \CA_H (\gamma^\alpha). $$
This shows that the map $h_{\beta\alpha;\{J^s\}} : CI^* (H,J^\beta; B) \rightarrow CI^* (H,J^\alpha; B)$
restricts to a map
$$
h_{\beta\alpha;\{J^s\}} : CI^*_{(c,\infty)} (H,J^\beta; B) \rightarrow CI^*_{(c,\infty)} (H,J^\alpha; B)
$$
for any $c \in \R$ and so induces a homomorphism
$$
(h_{\beta\alpha;\{J^s\}})_* : HI^*_{(c,\infty)} (H,J^\beta; B) \rightarrow HI^*_{(c,\infty)} (H,J^\alpha; B).
$$
Now consider the commutative diagram
$$
\xymatrix{
HI^* _{(c,\infty)} (H,J^\beta; B) \ar[r] \ar[d] &  HI^* (H,J^\beta; B) \ar[d] \\
HI^* _{(c,\infty)} (H,J^\alpha; B) \ar[r] & HI^* (H,J^\alpha; B )}
$$
where all downward arrows are induced by the
canonical homomorphisms $h_{\beta\alpha;\{J^s\}}$ and the horizontal
ones by the canonical inclusion-induced map $j_c$.

Since $h_{\beta\alpha;\{J^s\}}$ on the right hand side is an isomorphism, if
$[a]^\beta: =  h_{(H,J^\beta)}^{\text{\rm pss}} (a) \in \operatorname{Im} j_c^\beta$, so is
$[a]^\alpha \in \operatorname{Im} j_c^\alpha$.
Therefore, we have proved
$$
\rho(H, J_\alpha; B; a) \geq \rho(H, J_\beta; B; a).
$$
By changing the role of $\alpha$ and $\beta$, we also obtain
$$
\rho(H, J_\beta; B; a) \geq \rho(H,J_\alpha; B; a)
$$
which finishes the proof of $\rho(H,J_\alpha; B; a) = \rho(H,J_\beta; B; a)$.
\end{proof}

\begin{defn}
For generic $H$, we define
$$ \rho(H;a) := \rho(H,J;a) $$
for some $J \in \CP^{\rm reg}_H(\CJ^c_g(J^1B))$.
\end{defn}

\section{Basic properties of the assignment $H \mapsto \rho(H;a)$}

Now we study the dependence of $\rho(H;a)$ on $H$.
The proof of the following theorem will occupy the whole section.

\begin{thm}\label{thm:spectral-invariants}
For any nondegenerate $H^\alpha, \, H^\beta$,
the following properties hold:
\begin{enumerate}
\item for $ a \in H^* (B) $ we have
\be \label{eq:general-H-estimate}
\int_0^1 \min_{y} (H^\beta - H^\alpha) dt \leq \rho(H^\beta; a) - \rho(H^\alpha; a)
\leq \int_0^1 \max_{y} (H^\beta - H^\alpha) dt.
\ee
In particular, for $H \in \CH^{\rm reg}$ we have
\be \label{eq:H-estimate}
\int_0^1 \min_{y} H dt \leq \rho(H; a) \leq \int_0^1 \max_{y} H dt
\ee
\item We have
$$
| \rho(H^\beta; a) - \rho(H^\alpha; a) | \leq \osc (H^\beta - H^\alpha)
$$
where
$$
\osc (H^\beta - H^\alpha) = \max_{y} (H^\beta - H^\alpha) - \min_{y} (H^\beta - H^\alpha).
$$
In particular, one can extend the assignment $H \mapsto \rho(H; a)$ to all $\CH := C^\infty_0 (\R \times J^1B; \R)$ as a continuous function in the $C^0$-topology of $\CH$. We will continue to denote the extension by $\rho(H; a)$.
\end{enumerate}
\end{thm}

The proof of (2) immediately follows from \eqref{eq:general-H-estimate}
and so we have only to prove (1).

\begin{rem}[Choice of homotopy $\{H^s\}$]
For the proof of a similar inequality for the Lagrangian
spectral invariants constructed in \cite{oh:jdg}, the linear homotopy
\be\label{eq:linear-homotopy}
s \mapsto (1-s) H^\alpha + s H^\beta
\ee
was used. However this homotopy cannot be used for the purpose of
proving the inequality \eqref{eq:general-H-estimate}:
\emph{Common calculation
used in symplectic Floer theory from \cite{chekanov:dmj}, \cite{oh:jdg}
will inevitably give rise to some
conformal factor in estimate in the current contact case.}
It turns out that the correct homotopy
to be used is \emph{the one through the zero Hamiltonian}.
\end{rem}

\subsection{Choice of the homotopy for the continuity map}

For this purpose, we consider the following type of elongation function $\chi : \R \rightarrow [0,1]$ satisfying
\begin{enumerate}
\item
\be\label{eq:chi}
\chi(\tau) = \begin{cases} 1 \quad &\text{\rm for }\,  |\tau | \geq 2\\
0 \quad & \text{\rm for }\, |\tau| \leq 1
\end{cases}
\ee
\item $ \chi'(\tau) \leq 0 $ when $-2 \leq \tau \leq -1$ and $ \chi'(\tau) \geq 0$ when $1 \leq \tau \leq 2$.
\end{enumerate}

Now we consider a 2-parameter family of contactomorphisms
$$
\psi_{\tau,t} :=
\left\{
\begin{array}{ll}
\psi_{H^\alpha}^{\chi(\tau) t} & {\rm for} \quad \tau \leq 0 \\
\psi_{H^\beta}^{\chi(\tau) t} & {\rm for} \quad \tau \geq 0
\end{array}
\right.
$$
and denote the corresponding homotopy between $H^\alpha$ and $H^\beta$ by
$$
H^\chi (\tau,t,y) = H^{\chi(\tau)}(t,y)
$$
for $\tau \in \R$, $t \in [0,1]$, and $y \in J^1B$.

We postpone the proof of the following uniform bound for $\pi$-energy till the next section.

\begin{thm}[Uniform $\pi$-energy bound]\label{thm:Epi-bound} Assume that $H$ is compactly supported.
Let $\gamma^\alpha, \, \gamma^\beta \in \mathfrak{X}(J^1B,H;o_{J^1B}, o_{J^1B})$ and
$u \in \widetilde \CM(H^\chi,J;\gamma^\alpha,\gamma^\beta)$. Then we have
$$
E_H^\pi(u)  \leq  \widetilde \CA_{H^\alpha}(\gamma^\alpha) - \widetilde \CA_{H^\beta}(\gamma^\beta) + \int^{1}_{0} \max_{y}(H^{\beta}_{t}(y)- H^{\alpha}_{t}(y)) \, dt.
$$
\end{thm}

We will also prove the bound for the $\lambda$-energy as well in the next section. 
Recall the definition \eqref{eq:||H||} of quantities $E^\pm(H)$ and $\|H\|$ used below.

\begin{thm}[Uniform vertical energy bound]\label{thm:Eperp-bound}
Let $u$ be any finite energy solution of \eqref{eq:H-para-contacton}. Then we have
\be\label{eq:Eperp-bound}
E_H^\perp(u) \leq |\widetilde \CA_{H^\alpha}(\gamma^\alpha)| + |\widetilde \CA_{H^\beta}(\gamma^\beta)| + E^+(H^\beta) + E^-(H^\alpha)
\ee
\end{thm}

In particular, \eqref{eq:Epi-bound} implies the inequality
$$
\widetilde \CA_H(\gamma^\beta) \leq  \widetilde \CA_H(\gamma^\alpha) + \int_0^1 \max_{y} (H^\beta - H^\alpha) dt.
$$
Once this is achieved, by the same mini-max argument as the one used in \cite{oh:jdg}, we obtain
$$
\rho(H^\beta; a) - \rho(H^\alpha; a) \leq \int_0^1 \max_{y} (H^\beta - H^\alpha) dt.
$$
By changing the role of $\alpha$ and $\beta$, we also have
$$
\rho(H^\beta; a) - \rho(H^\alpha; a) \geq \int_0^1 \min_{y} (H^\beta - H^\alpha) dt.
$$

Now \eqref{eq:general-H-estimate} enables us to continuously extend to arbitrary Hamiltonian $H \in \CH$,
not just in $\CH^{\rm reg}$.
More precisely, we choose any $C^\infty$-approximation
$H_i \in \CH^{\rm reg}$ of $H$ and then define
$$
\rho(H;a) := \lim_{i \to \infty} \rho(H_i ; a)
$$
which does not depend on the choice of $C^\infty$ approximation.

\subsection{Spectrality of $\rho(H;a)$ and $C^0$-bounds thereof}

The remaining is to obtain the inequality \eqref{eq:H-estimate}.
For this purpose we need the following spectrality property. Recall the definition of the functional
$$
\widetilde{\CA}_H: \CL(J^1B,(o_{J^1B},o_{J^1B})) \to \R.
$$

\begin{prop}[Spectrality]\label{prop:spectrality}
For any $ H \in \CH $, the value $\rho(H;a)$
 is a critical value of $\widetilde{\CA}_H$, i.e.,
 it lies in $\Spec(H; o_{J^1B}, o_{J^1B}): = \Spec(\lambda;\psi_H^1(o_{J^1B}),o_{J^1B}) $.
\end{prop}

\begin{proof}
For each $H \in \CH$, let $H_i$ be a $C^\infty$ approximation of $H$ such that $\supp X_{H_i} \subset D^r (J^1B) $ for large $r > 0$ which is independent of $i$. We have shown that $\rho(H;a)$ is a finite value. It remains to prove that
$$
\rho(H;a) = \widetilde \CA_H (\gamma)
$$
for some $\gamma \in \mathfrak{X} (H; o_{J^1B}, o_{J^1B})$. By definition, we have
$$\rho(H;a) = \lim_{i \to \infty} \rho(H_i ; a).$$
Note that
$$\rho(H_i ; a) = \widetilde \CA_{H_i} (\gamma_i)$$
for some path $\gamma_i : [0,1] \to J^1B \in \mathfrak{X} (H_i ; o_{J^1B}, o_{J^1B})$,
or equivalently,
$$
\overline \gamma _i \in \mathfrak{Reeb} (\psi_{H_i} (o_{J^1B}), o_{J^1B}).
$$
Since $H_i \to H$, we have $\psi^1_{H_i} \to \psi^1_H$. Moreover since $\supp X_{H_i} \subset D^r (J^1B)$, we have
$$ | \dot {\overline \gamma} _i (t) | \leq C $$
for some $C > 0$ independent of $i$ and $t$ so that $\overline \gamma _i$ are equi-continuous.
On the other hand	the boundary condition $\overline \gamma _i (1) \in o_{J^1B}$ and $o_{J^1B}$ is compact. Then there exists a subsequence, still denoted by $\overline \gamma_i $, converging to a smooth path $\overline \gamma$ lying in
$\mathfrak{Reeb} (\psi_H^1(o_{J^1B}), o_{J^1B})$. Note that $\gamma \in \mathfrak{X} (H; o_{J^1B})$ via the gauge transformation. Therefore we have
$$ \rho (H; a) = \lim_{i \to \infty} \rho (H_i; a) = \lim_{i \to \infty} \widetilde \CA_{H_i} (\gamma _i)
= - \lim_{i \to \infty} \CA_0 (\overline \gamma _i) = - \CA_0 (\overline \gamma)
= \widetilde \CA_H (\gamma) $$
which finishes the proof.
\end{proof}

From this spectrality and $\mathfrak{Reeb} (o_{J^1B}, o_{J^1B}) = \emptyset$,
all elements of $\mathfrak{X} (0;o_{J^1B}, o_{J^1B})$ consisting of constant paths.
By the same approximation argument $H_i \to 0$ in $C^\infty$-topology utilizing the \eqref{eq:H-estimate},
we can easily check that
$$
\rho(0;a) = 0
$$
for every $a \in H^*(B)$.
Applying $H ^\beta = H$ and $H ^\alpha = 0$, we have the inequality \eqref{eq:H-estimate}.

\section{Energy estimate for the continuity map}
\label{sec:energy-estimate}

We take a (parametrically) generic $\{H^s\}$ in the sense that the parameterized moduli space
$\widetilde \CM (J,\{H^s\})$ is regular. We follow the calculation performed in the proof of
\cite[Proposition 10.2]{oh:entanglement1} below.

We first consider the gauge transformed paths
$$
\overline \gamma ^\alpha := (\phi^t_{H^\alpha})^{-1} (\gamma^\alpha) \quad
\overline \gamma ^\beta := (\phi^t_{H^\beta})^{-1} (\gamma^\beta).
$$
We consider  \eqref{eq:perturbed-contacton-bdy-HG} with $\rho$ replaced by $\chi$
and then its the gauge transformation $\Psi_{H^\chi}$
\be \label{eq:H-para-contacton}
\begin{cases}
\overline \del ^\pi w = 0, \quad
d(w ^* \lambda \circ j) = 0 \\
w(\tau,0) \in \psi_{\chi(\tau),1} (o_{J^1B}), \quad w(\tau,1) \in o_{J^1B} \\
\lim_{\tau \rightarrow - \infty} w(\tau,t) = \overline \gamma ^\alpha (T_\alpha t),
\quad \lim_{\tau \rightarrow \infty} w(\tau,t) = \overline \gamma ^\beta (T_\beta t).
\end{cases}
\ee
Then we have
$$
\widetilde \CA_{H^\beta}(\gamma^\beta) - \widetilde \CA_{H^\alpha}(\gamma^\alpha)
= - \int_{-\infty}^\infty \frac{d}{d\tau} \CA (w(\tau)) d\tau
$$
and
$$
\CA (w(\tau)) := \int_0^1 (w(\tau))^*\lambda
$$
where $w : \R \times [0,1] \rightarrow J^1B$ is a solution of \eqref{eq:H-para-contacton}.

\subsection{A priori uniform $\pi$-energy bound}

We now prove the following uniform bound for $\pi$-energy.

\begin{thm}[Uniform $\pi$-energy bound] Assume that $H$ is compactly supported.
Let
$$
\gamma^\alpha \in\mathfrak{X}(J^1B,H^\alpha;o_{J^1B}, o_{J^1B}), \quad \, \gamma^\beta \in \mathfrak{X}(J^1B,H^\beta;o_{J^1B}, o_{J^1B})
$$
and
$u \in \widetilde \CM(H^\chi,J;\gamma^\alpha,\gamma^\beta)$. Then we have
\be\label{eq:Epi-bound}
E^\pi(u) \leq \widetilde \CA_{H^\alpha}(\gamma^\alpha) - \widetilde \CA_{H^\beta}(\gamma^\beta) + \int^{1}_{0} \max_{y}(H^{\beta}_{t} - H^{\alpha}_{t}) dt
\ee
\end{thm}
\begin{proof}
We compute
\bea \label{eq:general-estimate}
- \frac{d}{d\tau} \CA (w(\tau))
&=& - \frac{d}{d\tau} \int_{[0,1]} w^* \lambda
= - \delta \CA_0 (w (\tau)) \cdot \frac{\del w}{\del \tau} \nonumber \\
&=& - \int_0^1 d \lambda \left(\frac{\del w}{\del \tau}, \frac{\del w}{\del t}\right)\, dt
- \lambda \left(\frac{\del w}{\del \tau}(\tau, 1)\right)
+ \lambda \left(\frac{\del w}{\del \tau} (\tau, 0)\right) \nonumber \\
& = & - \int_0^1 \left| \frac{\del w}{\del \tau} ^\pi\right|_{J^{\prime\chi}}^2 \, dt
+ \lambda \left(\frac{\del w}{\del \tau} (\tau, 0)\right)
\eea
where the last equality follows from
\beastar
d \lambda \left(\frac{\del w}{\del \tau}, \frac{\del w}{\del t}\right)
& = & \Big| \frac{\del w}{\del \tau} ^\pi \Big|_{J^{\prime\chi}}^2 \geq 0, \\
\frac{\del w}{\del \tau} (\tau, 1) & \in& To_{J^1B} \subset \xi.
\eeastar

Since $w (\tau, 0) \in \psi_{\chi(\tau),1} (o_{J^1B})$, we write
$$
w (\tau, 0) = \psi_{\tau,1}(q(\tau),0,0)
$$
for some $q(\tau) \in B$. Then
$$
\frac {\del w}{\del \tau} (\tau, 0) = \frac{\del \psi_{\tau,1}}{\del \tau} (q(\tau),0,0)
+ (d\psi_{\tau,1})_{(q(\tau),0,0)} \left(\frac{\del q(\tau)}{\del \tau}\right).
$$
On the other hand, from the moving boundary condition, we derive
\be\label{eq:dpsidtau}
\frac{\del \psi_{\tau, t}}{\del \tau} (y)
=
\left\{
\begin{array}{ll}
\chi'(\tau)t\, X^{\chi(\tau)t}_{H^\alpha} \left(\psi^{\chi(\tau)t}_{H^\alpha} (y)\right) & {\rm for} \quad \tau \leq 0 \\
\chi'(\tau)t\, X^{\chi(\tau)t}_{H^\beta} \left(\psi^{\chi(\tau)t}_{H^\beta} (y)\right) & {\rm for} \quad \tau \leq 0 \\
\end{array}
\right.
\ee
Therefore we obtain
\be \label{eq:s-Ham}
\lambda \left(\frac{\del w}{\del \tau} (\tau, 0)\right)
=
\left\{
\begin{array}{ll}
\chi'(\tau) H^{\alpha}_{\chi(\tau)} (w(\tau,0)) & {\rm for} \quad \tau \leq 0 \\
\chi'(\tau) H^{\beta}_{\chi(\tau)} (w(\tau,0)) & {\rm for} \quad \tau \leq 0
\end{array}
\right.
\ee
Here we use again the fact that
$$
(d\psi_{\tau,1})_{(q(\tau),0,0)}\left( \frac{\del q(\tau)}{\del \tau}\right) \in \xi.
$$
Using this, we now prove the following.

\begin{lem}\label{lem:lambdadwdtau}
\be\label{eq:lambdadwdtau}
\lambda \left(\frac{\del w}{\del \tau} (\tau, 0)\right)
\leq \int^{1}_{0} \max_{y}(H^{\beta}_{t} - H^{\alpha}_{t}) dt.
\ee
\end{lem}
\begin{proof} Using the inequalities
\beastar
\chi'(\tau) &\leq&  0  \quad \text{\rm for }\, \tau \leq 0,\\
\chi'(\tau) & \geq & 0 \quad \text{\rm for } \, \tau \geq 0
\eeastar
we compute
\beastar
\lambda \left(\frac{\del w}{\del \tau} (\tau, 0)\right)
& = &  \int^{0}_{-\infty} \chi'(\tau) H^{\alpha}_{\chi(\tau)} (w(\tau,0)) d\tau
+ \int^{+\infty}_{0} \chi'(\tau) H^{\beta}_{\chi(\tau)} (w(\tau,0)) d\tau \\
& \leq & \int^{0}_{-\infty} \chi'(\tau) \min_{y}(H^{\alpha}_{\chi(\tau)}) d\tau
+ \int^{+\infty}_{0} \chi'(\tau) \max_{y} (H^{\beta}_{\chi(\tau)}) d\tau \\
&=& - \int^{1}_{0} \min_{y}(H^{\alpha}_{t}) dt + \int^{1}_{0} \max_{y} (H^{\beta}_{t}) dt \\
&=& \int^{1}_{0} \max_{y}(H^{\beta}_{t} - H^{\alpha}_{t}) dt \\
\eeastar
\end{proof}

We now integrate \eqref{eq:general-estimate} over $-\infty <\tau < \infty$ and then substitute
\eqref{eq:lambdadwdtau} thereinto. This finishes the proof of \eqref{eq:Epi-bound}.
\end{proof}

\subsection{A priori uniform bound for vertical energy}
\label{subsec:Eperp-bound}

We also prove the bound for the $\lambda$-energy as well. For this purpose, we
introduce the following quantities
\bea\label{eq:||H||}
E^+(H) & := &  \int^{1}_{0} \max_{y} H_{t}(y)\, dt \nonumber\\
E^-(H) & := & \int^1_0 - \min_y  H_{t}(y)\,  dt
\eea
similarly as in the symplectic geometry.

The proof of the following is essentially the same as that of
\cite[Proposition 13.1]{oh:entanglement1}. Since the present setting is somewhat different
therefrom, we provide the full details of its proof for readers' convenience and for
the self-containedness of the paper.

\begin{thm}[Uniform vertical energy bound]
Let $u$ be any finite energy solution of \eqref{eq:H-para-contacton}. Then we have
$$
E_H^\perp(u) \leq |\widetilde \CA_{H^\alpha}(\gamma^\alpha)| + |\widetilde \CA_{H^\beta}(\gamma^\beta)| + E^+(H^\beta) + E^-(H^\alpha)
$$
\end{thm}
\begin{proof}
By the defining equation $d(w^*\lambda \circ j) = 0$ of contact instantons and the vanishing of charge,
we have a globally defined function $f: \R \times [0,1] \to \R$ in \eqref{eq:f-defn} such that
$$
w^*\lambda \circ j = df.
$$

By definition of $E^\perp$, we need to get a uniform bound for the integral
$$
\int_{\R \times [0,1]} (-w^*\lambda) \wedge d(\psi(f)) \geq 0.
$$
(See Definition \ref{defn:CC-energy}.) By integration by parts, we rewrite
$$
\int_{\R \times [0,1]} (-w^*\lambda) \wedge d(\psi(f)) =
\int_{\R \times [0,1]} d(\psi(f) w^*\lambda) - \psi(f)  dw^*\lambda.
$$
Recall that  $dw^*\lambda = \frac12 |d^\pi w|^2$ for any
$w$ satisfying $\delbar^\pi w = 0$. Then similarly as we prove
Theorem \ref{thm:Epi-bound}, we derive
\beastar
0 & \leq & \int_{\R \times [0,1]} (-w^*\lambda) \wedge d(\psi(f))
\leq \int_{\R \times [0,1]} d(\psi(f) w^*\lambda)\\
& = & \int_{\{-\infty\} \times [0,1] } \psi(f(-\infty,t)) ({\overline \gamma^-})^*\lambda
- \int_{\{\infty\} \times [0,1] } \psi(f(\infty,t)) ({\overline \gamma^+})^*\lambda\\
&{}& + \int_{-\infty}^{\infty} \psi(f(\tau,0)) \lambda\left(\frac{\del w}{\del \tau}(\tau,0) \right)\, d\tau
 - \int_{0}^{\infty} \psi(f(\tau,1)) \lambda\left(\frac{\del w}{\del \tau}(\tau,1) \right)\, d\tau
\eeastar
By the charge vanishing Theorem \ref{thm:charge-vanishing}, we
have the asymptotic convergence of  $w_\tau^*\lambda \to T^\pm \, dt$
as $|\tau| \to \infty$,
where we put
$$
 T^+: = \CA_{H^\beta}(\gamma^+), \quad T^- := \CA_{H^\alpha}(\gamma^-).
$$
 Then since $0 \leq \psi \leq 1$, we obtain
\beastar
\left|\int_0^1 \psi(f(-\infty,t)) ({\overline \gamma^-})^*\lambda \right| & \leq&  |T^-|,\\
\left|\int_0^1 \psi(f(\infty,t)) ({\overline \gamma^+})^*\lambda \right| & \leq & |T^+|.
\eeastar
On the other hand, we have
$$
\lambda\left(\frac{\del w}{\del \tau}(\tau,1)\right) = 0
$$
since the $\tau$-developing Hamiltonian $G$ vanishes since $\phi_H^1 = \psi_H^1 (\psi_H^1)^{-1} = id$ and hence $w(\tau, 1) \in o_{J^1B}$ which is
a Legendrian submanifold.
This proves
\beastar
&{}& \int_{\R \times [0,1]} (-w^*\lambda) \wedge d(\psi(f))\\
& \leq & |T^-| + |T^+| + \int_{0}^{\infty} \psi(f(\tau,0)) \lambda\left(\frac{\del w}{\del \tau}(\tau,0) \right)\, d\tau \\
& \leq &  |T^-| + |T^+| + \int_0^1 - \min \psi(f(\cdot,t)) H_t)\, dt + \int_0^1 \max (\psi(f(\cdot,t)) H_t) \, dt \\
& \leq & |T^-| + |T^+| + \int_0^1 (\max  H_t^\beta - \min H_t^\alpha) \, dt \\
&\leq & |T^-| + |T^+| + E^+(H^\beta) + E^-(H^\alpha).
\eeastar
Here for the penultimate inequality, we employ the following:
\begin{itemize}
\item
We use the same calculations as the ones performed in the proof of
Theorem \ref{thm:Epi-bound}, and apply Theorem \ref{thm:Epi-bound}.
\item Moreover, we also have used the inequality
$$
\chi'
\begin{cases}
 \leq 0 \quad & \text{\rm for }\, \tau \in (-\infty,0]\\
\geq 0 \quad & \text{\rm for }\, \tau \in [0,\infty).
\end{cases}
$$
\end{itemize}
Then for the last equality, we use the fact $0 \leq \psi(f) \leq 1$.
Combining all the above discussion, we have finished the proof of
$$
\int_{\R \times [0,1]} (-w^*\lambda) \wedge d(\psi(f)) \leq |T^-| + |T^+| + E^+(H^\beta) + E^-(H^\alpha).
$$
for any $\psi \in \CC$ and hence the proof of the proposition by definition of
$E^\perp$.
\end{proof}

\appendix

\section{Legendrian spectral invariants of $\GFQI$}
\label{sec:gfqi-spectral-invariants}

In this appendix we review basic results on the generating functions of
Legendrian submanifolds and their spectral invariants, or the Viterbo-type
invariants.

Let $\pi_E: E \to B$ be a vector bundle and $S: E \to \R$ be a
function that is quadratic at infinity, abbreviated as \GFQI

We define the subset of $\mathbb{R}$,
\be\label{eq:spectrum}
\Spec(S) = \{S(e) \in \R \mid dS(e) = 0 \}
\ee
and call it the \emph{spectrum} of $S$.
Since $B$ is assumed to be compact and $S$ is quadratic at infinity
$\Spec(S) \subset \R$ is a compact subset of measure zero in general.

We consider the sub-level set
$$
E^c = \{e \in E \mid S(e) \leq c\}
$$
for $ c \in \R \cup \{\infty\}$ and by $E^{-\infty}$ the set $E^{-c}$ for
a sufficiently large $\lambda$. (The pair $(E,E^{-c})$ is homotopy equivalent to $(E,E^{-\infty})$
for any sufficiently large $c$.) Let $i_c: (E^c,E^{-\infty}) \to (E, E^{-\infty})$
be the inclusion map, and the induced map on cohomology
$$
i_c^*: H^*(B) \cong H^*(E, E^{-\infty}) \to H^*(E^c, E^{-\infty}).
$$
We define
$$
c(a;S): = \inf\{c \in \R \mid i_c^*(a) \neq 0\}
$$
for each $a \neq 0 \in H^*(B)$.

Denote by $Q_0$ a generic unspecified fiberwise quadratic form on $E$.
The following lemma is essentially proved by
Viterbo \cite{viterbo} who considered the symplectic case, and  extended to the
contact case by Th\'eret \cite{theret}.

\begin{thm}\label{thm:spec-S} Assume that $B$ is a closed manifold. Let $S: E \to \R$ be \GFQI.
The map $(a,S) \mapsto c(a;S)$ satisfies the following:
\begin{enumerate}
\item (Spectrality)  $c(a;S) \in \Crit S$ for all $a \neq H^*(B)$.
\item ($C^0$ continuity)  Suppose that $S_1, \, S_2: E \to \R$ be
\GFQI's such that $S_1 \equiv S_2$ outside a compact subset $K \subset E$.
Then if $\|S_1-S_2\|_{C^0} \leq \epsilon$, then
$$
|c(a;S_1) - c(a;S_2)| \leq \epsilon.
$$
\item For any $a, \, b \in H^*(B)$, $c(a\cup b,S_1 + S_2) \geq c(a;S_1) + c(b,S_2)$.
\item Let $S: E \to \R$ be a $\GFQI$and $\overline S: E \to \R$ be
the $\GFQI$ defined by $\overline S(q,e) = -S(q,-e)$. Then
$$
c((\mu, \overline S) = - c(1;S)
$$
where $\mu = PD[pt] \in H^n(B)$ is the orientation class, i.e., the Poincar\'e dual to the point class.
\item  $c(1;S) \leq 0$.
\end{enumerate}
\end{thm}
\section{Review of compatible almost complex structures on $T^*B$}
\label{sec:JonT*B}

In this section we recall that if a Riemannian metric $g$ is given to
$B$, the associated Levi-Civita connection induces a natural
almost complex structure on $T^*B$ called the Sasakian almost
complex structure, which we denote by $J_g$.

It is well-known and easy to check that this canonical
almost complex structure has the following properties:
\begin{prop} We have
\begin{enumerate}
\item $J_g$ is compatible to the canonical symplectic
structure $\omega_0$ of $T^*B$.
\item On the zero section $o_B\subset T^*B \cong T_{(q,0)}o_B$, $J_g$
assigns to each $v\in T_qB \subset T_{(q,0)}(T^*B)$ the
cotangent vector $J_g(v)=g(v,\cdot )\in T^*_qB\subset
T_{(q,0)}(T^*B)$.  Here we use the canonical splitting
$$
T_{(q,0)}(T^*B)\cong T_q B \oplus T^*_q B.
$$
\item The metric $g_{J_g} := \omega_0(\cdot, J_g \cdot)$ on $T^*B$
defines a Riemannian metric that has bounded curvature and
injectivity radius bounded away from 0.
\item $J_g$ is invariant under the anti-symplectic reflection $\frak r: T^*B \to T^*B$
mapping $(q,p) \mapsto (q,-p)$.
\end{enumerate}
\end{prop}

We consider the class of compatible almost complex
structures $J$ on $T^*B$ such that
$$
J \equiv J_g\;\mbox{ outside a compact set in $T^*B$,}
$$
and denote the class by
\beastar
\JJ^c_g(T^*B) &:= &\{J\mid\, J\, \mbox{is
compatible to } \, \omega \, \mbox{and }\,  J\equiv J_g \\
&{}& \quad \mbox{ outside a compact subset in } \,
T^*B \}.
\eeastar
We define and denote the \emph{support} of $J$ by
$$
\operatorname{supp }J:=\mbox{ the closure of } \, \{ x\in T^*B\mid J(x)\not = J_g(x)\}.
$$
We then consider
$$
\CP(\JJ^c(T^*B)) : = C^\infty([0,1] \to \JJ^c(T^*B)).
$$
For each given $J = \{J_t\}_{0 \leq t \leq 1}$, we consider the associated family of compatible metrics
$g_{J_t}$. This family induces an $L^2$-metric on the space of paths on $T^*B$ defined by
\be\label{eq:L2metric}
\ll \xi_1, \xi_2 \gg_J = \int_0^1 g_{J_t}(\xi_1(t), \xi_2(t))\, dt
= \int_0^1\omega (\xi_1(t), J_t \,\xi_2(t))\, dt.
\ee
We denote by $\widetilde g: = g_{J_g}$ the induced metric on $T^*B$.

We can express such a lifted CR-almost complex structure $J$ on $J^1B$
in terms of the coordinate $w = (u,f)$ where $u = \pi \circ w$ and $f = z \circ w$
as follows.

Recall the general decomposition
$$
dw = d^\pi w + w^*\lambda \, R_\lambda.
$$
In the current case of one jet bundle with $\lambda = dz - \pi_{\text{\rm cot}}
^*\theta$,
we can express $dw$ also as
$$
dw = Du + df
$$
in terms of the expression $w = (u,f)$: we have $d^\pi w = Du$ where
$Du$ is the horizontal lift of $du$ for the projection $\xi \to T(T^*B)$.
More specifically we have
$$
Du = (du)^\sharp: T\Sigma \to \xi
$$
is the horizontal lifting of $du \in \Omega^1(u^*T(T^*B))$ to one in $\Omega^1(u^*\xi)$
which induced by the map
$$
\frac{\del}{\del q_i} \to \frac{\del}{\del q_i} + q_i \frac{\del}{\del z} =: \frac{D}{\del q_i},
\quad \frac{\del}{\del p_i} \to  \frac{\del}{\del p_i}.
$$
Let $J$ be a $T^*B$-lifted CR almost complex structure on $J^1B$ and consider the case $\dot \Sigma = \R \times [0,1]$.
Then we have the decomposition
$$
Du = (Du)^{(1,0)} + (Du)^{(0,1)}
$$
with the complex linear and the anti-complex linear part of $Du:(T\Sigma, j) \to(\xi,J)$.
By definition, we have
$$
\delbar^\pi w = (Du)^{(0,1)}, \quad \del^\pi w = (Du)^{(1,0)}.
$$

\section{Sasakian almost complex structure: Proof of Lemma \ref{lem:d|p|2}}
\label{sec:sasaki-J}

Let $g$ be a Riemannian metric of $B$, and consider its dual metric on $T^*B$
which we denote by  $h$.
We then consider the induced kinetic energy Hamiltonian function $K:T^*B\to\R$;
$$
K(\alpha)=\frac{1}{2}|\alpha|^2_h.
$$
Its associated Hamiltonian vector field
$X_{K}$ is defined to satisfy $\omega_{0}(X_{K},\ \cdot\ )=d{K}$, and the flow 
of $X_{K}$ recovers the geodesic flow on the cotangent bundle.

Let us start with Levi-Civita connection $\nabla=\nabla^g$ and an induced (co-)frame fields 
$H_i,V_i$ (and $H^i,V^i$) on $T^*B$ given as follows:
\beastar
H_i &= & \partial_{q^i}+p_a\Gamma^a_{ij}\partial_{p_j}, \quad V_i= \partial_{p_i},\\
H^i &=  & dq^i, \quad V^i=  dp_i-p_a\Gamma^a_{ij}dq^j.
\eeastar
Here $\Gamma^a_{ij}$ are Christoffel symbols for the connection $\nabla$ and we used the Einstein summation convention.

In Riemannian geometry, they are commonly denoted by
$$
H_i = \frac{D}{\del q^i}, \quad V^i = \nabla p_i
$$
with respect to the splitting $T(T^*B) = H \oplus B \cong TB \oplus T^*B$. We will 
also adopt this notation which facillates the tensor calculations below.

An induced Riemannian metric $\widetilde h$ on $T^*B$ with respect to the (co-)frame fields is given by
$$
h_{ij}dq^idq^j+h^{ij} dp_i dp_j
$$
where $(h^{ij})_{i,j}$ is the inverse matrix of $(h_{ij})_{i,j}$ and $\delta p_i=V^i$.
In a matrix form we have
\begin{displaymath}
\left(\begin{array}{c|c}
h_{ij} & 0 \\
\hline
0& h^{ij}
\end{array}\right).
\end{displaymath}

The canonical symplectic 2-form on $T^*B$ is given by
$$
\omega=\sum_{i=1}^n dq^i\wedge dp_i = \sum_{i=1}^n H^i\wedge  V^i.
$$
The so called \emph{Sasakian almost
complex structure $J_h$} associated to the Levi-Civita connections of $h$ is given as follows.
First the Levi-Civita connection induces the splitting
$$
T_{(q,p)}(T^*B) = H_{(q,p)} \oplus V_{(q,p)} \simeq T_q B \oplus T_q^*B
$$
at each point $(q,p) \in T^*B$, where the isomorphism is obtained by
\be\label{eq:split-iso}
H_j \mapsto \frac{\del}{\del q^j}, \qquad V_j \mapsto dq^j.
\ee

In canonical coordinates, the almost complex structure $J_h:T(T^*B) \to T(T^*B)$
is given by the formulae
\beastar
H_i \mapsto  h_{ij}V^j, \quad
V^i \mapsto -h^{ij}H_j,
\eeastar
which can be expressed in the following matrix
\begin{displaymath}
\left(\begin{array}{c|c}
0 & -h^{ij} \\
\hline h_{ij} & 0
\end{array}\right)
\end{displaymath}
with respect to the above frame fields.
Then the compatibility condition
$$
\widetilde h(\cdot,\cdot)=\omega_0(\cdot,J_h\cdot)
$$
between the triple $(\widetilde h,\omega_0,J_h)$ can be guaranteed by the following matrix multiplication:
\begin{displaymath}
\left(\begin{array}{c|c}
h_{ij} & 0 \\
\hline
0& h^{ij}
\end{array}\right)
=
\left(\begin{array}{c|c}
0 & \delta_{ij} \\
\hline
-\delta_{ij}& 0
\end{array}\right)\cdot
\left(\begin{array}{c|c}
0 & -h^{ij} \\
\hline
h_{ij} & 0
\end{array}\right)
\end{displaymath}

With this preparation, a direct, but somewhat tedious computation, shows the following
identity which is equivalent to Lemma \ref{lem:d|p|2}.
\begin{prop} We have
$$
dK\circ J =-  \sum_{i=1}^n p_i dq^i = - \theta.
$$
\end{prop}
\begin{proof} We utilize Einstein's summation convention below.
By definition, we have
$$
K(\alpha) = \frac12 \langle \alpha, \alpha \rangle_h = \frac12 h^{ij} p_i p_j
$$
when $\alpha = p_i dq^i$ in the canonical coordinates $(q^1,\ldots, q^n,p_1, \ldots, p_n)$.

Then we compute
$$
dK = \frac12 d\langle \alpha, \alpha \rangle_h = \frac12 d(h^{ij} p_ip_j) = 
h^{ij} p_i \nabla p_j  
$$
where $\nabla p_j = V^j$ is the covariant differential. Therefore we have
$$
dK \circ J = h^{ij} p_i \nabla p_j \circ J 
= - h^{ij} p_i (h_{aj} H^a) = - \delta_a^i p_i H^a = - p_i dq^i = - \theta.
$$
\end{proof}


\begin{thebibliography}{GHMSS}

\bibitem[A]{abouzaid} Abouzaid, M., {\em A cotangent fibre generates the Fukaya category}, Adv. Math. 228 (2011), no. 2, 894--939.

\bibitem[AJ]{akhao-joyce} Akhao M., Joyce, D., {\em Immersed Lagrangian
Floer theory}, J. Differ. Geom. 86(3),  (2010), 381--500.

\bibitem[AOS]{amorim-oh-santos} Amorim, L., Oh, Y.-G.,
dos Santos, Joana Oliveira, {\em Exact Lagrangian submanifolds,
Lagrangian spectral invariants and Aubry-Mather theory},
 Math. Proc. Cambridge Philos. Soc. 165 (2018), no. 3, 411--434.

\bibitem[BKO]{BKO} Bae, Youngjin; Kim, Seonhwa; Oh, Y.G., {\em Formality of Floer complex of the ideal boundary
of hyperbolic knot complement}, Asian J. Math. 25 (2021), no. 1, 117--175.

\bibitem[B]{bhupal} Bhupal, M., {\em A partial order on the group of contactomorphisms of $\R^{2n+1}$ via
generating functions}, Turkish J. Math. \textbf{25} (2001), 125--235.

\bibitem[BCT]{BCT} Bravetti, A., Cruz, H., Tapias, D., {\em Contact Hamiltonian mechanics},
Ann. Physics 376 (2017), 17--39.

\bibitem[Che]{chekanov:dmj} Chekanov, Yu. V., {\em Lagrangian intersections, symplectic energy, and areas of holomorphic curves}, Duke
Math. J. 95 (1998), 213--226.

\bibitem[EHS]{EHS-gafa}  Eliashberg, Y.,  Hofer, H.,  Salamon, D.,
{\em Lagrangian intersections in contact geometry},
 Geom. Funct. Anal. 5 (1995), no. 2, 244--269.
 
\bibitem[Fl1]{floer:Morse} Floer, A., {\em  Morse theory for Lagrangian intersections}, J. Differ. Geom. 28 (1988), 513--547.

\bibitem[Fl2]{floer:Ham} \bysame, {\em Symplectic fixed points and holomorphic spheres}, Commun. Math. Phys. 120 (1989), 575--614.

\bibitem[Fl3]{floer:Witten} \bysame, {\em Witten's complex and infinte-dimensional Morse theory}, J. Differ. Geom.
30 (1989), 207 -- 221.

\bibitem[FSS1]{fukaya-seidel-smith1} Fukaya, K., Seidel, P.,  Smith, I., {\em Exact Lagrangian submanifolds in simply-connected cotangent bundles},
 Invent. Math. 172 (2008), no. 1, 1--27.

\bibitem[FSS2]{fukaya-seidel-smith2} Fukaya, K., Seidel, P., Smith, I., {\em The symplectic
geometry of cotangent bundles from a categorical viewpoint. Homological mirror symmetry}, 1--26, Lecture Notes in Phys., 757, Springer,
Berlin, 2009.

 \bibitem[GHS]{GHS} Cristofaro-Gardiner, D., Humili\`e re, V.,
Seyfaddini, S., {\em  Proof of the simplicity conjecture}, preprint,
 arXiv:2001.01792.

\bibitem[GHMSS]{GHMSS}
Cristofaro-Gardiner, D., Humil\`ere, V.,  Mak, Cheuk Yu, Seyfaddini, S.,
Smith, I., {\em Quantitative Heegaard Floer cohomology and the Calabi invariant},  Forum Math. Pi 10 (2022), Paper No. e27.

\bibitem[G]{gromov:pseudo} Gromov, M.,
{\em Pseudo-holomorphic curves in symplectic manifolds},
Invent. Math. 82 (1985), 307--347.

\bibitem[H]{hofer:symplectization} Hofer, H., {\em Pseudoholomorphic curves in symplectizations with applications to the Weinstein conjecture in dimension three}, Invent. Math. 114 (1993), no. 3, 515--563.

\bibitem[LO]{lim-oh} Lim, Jin-Wook, Oh, Y.-G., {\em
Nonequilibrium thermodynamics as a symplecto-contact reduction and relative information entropy}, 
Rep. Math. Physics, (to appear),  arXiv:2209.10660.

\bibitem[M1]{milinkovic1} Milinkovi\'c, D., {\em On equivalence of two constructions of invariants of Lagrangian
submanifolds}, Pacific J. Math. 195 (2000), no. 2, 371--415.

 \bibitem[M2]{milinkovic2} \bysame,
 {\em Action spectrum and Hofer's distance between Lagrangian submanifolds}, Differential Geom. Appl. 17 (2002), no. 1, 69--81.

\bibitem[M3]{milinkovi3} \bysame
{\em Geodesics on the space of Lagrangian submanifolds in cotangent bundles}, Proc. Amer. Math. Soc. 129 (2001), no. 6, 1843--1851.

 \bibitem[MVZ]{MVZ} Monzner, A., Vichery, N., Zapolsky, F.,
 {\em Partial quasimorphisms and quasistates on cotangent bundles,
and symplectic homogenization}, J. Mod. Dyn. 6 (2012), no. 2, 205--249.

\bibitem[Mr]{mrugala} Mrugala, R., {\em Geometrical formulation of equilibrium
phenomenological thermodynamics}, Reports on Math. Phys. 14 (1978), no 3.
419--427.

\bibitem[MNSS]{MNSS} Mrugala, R., Nulton, J. D., Sch\"on, J. C.,
Salamon, P., {\em Statistical approach to the geometric structure of thermodynamics}, Phys. Rev. A (3) 41 (1990), no. 6, 3156--3160.

\bibitem[MS1]{mueller-spaeth2} M\"uller, S., Spaeth, P., {\em Topological contact dynamics II: topological automorphisms, contact homeomorphisms, and non-smooth contact dynamical systems}, Trans. Amer. Math. Soc. 366 (2014), no. 9, 5009--5041.

\bibitem[MS2]{mueller-spaeth1} M\"uller, S., Spaeth, P., {\em Topological contact dynamics I: symplectization and applications of the energy-capacity inequality}, Adv. Geom. 15 (2015), no. 3, 349--380.

\bibitem[MS3]{mueller-spaeth3} M\"uller, S., Spaeth, P., {\em Topological contact dynamics III: uniqueness of the topological Hamiltonian and C0-rigidity of the geodesic flow}, J. Symplectic Geom. 14 (2016), no. 1, 1--29.

\bibitem[N]{nadler} Nadler, D., {\em Microlocal branes are constructible sheaves}, Selecta Math. (N.S.) 15 (2009), no. 4, 563--619.

\bibitem[NZ]{nadler-zaslow}
Nadler, D., Zaslow, E., {\em Constructible sheaves and the Fukaya category}, J. Amer. Math. Soc. 22 (2009), no. 1, 233--286.

\bibitem[Oh1]{oh:jdg} Oh, Y.-G., {\em Symplectic topology as the geometry of action functional. I. Relative Floer theory on the cotangent bundle}, J. Differential Geom. 46 (1997), no. 3, 499--577.

\bibitem[Oh2]{oh:cag} \bysame, {\em Symplectic topology as the geometry of action functional. II. Pants product and cohomological invariants}, Comm. Anal. Geom. 7 (1999), no. 1, 1--54.

\bibitem[Oh3]{oh:book1} \bysame, Symplectic Topology and Floer homology, vol 1, New Mathematical Monographs 29,
Cambridge University Press, 2015, Cambridge.

\bibitem[Oh4]{oh:book2} \bysame, Symplectic Topology and Floer homology, vol 2, New Mathematical Monographs 29,
Cambridge University Press, 2015, Cambridge.

\bibitem[Oh5]{oh:lag-capacity} \bysame, {\em Hamiltonian $C^0$-continuity
of Lagrangian capacity on the cotangent bundle},  Kyoto J. Math. 57 (2017), no. 3, 613--636.

\bibitem[Oh6]{oh:contacton} \bysame, {\em Analysis of contact Cauchy-Riemann maps III: energy, bubbling and Fredholm theory}, Bulletin of Math. Sciences, published on line,
https://doi.org/10.1142/S1664360722500114.

\bibitem[Oh7]{oh:contacton-Legendrian-bdy} \bysame, {\em Contact Hamiltonian dynamics and perturbed contact instantons
with Legendrian boundary condition}, preprint, 2021.

\bibitem[Oh8]{oh:entanglement1} \bysame, {\em Geometry and analysis of contact instantons and entanglement of Legendrian links I}, preprint, 2021, arXiv:2111.02597.

\bibitem[Oh9]{oh:contacton-gluing} \bysame, {\em Gluing theories of contact instantons and
pseudoholomorphic curves in SFT}, preprint 2022, arXiv:2205.00370.

\bibitem[Oh10]{oh:perturbed-contacton-bdy}
\bysame, \emph{Geometric analysis of perturbed contact instantons with
  {L}egendrian boundary conditions}, preprint, arXiv:2205.12351.

\bibitem[Oh11]{oh:contacton-transversality}
\bysame, \emph{Bordered contact instantons and their {F}redholm theory and
  generic transversalities}, preprint, submitted for the proceedigns of Bumsig
  Kim's Memorial Conference, October 2021, KIAS, arXiv:2209.03548(v2).

  \bibitem[Oh12]{oh:shelukhin-conjecture} \bysame
  \emph{Contact instantons, anti-contact involution
  and proof of Shelukhin's conjecture}, preprint,  arXiv:2212.03557.

 \bibitem[Oh13]{oh:entanglement2} \bysame, {\em Geometry and analysis of contact instantons and entanglement of Legendrian links I},
 in preparation.

 \bibitem[OM]{oh:hameo1} Oh, Y.-G., M\"uller S., {\em The group of
Hamiltonian homeomorphisms and $C^0$-symplectic topology},
J. Symp. Geom.  (2007), 167--219.

 \bibitem[OW1]{oh-wang1} Oh, Y.-G; Wang, R., {\em Analysis of contact Cauchy-Riemann maps I:
a priori $C^k$ estimates and asymptotic convergence}, Osaka J. Math. 55 (2018), no. 4, 647--679.

\bibitem[OW2]{oh-wang2} \bysame, {\em Analysis of contact Cauchy-Riemann maps II: Canonical neighborhoods
and exponential convergence for the Morse-Bott case}, Nagoya Math. J. 231 (2018), 128--223.

\bibitem[OY1]{oh-yso:index} Oh. Y.-G., Yu S.,  {\em Contact instantons with Legendrian boundary condition: A priori estimates, asymptotic convergence and index formula}, submitted.

\bibitem[OY2]{oh-yso:weinstein} \bysame,
{\em Contact action functional, calculus of variation and
  canonical generating function of {L}egendrian submanifolds}, preprint 2023, 
  arXiv:2311.18510(v2).

\bibitem[OY3]{oh-yso:equivalence} \bysame, in preparation.

\bibitem[On]{ono:lag-leg}  Ono, Kaoru, {\em Lagrangian intersection under Legendrian deformations},  Duke Math. J. 85 (1996), no. 1, 209--225.

\bibitem[PS]{polterov-shelukhin}
Polterovich, L., Shelukhin, E., {\em Lagrangian configurations and Hamiltonian maps},  arXiv:2102.06118.

\bibitem[RS1]{rs:maslovindex} Robbin, J., Salamon, D., {\em The Maslov index for paths},
Topology 32 (1993), no. 4, 827--844.
58F05 (58E05).

\bibitem[RS2]{rs:spectral} \bysame, {\em The spectral flow and the Maslov index}, Bull. London Math. Soc. 27 (1995), no. 1, 1--33. 58E05 (47A53 47N99 57R57 58G10).

\bibitem[SU]{sachs-uhlen} Sachs, J., Uhlenbeck, K., {\em The existence of minimal immersions of 2 spheres},
Ann. Math. 113 (1981), 1--24.

\bibitem[Sa]{sandon} Sandon, S., {\em Contact homology, capacity and non-squeezing in $\R^{2n} \times S^1$
via generating functions}, Ann. Inst. Fourier \textbf{61} (2011), 145--185.

\bibitem[Se]{seidel:pi1} Seidel, P., {\em $\pi_1$ of symplectic automorphism groups and invertibles in quantum
cohomology rings}, Geom. Funct. Anal. 7 (1997), 1046--1095.

\bibitem[Sey]{seyfad} Seyfaddini, S., {\em $C^0$-limits of Hamiltonian paths
and the Oh-Schwartz spectral invariants}, Int. Math. Res. Notices,
2013, no 21., 4920--4960.

\bibitem[T]{theret} Th\'eret D. Th\'ese de doctorat, Universit\'e Denis Diderot (Paris 7), 1995.

\bibitem[V]{viterbo} Viterbo, C., {\em Symplectic topology as the geometry
of generating functions}, Math. Ann. 292 (1992), no. 4, 685--710.
\end{thebibliography}
\end{document}